\documentclass{amsart}
\usepackage{mathrsfs}
\usepackage{amssymb,latexsym,amsmath,amsthm,enumitem,mathrsfs}
\usepackage[margin=1in]{geometry}
\usepackage{color}
\usepackage{stmaryrd}
\usepackage{mathrsfs}
\usepackage{lineno}
\usepackage{bbm}   
\usepackage{scalerel,stackengine}
\usepackage{ esint }
\usepackage{graphicx}
\usepackage{amssymb}
\newtheorem{theorem}{Theorem}
\newtheorem{lemma}{Lemma}

\newcommand{\R}{\mathbb R}

\newcommand{\Z}{\mathbb Z}
\newcommand{\C}{\mathbb C}
\newcommand{\N}{\mathbb N}

\renewcommand{\S}{\mathcal S}
\newcommand{\w}{\omega}
\newcommand{\W}{\mathbb{W}}
\newcommand{\e}{\varepsilon}
\newcommand{\g}{\gamma}
\newcommand{\p}{\varphi}
\newcommand{\s}{\psi}

\renewcommand{\a}{\alpha}
\renewcommand{\b}{\beta}
\renewcommand{\d}{\delta}

\renewcommand{\P}{\mathbb{P}}

\newcommand{\mc}{\mathcal}
\newcommand{\mb}{\mathbb}

\def\set4{\mathcal I}
\def\tup14{(1,2,3,4)}

\newcommand\vwidehat[1]{\arraycolsep=0pt\relax%
\begin{array}{c}
\stretchto{
  \scaleto{
    \scalerel*[\widthof{\ensuremath{#1}}]{\kern-.5pt\bigwedge\kern-.5pt}
    {\rule[-\textheight/2]{1ex}{\textheight}} 
  }{\textheight} %
}{0.5ex}\\           
#1\\                 
\rule{-1ex}{0ex}
\end{array}
}
\newtheorem*{comm*}{Comment}

\newtheorem{definition}{Definition}
\newtheorem*{lemma*}{Lemma}
\newtheorem*{theorem*}{Theorem}

\newcommand{\pp}{\mathbin{\!/\mkern-5mu/\!}}

\usepackage{ bbold }
\usepackage{bbm}   
\usepackage{scalerel,stackengine}
\stackMath
\newcommand\widecheck[1]{%
\savestack{\tmpbox}{\stretchto{%
  \scaleto{%
    \scalerel*[\widthof{\ensuremath{#1}}]{\kern-.6pt\bigwedge\kern-.6pt}%
    {\rule[-\textheight/2]{1ex}{\textheight}}
  }{\textheight}%
}{0.5ex}}%
\stackon[1pt]{#1}{\scalebox{-1}{\tmpbox}}%
}
\usepackage{ esint }

\usepackage{enumitem}

\newcommand{\supp}{\mathrm{supp}}

\newtheorem{prop}[theorem]{Proposition}

\begin{document}

\author{Larry Guth}
\address{Department of Mathematics\\
Massachusetts Institute of Technology\\
Cambridge, MA 02142-4307, USA}
\email{lguth@math.mit.edu}

\author{Dominique Maldague}
\address{Department of Mathematics\\
Massachusetts Institute of Technology\\
Cambridge, MA 02142-4307, USA}
\email{dmal@mit.edu}

\makeatletter
\@namedef{subjclassname@2020}{\textup{2020} Mathematics Subject Classification}
\makeatother

\date{}

\title{Amplitude dependent wave envelope estimates for the cone in $\R^3$ }
\maketitle 
\begin{abstract} For functions $f$ with Fourier transform supported in the truncated cone, we bound superlevel sets $\{x\in\R^3:|f(x)|>\a\}$ using an $\a$-dependent version of the wave envelope estimate (Theorem 1.3) from \cite{locsmooth}. Our estimates imply both sharp square function and decoupling inequalities for the cone. We also obtain sharp small cap decoupling for the cone, where small caps $\g$ subdivide canonical $1\times R^{-1/2}\times R^{-1}$ planks into $R^{-\b_2}\times R^{-\b_1}\times R^{-1}$ sub-planks, for $\b_1\in[\frac{1}{2},1]$ and $\b_2\in[0,1]$. 
\end{abstract}

\section{Introduction}

We use the high/low method to derive superlevel set estimates for functions supported in the cone in $\R^3$. By high/low method, we mean an argument with cases depending on whether functions are high or low-frequency dominated. This type of argument was used in \cite{wellspaced} to prove incidence estimates for certain configurations of tubes, in \cite{gmw} to prove $L^6$ decoupling for the parabola with a $(\log R)$-power bound, and \cite{lvlsets} to prove sharp level set estimates for small caps of the parabola. \cite{locsmooth} used a sophisticated high/low method to prove a more detailed version of the square function estimate for the cone, which we will call a wave envelope estimate. 

The statement of the wave envelope estimate (Theorem 1.3) from \cite{locsmooth} was one of the main novelties of that paper. This main estimate implies the square function estimate for the cone (in $2+1$ dimensions), but it is formulated in a different way. It was inspired by decoupling theory, which gets a lot of leverage from induction on scales. The square function estimate as stated does not interact well with induction on scales. The main inequality of \cite{locsmooth} is a stronger statement, which does interact well with induction on scales. In spite of the inspiration coming from decoupling theory, the main inequality of \cite{locsmooth} is independent from decoupling for the cone. Neither inequality implies the other one. Our Theorem \ref{mainC} is an amplitude-dependent version of the wave envelope estimate from \cite{locsmooth}. It implies both the square function estimate for the cone and decoupling for the cone. And it implies a new result, called small cap decoupling for the cone.

Consider the truncated cone 
$\Gamma=\{\xi\in\R^3:\xi_1^2+\xi_2^2=\xi_3^2,\quad\frac{1}{2}\le|\xi_3|\le 1\}$. 
For a large parameter $R\gg1$, let $\mc{N}_{R^{-1}}(\Gamma)$ denote the $R^{-1}$-neighborhood of $\Gamma$. Decompose $\mc{N}_{R^{-1}}(\Gamma)$ into a collection ${\bf{S}}_{R^{-1/2}}$ of cone planks $\theta$ which have approximate dimensions $1\times R^{-1/2}\times R^{-1}$. For Schwartz functions $f:\R^3\to\C$ with Fourier transform supported in $\mc{N}_{R^{-1}}(\Gamma)$, define Fourier projections $f_\theta:\R^3\to\C$ by
\[ f_\theta(x)=\int_\theta\widehat{f}(\xi)e^{2\pi i x\cdot\xi}d\xi. \]
In \cite{locsmooth}, a detailed Fourier decomposition of $\sum_\theta|f_\theta|^2$ leads to considering partial-sum and localized versions of the square function which we now describe. For each dyadic $s\in[R^{-1/2},1]$, let ${\bf{S}}_s$ be a partition of $\mc{N}_{s^2}(\Gamma)$ into canonical cone planks $\tau$ of dimension $1\times s\times s^2$. For each $\theta\in{\bf{S}}_{R^{-1/2}}$, there is a dual rectangular box $\theta^*$ centered at the origin with dimensions $1\times R^{1/2}\times R$. Translates of $\theta^*$ are called wave packets. For each $\tau\in{\bf{S}}_s$, there is also a dual rectangular box $\tau^*$ of dimension $1\times s^{-1}\times s^{-2}$. Define $U_{\tau,R}$ to be a scaled version of $\tau^*$ with dimensions $s^2R\times sR\times R$, which roughly contains $\theta^*$ for each $\theta\subset\tau$. Let $U\|U_{\tau,R}$ be a tiling of $\R^3$ by translates of $U_{\tau,R}$ and call each $U$ a wave envelope. The partial and localized square function corresponding to the wave envelope $U$ is 
\[ S_{U}f(x)=\Big( \sum_{\theta\subset\tau}|f_\theta|^2(x)W_U(x)\Big)^{1/2} \]
where $W_{U}$ is an $L^\infty$-normalized weight function localized to $U$. The precise definitions of ${\bf{S}}_s$, $U_{\tau,R}$, and $W_{U}$ are contained in \textsection\ref{mainCpf}. With this notation, we now recall the wave envelope estimate from \cite{locsmooth}. 

\begin{theorem}[Theorem 1.3 from \cite{locsmooth}]\label{locsmooththm}
For any Schwartz function $f:\R^3\to\C$ with $\supp\widehat{f}\subset\mc{N}_{R^{-1}}(\Gamma)$, 
\begin{equation}\label{mainCeqnlocsmooth} \int|f|^4 \le C_\e R^\e \sum_{\substack{R^{-1/2}<s<1\\ s\quad\text{dyadic}}}\sum_{\tau\in{\bf{S}}_s}\sum_{U\|U_{\tau,R}}|U|^{-1}\|S_Uf\|_2^4 . \end{equation}
\end{theorem}
\noindent Note that in \cite{locsmooth}, the definition of $S_Uf$ uses a sharp cutoff $\chi_U$ in place of $W_U$, but this is a superficial difference (see the discussion following \eqref{weightdef}). Our amplitude-dependent wave envelope estimate is

\begin{theorem}\label{mainC} For any Schwartz function $f:\R^3\to\C$ with Fourier transform supported in $\mc{N}_{R^{-1}}(\Gamma)$ and any $\a>0$,
\begin{equation}\label{mainCeqnintro} \alpha^4 |\{x\in\R^3:|f(x)|>\a\}| \le C_\e R^\e \sum_{\substack{R^{-1/2}<s<1\\ s\quad\text{dyadic}}}\sum_{\tau\in{\bf{S}}_s}\sum_{U\in\mc{G}_\tau(\a)}|U|^{-1}\|S_Uf\|_2^4  \end{equation}
where $\mc{G}_\tau(\a)=\{U\|U_{\tau,R}:\,C_\e R^\e |U|^{-1}\|S_Uf\|_2^2\ge \frac{\a^2}{(\#\tau)^2} \}$ and $\#\tau=\#\{\tau\in{\bf{S}}_s:f_{\tau}\not\equiv 0\}$. 
\end{theorem}

Comparing Theorem \ref{mainC} with Theorem \ref{locsmooththm}, we see that Theorem \ref{locsmooththm} replaces the left hand side of \eqref{mainCeqnintro} by $\|f\|_{L^4(\R^3)}^4$ and sums over all $U\|U_{\tau,R}$ on the right hand side instead of the restricted set $\mc{G}_\tau(\a)$. Each term on the right hand side in Theorem \ref{mainC} is necessary, up to the $C_\e R^\e$ factor, which follows by the same argument as in \cite{locsmooth}. Theorem \ref{mainC} may be viewed as a strengthening of Theorem \ref{locsmooththm} since by a dyadic pigeonholing argument, there is some $\a>0$, depending on $f$, so that $\|f\|_{L^4(\R^3)}^4\lesssim (\log R)\a^4|\{x:|f(x)|>\a\}|$. 

To start digesting the amplitude-dependent wave envelope estimate of Theorem \ref{mainC}, consider the example $f=\sum_\theta f_\theta$ where each $|f_\theta|\sim \chi_{B_R}$ for a fixed $R$-ball $B_R$. Use Theorem \ref{mainC} to bound the high set 
\[ H=\{x:|f(x)|>\frac{1}{100}R^{1/2}\}, \]
noting that $\|f\|_\infty\lesssim \#\theta\sim R^{1/2}$. We essentially only have the $s=R^{-1/2}$ term on the right hand side of \eqref{mainCeqnintro}. This is because the defining property of $\mc{G}_\tau(\frac{1}{100}R^{1/2})$ is $C_\e R^\e|U|^{-1}\|S_Uf\|_2^2\ge  (\frac{1}{100}R^{1/2})^2\cdot \frac{1}{(\#\tau)^2}$, but for our example, we also have the upper bound
\[ |U|^{-1}\|S_Uf\|_2^2=|U|^{-1}\int\sum_{\theta\subset\tau}|f_\theta|^2W_U \lesssim |U|^{-1}\int(\#\theta\subset\tau)W_U\lesssim sR^{1/2}\lesssim\frac{R^{1/2}}{\#\tau}.  \]
Thus, Theorem \ref{mainC} gives the bound 
\begin{equation}\label{highset} (R^{1/2})^4|H|\le C_\e R^\e \sum_{\theta\in{\bf{S}}_{R^{-1/2}}}\sum_{U\in\mc{G}_\theta(\frac{1}{100}R^{1/2})}|U|^{-1}\Big(\int |f_\theta|^2 W_U\Big)^2\lesssim C_\e R^\e(R^{1/2})|B_R|.  \end{equation}
This is essentially sharp, which follows from letting $f$ be an exponential sum localized to $B_R$ with frequencies $\{(\frac{1}{2}-\frac{n^2}{2R},\frac{n}{R^{1/2}},\frac{1}{2}+\frac{n^2}{2R}):1\le n\le R^{1/2}\}$. The $L^6$ decoupling inequality for the cone gives the same sharp bound for $|H|$. If we use Theorem \ref{locsmooththm}, which is not amplitude-dependent, to bound $|H|$, then the $s=1$ term dominates the right hand side, giving the weaker bound $(R^{1/2})^4|H|\le C_\e R^\e (R^{1/2})^2|B_R|$.

Our primary application of the amplitude-dependent Theorem \ref{mainC} is to new small cap decoupling results for the cone. The only previous result of this type that we are aware of is the special case Theorem 3.6 from \cite{smallcap}. The cone planks $\theta\in{\bf{S}}_{R^{-1/2}}$ we defined above are called \emph{canonical} because they are maximal subsets of $\mc{N}_{R^{-1}}(\Gamma)$ which are comparable, up to some uniform absolute constants, to convex sets. Small cap decoupling involves further subdivision of canonical $1\times R^{-1/2}\times R^{-1}$ planks $\theta$ into $R^{-\b_2}\times R^{-\b_1}\times R^{-1}$ small caps $\g$ (with respect to the same frame as $\theta$), where $\b_1\in[\frac{1}{2},1]$ and $\b_2\in[0,1]$. Recall that $(\ell^2,L^p)$ decoupling into canonical $\theta$ follows from the Pramanik-Seeger argument (recorded by Bourgain-Demeter in \cite{BD}). The idea is that (a piece of) the cone $\Gamma$ may be viewed as a subset of a certain neighborhood of a cylinder over the parabola. After applying cylindrical decoupling over the parabola to decouple to $f_{\tau}$, where $\tau$ are relatively coarse cone planks, a Lorentz rescaling  transforms $f_{\tau}$ into a function with Fourier support on all of $\Gamma$ again and we may iterate the process until we reach the $f_\theta$. This argument no longer works to prove small cap decoupling for the cone from small cap decoupling for the parabola since the Lorentz rescaling changes the dimensions of the small caps, creating an inefficiency in successively applying cylindrical small cap decoupling inequalities. In \textsection\ref{int}, we present the idea of how to obtain small cap decoupling from the refined bounds from Theorem \ref{mainC}. The $(\ell^p,L^p)$ small cap theorem for $\Gamma$ we obtain is the following. 
\begin{theorem}\label{smallcapthm} Let $\b_1\in[\frac{1}{2},1]$ and $\b_2\in[0,1]$. For $p\ge 2$, 
\[ \int_{\R^3}|f|^p\le C_\e R^\e(R^{(\b_1+\b_2)(\frac{p}{2}-1)}+R^{(\b_1+\b_2)(p-2)-1}+R^{(\b_1+\b_2-\frac{1}{2})(p-2)})\sum_\g\|f_\g\|_{L^p(\R^3)}^p\]
for any Schwartz function $f:\R^3\to\C$ with Fourier transform supported in $\mc{N}_{R^{-1}}(\Gamma)$. 
\end{theorem}
The $p=4$ case of this theorem follows from the non-amplitude dependent wave packet estimate of Guth, Wang, and Zhang using analogous arguments as in Theorem 3.6 from \cite{smallcap}. The powers of $R$ in the upper bound come from considering three natural sharp examples for the ratio $\|f\|_{L^p(B_R)}^p/(\sum_\g\|f_\g\|_p^p)$. The first is the square root cancellation example, where $|f_\g|\sim \chi_{B_R}$ for all $\g$ and $f=\sum_\g e_\g f_\g$ where $e_\g$ are $\pm1$ signs chosen (using Khintchine's inequality) so that $\|f\|_{L^p(B_R)}^p\sim R^{(\b_1+\b_2) p/2}R^3$. Then
\[ \|f\|_p^p/(\sum_\g\|f_\g\|_p^p)\gtrsim (R^{(\b_1+\b_2) p/2}R^3)/(R^{\b_1+\b_2)}R^3)\sim R^{(\b_1+\b_2)(\frac{p}{2}-1)}. \]
The second example is the constructive interference example. Let $f_\g=R^{\b_1+\b_2+1}\widecheck{\eta}_\g$ where $\eta_\g$ is a smooth bump function approximating $\chi_\g$. Since $|f|=|\sum_\g f_\g|$ is approximately constant on unit balls and $|f(0)|\sim R^{\b_1+\b_2}$, we have
\[\|f\|_p^p/(\sum_\g\|f_\g\|_p^p)\gtrsim (R^{(\b_1+\b_2) p})/(R^{\b_1+\b_2}R^{\b_1+\b_2+1})\sim R^{(\b_1+\b_2)(p-2)-1}. \]
There is one more example which may dominate the ratio: The block example is $f=R^{\b_1+\b_2+1}\sum_{\g\subset \theta}\widecheck{\eta}_\g$ where $\theta$ is a canonical $1\times R^{-1/2}\times R^{-1}$ block. Since $f=f_\theta$ and $|f_\theta|$ is approximately constant on dual $\sim 1\times R^{1/2}\times R$ blocks $\theta^*$, we have
\[ 
    \|f\|_p^p/(\sum_\g\|f_\g\|_p^p)\gtrsim (R^{(\b_1+\b_2-\frac{1}{2})p}R^{\frac{3}{2}})/(R^{(\b_1+\b_2-\frac{1}{2})}R^{\b_1+\b_2+1})=R^{(\b_1+\b_2-\frac{1}{2})(p-2)}. \]
The square root cancellation, constructive interference, and block examples demonstrate that Theorem \ref{smallcapthm} is sharp, up to the $C_\e R^\e$ factor. The full proof of Theorem \ref{smallcapthm} as a relatively straightforward corollary of Theorem \ref{mainC} is contained in \textsection\ref{conesmallcap}. 

The statement of Theorem \ref{mainC} is one of the main novelties of the paper.  Given the statement, the proof follows analogous steps to the proof of Theorem 1.3 from \cite{locsmooth}. There are two main differences in the proof of Theorem \ref{mainC}. The first is that our Kakeya-type step must now depend on an amplitude parameter. The second is that the ``base case" for the induction-on-scales argument no longer uses the $L^4$ square function estimate for the parabola, which is known by an elementary C\`{o}rdoba-type argument. Instead, we must prove a version of Theorem \ref{mainC} for the parabola, which we explain now. 

Define the truncated parabola $\P^1=\{(t,t^2):|t|\le 1\}$. Let $\sqcup \theta=\mc{N}_{R^{-1}}(\P^1)$ be a partition of the $R^{-1}$-neighborhood of $\P^1$ into canonical $R^{-1/2}\times R^{-1}$ blocks $\theta$. Use the notation $\ell(\theta)=R^{-1/2}$ to index over canonical  $\theta$. The dual set $\theta^*$ of a canonical block $\theta$ of the parabola is a $R^{1/2}\times R$ block centered at the origin. For each dyadic parameter $W$, $R^{1/2}\le W\le R$ and each canonical $\sim (W/R)\times (W/R)^2$ block $\tau$, define the wave envelope $U_{\tau,R}$ to be the convex hull of $\cup_{\theta\subset\tau}\theta^*$, which has dimensions $\sim W\times R$, and let $U\|U_{\tau,R}$ be a tiling of $\R^2$ by translates of $U_{\tau,R}$. As before, let $S_{U}f$ denote the partial, localized square function $S_{U}f(x)^2=\sum_{\theta\subset\tau}|f_\theta|^2(x)W_U(x)$ where $W_{U}$ is an $L^\infty$-normalized weight function adapted to $U$. 
\begin{theorem} \label{mainP} For any $\e>0$, there exists $C_\e>0$ such that the following holds for any sufficiently large $R$. Let $f:\R^2\to\C$ be a Schwartz function with Fourier transform supported in $\mc{N}_{R^{-1}}(\P^1)$. Then for any $\a>0$,
\[ \a^4|U_\a|\le C_\e R^\e\sum_{\substack{R^{1/2}\le W\le R\\W\,\text{dyadic}}}\sum_{\substack{ \ell(\tau)=\frac{W}{R}}}\sum_{U\in\mc{G}_\tau(\a)}|U|^{-1}\|S_Uf\|_2^4 \]
in which $\mc{G}_\tau(\a)=\{U\|U_{\tau,R}:\,C_\e R^\e|U|^{-1}\|S_Uf\|_2^2\ge \frac{\a^2}{(\#\tau)^2}\}$ and $\#\tau$ means $\#\{\tau:\ell(\tau)=s,\quad f_\tau\not\equiv 0\}$.  
\end{theorem}
Theorem \ref{mainP} implies the sharp superlevel set estimates recorded in Theorem 3 of \cite{lvlsets} as well as sharp $(\ell^2,L^6)$ decoupling \cite{BD}. We also prove the following general $(\ell^q,L^p)$ estimates for $\P^1$, which may be compared with Corollary 5 from \cite{lvlsets}. For $\b\in[\frac{1}{2},1]$, let $\sqcup \g=\mc{N}_{R^{-1}}(\P^1)$ be a partition of $\mc{N}_{R^{-1}}(\P^1)$ by $\sim R^{-\b}\times R^{-1}$ small caps $\g$.

\begin{theorem}\label{smallcapthmB}
Let $\b\in[\frac{1}{2},1]$. For any $p,q$ satisfying $\frac{3}{p}+\frac{1}{q}\le 1$, 
\begin{equation}\label{smallcapLp} \int_{B_R}|\sum_\g f_\g|^p\le C_\e R^\e\Big[1+R^{\b(\frac{1}{2}-\frac{1}{q})p}+R^{\b(p-\frac{p}{q}-1)-1}\Big]\big(\sum_\g\|f_\g\|_{L^p(\R^2)}^q\big)^{\frac{p}{q}} .\end{equation}
\end{theorem}
The powers of $R$ come from parabola analogs of the square root cancellation and constructive interference examples described above for the cone. The restriction $\frac{3}{p}+\frac{1}{q}\le 1$ prevents the block example from dominating, which we do not investigate in this paper.

The paper is organized as follows. In \textsection\ref{int} and \textsection\ref{pf}, we present some intuition behind deriving small cap estimates from the refined square function estimates and the proof of the refined square function estimates, respectively. Then in \textsection\ref{tools}, we develop multi-scale high/low frequency tools and prove Theorem \ref{mainP}. The proof of Theorem \ref{smallcapthmB} is contained in \textsection\ref{smallcapP}. We prove Theorem \ref{mainC} in \textsection\ref{mainCpf}-\ref{lorsec}, which follows the outline of \cite{locsmooth}, and prove Theorem \ref{smallcapthm} in \textsection\ref{conesmallcap}. The appendix describes how to adapt the argument to prove Theorem \ref{mainC} for general cones in $\R^3$. 

For $a,b>0$, the notation $a\lesssim b$ means that $a\le Cb$ where $C>0$ is a universal constant whose definition varies from line to line, but which only depends on fixed parameters of the problem. Also, $a\sim b$ means $C^{-1}b\le a\le Cb$ for a universal constant $C$.

LG is supported by a Simons Investigator grant. DM is supported by the National Science Foundation under Award No. 2103249. DM would like to thank Ciprian Demeter for sharing ideas about small cap decoupling for the cone, which helped formulate Theorem \ref{smallcapthm}.

\subsection{Obtaining small cap decoupling from refined square function estimates \label{int}}

We demonstrate the ideas for how to obtain small cap decoupling, Theorem \ref{smallcapthm}, from the amplitude-dependent wave envelope estimate, Theorem \ref{mainC}, in two special cases. 

After a series of pigeonholing steps similar to the process described in Section 5 of \cite{gmw}, proving Theorem \ref{smallcapthm} reduces to demonstrating that 
\begin{equation}\label{red} \a^p|U_\a|\le C_\e R^\e [(R^{(\b_1+\b_2)(\frac{p}{2}-1)}+R^{(\b_1+\b_2)(p-2)-1}+R^{(\b_1+\b_2-\frac{1}{2})(p-2)})]\sum_\g\|f_\g\|_2^2  \end{equation}
for functions $f$ satisfying the extra assumption that $\|f_\g\|_\infty\lesssim 1$ for all $\g$. The cases $2\le p\le 4$ and $6\le p< \infty$ involve interpolation with trivial $L^2$ and $L^\infty$ estimates respectively, so we focus on the range $4\le p\le 6$. 

Begin with the case where the $s=R^{-1/2}$ term dominates the right hand side of \eqref{mainCeqn}:
\[ \a^4|U_\a| \le C_\e R^\e \sum_{\theta\in{\bf{S}}_{R^{-1/2}}} \sum_{U\in\mc{G}_\theta(\a)}|U|^{-1}\Big(\int |f_\theta|^2 W_U\Big)^2 \]
for $\mc{G}_{\theta}(\a)=\{U\|\theta^*: \,C_\e R^\e|U|^{-1}\int|f_\theta|^2W_U\ge \a^2 R^{-1} \}$. The first step is to use the defining property of $\mc{G}_\theta(\a)$ to get an $L^p$-type estimate:
\[ \a^4|U_\a| \le C_\e R^\e  \sum_{\theta\in{\bf{S}}_{R^{-1/2}}} \sum_{U\in\mc{G}_\theta(\a)}\Big(\frac{R}{\a^2}C_\e R^\e |U|^{-1}\int|f_\theta|^2W_U\Big)^{\frac{p}{2}-2}|U|^{-1}\Big(\int |f_\theta|^2 W_U\Big)^2. \]
It then suffices to check that
\begin{align}\label{suff2} R^{\frac{p}{2}-2} \sum_{\theta\in{\bf{S}}_{R^{-1/2}}} \sum_{U\in\mc{G}_\theta(\a)}|U|\Big( |U|^{-1}\int|f_\theta|^2W_U\Big)^{\frac{p}{2}}\le  C_\e R^\e R^{(\b_1+\b_2)(p-2)-1}\sum_\g\|f_\g\|_2^2. \end{align} 
By H\"{o}lder's inequality, 
\[ \Big( |U|^{-1}\int|f_\theta|^2W_U\Big)^{\frac{p}{2}}\lesssim |U|^{-1}\int|f_\theta|^{p}W_U. \]
We need a way to bound integrals with $f_\theta$ in terms of expressions involving $f_\g$. We do this using flat decoupling, which is equivalent to the following argument. Using the assumption $\|f_\g\|_\infty\lesssim 1$, remove $L^\infty$ factors of $f_\theta=\sum_{\g\subset\theta}f_\g$: 
\[ |U|^{-1}\int|f_\theta|^{p}W_U\lesssim (\#\g\subset\theta)^{p-2}|U|^{-1}\int|f_\theta|^{2}W_U\lesssim (R^{\b_1+\b_2-\frac{1}{2}})^{p-2}|U|^{-1}\int|f_\theta|^{2}W_U. \] 
The summary of the previous few steps is that the left hand side of \eqref{suff2} is bounded by 
\[  R^{\frac{p}{2}-2}\sum_{\theta\in{\bf{S}}_{R^{-1/2}}} \sum_{U\in\mc{G}_\theta(\a)}R^{(\b_1+\b_2-\frac{1}{2})(p-2)}\int|f_\theta|^2W_U \lesssim \sum_{\theta\in{\bf{S}}_{R^{-1/2}}} R^{(\b_1+\b_2)(p-2)-1}\int_{\R^3}|f_\theta|^2. \]
By Plancherel's theorem, $\sum_{\theta\in{\bf{S}}_{R^{-1/2}}}\|f_\theta\|_2^2\lesssim\sum_\g\|f_\g\|_2^2$, which finishes the verification of \eqref{suff2}. 

The second case we consider is the other extreme of Theorem \ref{mainC}, which is when the $s=1$ term dominates. Then there is essentially one $\tau\in{\bf{S}}_1$, the wave envelope $U_{\tau,R}$ is a ball of radius $R$, and 
\[ \a^4|U_\a|\le C_\e R^\e \sum_{B_R\in\mc{G}(\a)}|B_R|^{-1}\big(\int\sum_{\theta\in{\bf{S}}_{R^{-1/2}}}|f_\theta|^2W_{B_R}\Big)^2 \]
where $\mc{G}(\a)=\{B_R:\,C_\e R^\e |B_R|^{-1}\int\sum_\theta|f_\theta|^2W_{B_R}\ge \a^2\}$. The first step is again to write an $L^p$-expression: 
\[ \a^4|U_\a|\le C_\e R^\e \sum_{B_R\in\mc{G}(\a)}\Big(\frac{C_\e R^\e}{\a^2}|B_R|^{-1}\int \sum_{\theta\in{\bf{S}}_{R^{-1/2}}}|f_\theta|^2W_{B_R} \Big)^{\frac{p}{2}-2}|B_R|^{-1}\big(\int\sum_{\theta\in{\bf{S}}_{R^{-1/2}}}|f_\theta|^2W_{B_R}\Big)^2 \]
and note that it suffices to verify that
\begin{align}\label{suff3}  
 \sum_{B_R\in\mc{G}(\a)}|B_R|\Big(|B_R|^{-1}\int \sum_{\theta\in{\bf{S}}_{R^{-1/2}}}|f_\theta|^2W_{B_R} \Big)^{\frac{p}{2}} \le C_\e R^\e R^{(\b_1+\b_2)(\frac{p}{2}-1)}\sum_\g\|f_\g\|_2^2. 
\end{align}
We will use an additional property we may assume about the weight functions $W_{B_R}$, which is that $\widehat{W}_{B_R}$ is supported in a ball of radius $R^{-1}$ centered at the origin. A local $L^2$-orthogonality argument shows that for each $B_R$ and $\theta\in{\bf{S}}_{R^{-1/2}}$, 
\[ \int|f_\theta|^2W_{B_R}\lesssim \sum_{\g\subset\theta}\int|f_\g|^2W_{B_R}. \]
By H\"{o}lder's inequality and the property $\|f_\g\|_\infty\lesssim 1$, we have
\[ \Big(|B_R|^{-1}\int\sum_\g|f_\g|^2W_{B_R}\Big)^{\frac{p}{2}} \lesssim |B_R|^{-1}\int(\sum_\g|f_\g|^2)^{\frac{p}{2}}W_{B_R} \lesssim \#\g^{\frac{p}{2}-1}|B_R|^{-1}\int\sum_\g|f_\g|^2W_{B_R} . \]
Finally, noting the bound $\#\g^{\frac{p}{2}-1}\lesssim R^{(\b_1+\b_2)(\frac{p}{2}-1)}$ finishes our verification of \eqref{suff3}.

\subsection{Proof strategy \label{pf}}

The main new idea we introduce to prove Theorem \ref{mainC} is already present in the simpler set-up of Theorem \ref{mainP}, which we will focus on here. 
First we write a non-$\a$-dependent version of Theorem \ref{mainP}: for Schwartz $f:\R^3\to\C$ with Fourier transform supported in $\mc{N}_{R^{-1}}(\P^1)$, 
\begin{equation}\label{nonalphadep} 
\int|f|^4\lesssim \sum_{\substack{R^{\frac{1}{2}}\le W\le R\\ W\,\text{dyadic}}}\sum_{\ell(\tau)=\frac{W}{R}}\sum_{U\|U_{\tau,R}}|U|^{-1}\Big(\int\sum_{\theta\subset\tau}|f_\theta|^2W_U\Big)^2. 
\end{equation}
We will show how this inequality follows from the C\'{o}rdoba $L^4$ square function estimate for the parabola combined with a high-low frequency decomposition. In particular, the $L^4$ square function estimate is 
\begin{equation}\label{sqfn} \int|f|^4\lesssim \int|\sum_\theta|f_\theta|^2|^2. \end{equation}
Using Plancherel's theorem, we study the integral on the right hand side in frequency space:
\[ \int|\sum_\theta|f_\theta|^2|^2=\int|\sum_\theta\widehat{|f_\theta|^2}|^2. \]
Each $\widehat{|f_\theta|^2}=\widehat{f_\theta}*\widehat{\overline{f_\theta}}$ is supported in the $R^{-1/2}\times R^{-1}$ rectangle $\theta-\theta$. The $\{\theta-\theta\}$ form a set of  $\sim R^{-1/2}\times R^{-1}$ rectangles centered at the origin with $\gtrsim R^{-1/2}$-separated orientations. On the outermost part of this set ($|\xi|\gtrsim R^{-1/2}$) the $\theta-\theta$ are finitely overlapping. At the innermost part ($|\xi|\le R^{-1}$), all of the $\theta-\theta$ overlap. Decompose frequency space into dyadic annuli: 
\begin{equation}\label{Fdecomp} \sum_{\substack{R^{\frac{1}{2}}\le W\le R\\ W\,\text{dyadic}}}\int_{|\xi|\sim W^{-1}}|\sum_\theta\widehat{|f_\theta|^2}|^2+\int_{|\xi|\le R^{-1}} |\sum_\theta\widehat{|f_\theta|^2}|^2. \end{equation}
On the annulus $|\xi|\sim W^{-1}$, 
the functions $\{\sum_{\theta\subset\tau}|f_\theta|^2\}_{\ell(\tau)=\frac{W}{R}}$ are finitely overlapping. It follows that
\[ \int_{|\xi|\sim W^{-1}}|\sum_\theta\widehat{|f_\theta|^2}|^2\lesssim \sum_{\ell(\tau)=\frac{W}{R}}\int_{|\xi|\lesssim W^{-1}}|\sum_{\theta\subset\tau}\widehat{|f_\theta|^2}|^2. \]
The sets $\{\theta-\theta\}_{\theta\subset\tau}$ intersected with the $W^{-1}$ ball are contained in a single $\sim W^{-1}\times R^{-1}$ block we will call $U_{\tau,R}^*$. For a smooth bump function $\eta_{U_{\tau,R}^*}\approx \chi_{U_{\tau,R}^*}$, we then have
\begin{align*} 
\int_{|\xi|\lesssim W^{-1}}|\sum_{\theta\subset\tau}\widehat{|f_\theta|^2}|^2&\le \int|\sum_{\theta\subset\tau}\widehat{|f_\theta|^2}\eta_{U_{\tau,R}^*}|^2=  \int|\sum_{\theta\subset\tau}{|f_\theta|^2}*\widecheck{\eta}_{U_{\tau,R}^*}|^2\lesssim \sum_{U\|U_{\tau,R}}|U|^{-1}\Big(\int \sum_{\theta\subset\tau}|f_\theta|^2W_U\Big)^2
\end{align*}
where $W_{U}$ are an $L^\infty$-normalized weight functions whose order of decay away from $U$ we may prescribe. 
In summary, the Fourier localization of $\sum_\theta|f_\theta|^2$ to the set $|\xi|\gtrsim W^{-1}$ gives approximate $L^2$ orthogonality of the $\sum_{\theta\subset\tau}|f_\theta|^2$ and the Fourier localization of $\sum_{\theta\subset\tau}|f_\theta|^2$ to the ball $|\xi|\lesssim W^{-1}$ gives the average of $\sum_{\theta\subset\tau}|f_\theta|^2$ over the wave envelope $U_{\tau,R}$. 

Theorem \ref{mainP} is the $\a$-dependent version of \eqref{nonalphadep}:
\[ \a^4|U_\a|\lessapprox \sum_{\substack{R^{\frac{1}{2}}\le W\le R\\ W\,\text{dyadic}}}\sum_{\tau\in{\bf{S}}_{\frac{W}{R}}(\P^1)}\sum_{U\in\mc{G}_{\tau}}|U|^{-1}\|S_Uf\|_{2}^4 \]
where $U_\a=\{x:|f(x)|>\a\}$ and $\mc{G}_\tau=\mc{G}_\tau(\a)=\{U\|U_{\tau,R}:|U|\|S_Uf\|_2^2\gtrapprox \frac{\a^2}{(\#\tau)^2}$. We prove this via a stopping-time algorithm tuned to a pruning process for wave envelopes. We now describe a simplified and heuristic version of the algorithm.  

\vspace{2mm}
\noindent\fbox{Step 1:} For each $\theta\in{\bf{S}}_{R^{-1/2}}(\P^1)$, define the dual set $\theta^*=\{x:|x\cdot y|\le 1\quad\forall y\in\theta-\theta\}$ and write $U_{\theta,R}=\theta^*$.
Let $\sum_{U\|U_{\theta,R}}\s_{U}$ be a partition of unity subordinate to the tiling $U\|U_{\theta,R}$ and satisfying the additional property that $\widehat{\s_{U}}$ is supported in a $U_{\theta,R}^*=\theta-\theta$. Organize the wave packets $U\|U_{\theta,R}$ into a good and a bad set defined by
\[ \mc{G}_{\theta}=\{U\|U_{\tau,R}:\|\s_{U}f_\theta\|_\infty\gtrsim \frac{\a}{\#\theta}\}\qquad\text{and}\qquad \mc{B}_\theta=\mc{G}_\theta^c. \]
On $U_\a$, we show that $\a\sim f(x)\sim \sum_{\theta}\sum_{U\in\mc{G}_\theta}\s_Uf_\theta$.
Note that for $U\in\mc{G}_\theta$, $\s_U(x)|f_\theta|(x)\approx S_{U}f(x)$ and the property $\|\s_Uf\|_\infty\gtrsim\frac{\a}{\#\theta}$ is like the condition $|U|^{-1}\|S_{U}f\|_2^2\gtrsim\frac{\a^2}{(\#\theta)^2}$ for $x\in U$. Write $F=\sum_\theta\sum_{U\in\mc{G}_\theta}\s_Uf_\theta$ and $F_\theta=\sum_{U\in\mc{G}_\theta}\s_Uf_\theta$, noting that the Fourier supports of each $f_\theta$ and $F_\theta$ are essentially the same. Notice that if we carry out the argument verifying \eqref{nonalphadep} for $F$, then we would be done. The strategy is to identify a case where we can guarantee that with $F$ in place of $f$, the term corresponding to $W=R^{1/2}$ dominates in \eqref{Fdecomp}. 

Specifically, let $\tau\in{\bf{S}}_{R^\e R^{-1/2}}(\P^1)$. As above, let $U_{\tau,R}$ be the convex hull of $\cup_{\theta\subset\tau}\theta^*$ and $\sum_{U\|U_{\tau,R}}\s_U$ be a partition of unity subordinate to tiling $U\|U_{\tau,R}$. Organize the wave envelopes $U$ into good and bad by
\[ \mc{G}_\tau=\{U\|U_{\tau,R}:|U|^{-1}\|S_UF\|_{2}^2\ge \frac{\a^2}{(\#\tau)^2}\}\quad\text{and}\quad \mc{B}_\tau=\mc{G}_\tau^c  . \]
Further refine $F$ by letting
\[ F^{\mc{G}}=\sum_{\tau\in{\bf{S}}_{R^\e R^{-1/2}}(\P^1)}\sum_{U\in\mc{G}_{\tau}}\s_U\sum_{\theta\subset\tau}F_\theta \qquad\text{and}\qquad F^{\mc{B}}=\sum_{\tau\in{\bf{S}}_{R^\e R^{-1/2}}(\P^1)}\sum_{U\in\mc{G}_{\tau}}\s_U\sum_{\theta\subset\tau}F_\theta .\]
Suppose that 
\begin{equation}\label{sup1}
    |U_\a|\lesssim |\{x:\a\sim |F^{\mc{B}}(x)|\}|. 
\end{equation}
Note that on the set on the right hand side, which we will denote $U_\a^{\mc{B}}$, 
\[ \sum_{\tau\in{\bf{S}}_{R^\e R^{-1/2}}(\P^1)}|F^{\mc{B}}_\tau|^2(x)\lesssim |\sum_{\tau\in{\bf{S}}_{R^\e R^{-1/2}}(\P^1)}|F^{\mc{B}}_\tau|^2*\widecheck{\eta}_{\ge R^{-\e} R^{-1/2}}(x)|.  \]
Indeed, if $\sum_{\tau}|F_\tau^{\mc{B}}|^2$ did not satisfy this high-dominance condition, then by Cauchy-Schwarz,
\[ \a^2\lesssim \#\tau|\sum_\tau|F_{\tau}^{\mc{B}}|^2*\widecheck{\eta}_{\le R^{-\e}R^{-1/2}}(x)| .\]
Using Plancherel's theorem and the delicate definition of $F^{\mc{B}}$, for each $\tau\in{\bf{S}}_{R^\e R^{-1/2}}$, we have roughly 
\begin{align*} 
||F_{\tau}^{\mc{B}}|^2*\widecheck{\eta}_{\le R^{-\e}R^{-1/2}}(x)|\lesssim \sum_{\theta\subset\tau}\sum_{U\in\mc{B}_\tau}\s_U(x)|F_\theta
|^2*|\widecheck{\eta}_{\le R^{-\e}R^{-1/2}}|(x)\lesssim \sum_{U\in\mc{B}_\tau}\s_U(x)|U|^{-1}\int\sum_{\theta\subset\tau}|F_\theta|^2W_U. 
\end{align*} 
Combine this with the Cauchy-Schwarz step above and use the definition of the bad set $\mc{B}_{\tau}$, for $x\in U\in{\mc{B}_{\tau}}$ to get the inequality
\[ \a^2\lesssim \#\tau\sum_\tau\frac{\a^2}{(\#\tau)^2},\]
which we may arrange to be a contradiction by choosing appropriate implicit constants. We call $\sum_\tau|F_{\tau}^{\mc{B}}|^2$ weakly high-dominated on $U_\a^{\mc{B}}$. Since the $\tau$ are only a factor of $R^\e$ more coarse than the $\theta$, this is essentially the property we wished to verify, 
which concludes this case. 

In the case that \eqref{sup1} does not hold, we have 
\[ |U_\a|\lesssim |\{x:\a\sim|F^{\mc{G}}(x)|\}, \]
which initiates the next step of the algorithm. 
\vspace{2mm}
\newline\noindent\fbox{Step $k$:} The input at this stage is the inequality 
\[ |U_\a|\lesssim |\{x:\a\sim |F_{k-1}^{\mc{G}}(x)| \}|\]
where $F_{k-1}^{\mc{G}}$ is a version of $f$ which contains only good wave envelopes (in an analogous sense as above) at scales $W$, $R^{1/2}\le W\le R^{(k-1)\e}R^{1/2}$. Step $k$ of the algorithm defines 
\[ F_{k-1}^{\mc{G}}=F_k^{\mc{G}}+F_k^{\mc{B}}\]
where $F_k^{\mc{G}}$ only has wave packets at scale $R^{k\e}R^{1/2}$ which are good. In the case that 
\[ |U_\a|\lesssim |\{x:\a\sim |F_k^{\mc{B}}(x)|\}|, \]
then we can show that $F_k^{\mc{B}}$ satisfies a weak high-dominance property which suffices to verify Theorem \ref{mainP}. In the case that $|U_\a|\lesssim |\{x:\a\sim |F_k^{\mc{G}}(x)|\}|$, we initiate the next step of the algorithm. 

\noindent\fbox{Final step:} The final step of the algorithm only happens if 
\[ |U_\a|\lesssim |\{x:\a\sim|F_1^{\mc{G}}(x)|\} \]
where $F_1^{\mc{G}}$ has good wave envelopes for all scales. In particular, we have good wave envelopes at scale $W=R$, which basically means that 
\[ \a^2\lesssim |B_R|^{-1}\int\sum_\theta|f_\theta|^2W_{B_R} \]
for each $B_R$ that $F_1^{\mc{G}}$ is essentially supported. Conclude directly that 
\[ \a^4|U_\a|\lessapprox \sum_{B_R\in\mc{G}(\a)}|B_R|\Big(\int\sum_\theta|f_\theta|^2W_{B_R}\Big)^2 \]
where $\mc{G}(\a)=\{B_R:\,C_\e R^\e |B_R|^{-1}\int\sum_\theta|f_\theta|^2W_{B_R}\ge \a^2\}$. 

\section{Proof of Theorem \ref{mainP} \label{tools}}

The proof of Theorem \ref{mainP} follows from a wave-envelope pruning process, which leads to a decomposition of $f=\sum_m f_m$ where there are $\lessapprox 1$ many terms in the sum. Each $f_m$ has the property that a corresponding square function at a certain scale is weakly high-frequency dominated. Then we analyze each $f_m$ and the relevant square functions using the geometry of the frequency support. In this section, we introduce notation for different scale neighborhoods of $\P^1$, a pruning process for wave envelopes at various scales, some high/low lemmas which are used to analyze the high/low frequency parts of square functions, and a version of a bilinear restriction theorem for $\P^1$.   

Begin by fixing some notation, as above. Let $\e>0$. Let $R\in 4^\N$ be a parameter we will take to be sufficiently large, depending on $\e>0$. Fix a ball $B_R\subset\R^2$ of radius $R$. The parameter $\a>0$ describes the superlevel set 
\[ U_\a=\{x\in B_R:|f(x)|\ge \a\}.\]
We analyze scales $W_k\in 2^\N$ satisfying $R^{1-(k-1)\e}\le W_k<2R^{1-(k-1)\e}$ and $R^{1/2}\le W_k\le R$. Let $N$ distinguish the index for which $W_N$ is closest to $R^{1/2}$. Since $R^{1/2}$ and $W_N$ differ at most by a factor of $R^\e$, we will ignore the distinction between $W_N$ and $R^{1/2}$ in the rest of the argument. Note that $N\sim \e^{-1}$.  

Define the following collections, each of which partitions a neighborhood of $\P$ into approximate rectangles. 
\begin{enumerate}
    \item $\{\theta\}$ is a partition of $\mc{N}_{R^{-1}}(\P^1)$ by approximate $R^{-1/2}\times R^{-1}$ rectangles.  
    \item $\{\tau_k\}$ is a partition of $\mc{N}_{W_k^2/R^2}(\P^1)$ by approximate $(W_k/R)\times (W_k/R)^2$ rectangles. Assume the additional property that $\theta\cap\tau_k=\emptyset$ or $\theta\subset\tau_k$. 
\end{enumerate}
Since $R\in 4^\N$ and $W_k\in 2^\N$ with $W_k\le R$, each $R/W_k\in 2^\N$. One way to precisely define the $\{\tau_k\}$ (including $\{\tau_N\}=\{\theta\}$) is by
\begin{equation}\label{blocks} 
\bigsqcup_{|l|\le RW_k^{-1}-2} \{(\xi_1,\xi_2)\in\mc{N}_{W_k^2/R^2}(\P^1):lW_kR^{-1}\le \xi_1<(l+1)W_kR^{-1} \}  \end{equation}
and the two end pieces
\[  \{(\xi_1,\xi_2)\in\mc{N}_{W_k^2/R^2}(\P^1):\xi_1<-1+W_kR^{-1}\} \sqcup  \{(\xi_1,\xi_2)\in\mc{N}_{W_k^2/R^2}(\P^1):1-W_kR^{-1}\le \xi_1\}. \]

\subsection{A pruning process for the wave envelopes \label{pruning}} 
We will define wave envelopes corresponding to each $\tau_k$, for each scale $W_k$. 


Suppose that $\tau_k$ is the $l$th-piece $\tau_k=\{(\xi_1,\xi_2)\in \mc{N}_{W_k^2/R^2}:lW_kR^{-1}\le \xi_1<(l+1)W_kR^{-1}\}$. Then define 
\[ \tau_k^*=\{x_1(1,2lW_kR^{-1})+x_2(-2lW_kR^{-1},1):|x_1|\le R/W_k,\quad |x_2|\le R^2/W_k^2 \}\]
and note that
\[ \frac{1}{20}\tau_k^*\subset\{x\in\R^2:|x\cdot(\xi-(lW_k/R,l^2W_k^2/R^2))|\le 1\quad\forall\xi\in\tau_k\}\subset 20\tau_k^*. \]
For each $\tau_k$, fix a wave envelope $U_{\tau_k,R}$ which is a rectangle of dimension $W_k\times R$ centered at the origin and with long side parallel to the long side of $\tau_k^*$. The motivation for this definition of wave envelope is that the convex hull of $\cup_{\theta\subset\tau_k}\theta^*$ is comparable to $U_{\tau_k,R}$.

For $\tau_k$, tile the plane by translates $U$ of $U_{\tau_k,R}$, denoted by $U\| U_{\tau_k,R}$. We will define an associated partition of unity $\sum_{U\|U_{\tau_k,R}}\s_{U}$. 
First let $\p(\xi)$ be a bump function supported in $[-\frac{1}{4},\frac{1}{4}]^2$. For each $m\in\Z^2$, let 
\[ \s_m(x)=c\int_{[-\frac{1}{2},\frac{1}{2}]^2}|\widecheck{\p}|^2(x-y-m)dy, \]
where $c$ is chosen so that $\sum_{m\in\Z^2}\s_m(x)=c\int_{\R^2}|\widecheck{\p}|^2=1$. Since $|\widecheck{\p}|$ is a rapidly decaying function, for any $n\in\N$, there exists $C_n>0$ such that
\[ \s_m(x)\le c\int_{[0,1]^2}\frac{C_n}{(1+|x-y-m|^2)^n}dy \le \frac{\tilde{C}_n}{(1+|x-m|^2)^n}. \]
Define the partition of unity $\s_{U}$ associated to $U\| U_{\tau_k,R}$ to be $\s_{U}(x)=\s_m\circ A_{\tau_k}$, where $A_{\tau_k}$ is a linear transformation taking $U_{\tau_k,R}$ to $[-\frac{1}{2},\frac{1}{2}]^2$ and $A_{\tau_k}(U)=m+[-\frac{1}{2},\frac{1}{2}]^2$. The important properties of $\s_{U}$ are (1) rapid decay off of $U$ and (2) Fourier support contained in a dual set $U^*=\{\xi:|x\cdot\xi|\le 1\,\,\forall x\in U\}$ which is a rectangle at the origin of dimension $W_k^{-1}\times R^{-1}$. 

We need one more auxiliary function related to each $U$ before we can define the pruning process. For each $U_{\tau_k,R}$, let 
$W_{U_{\tau_k,R}}$ be the weight function adapted to $U_{\tau_k,R}$ defined by 
\begin{equation}\label{weightdef} W_{U_{\tau_k,R}}(x)=W\circ R_{\tau_k}(x)\end{equation}
where $W$ is defined by 
\[ W(x,y)=\frac{1}{(1+|x|^2)^{100}(1+|y|^2)^{100}},\qquad \|W\|_1=1,\]
and $R_{\tau_k}:\R^2\to\R^2$ is the linear transform which first rotates $2U_{\tau_k,R}$ to $[-W_k,W_k]\times[-R,R]$ and then scales $[-W_k,W_k]\times [-R,R]$ to the cube $[-1,1]^2$. For each $U\|U_{\tau_k,R}$, let $W_{U}=W_{U_{\tau_k,R}}(x-c_{U})$ where $c_U$ is the center of $U$. 

Note that with this definition of $W_U$, we can easily verify that Theorem \ref{locsmooththm} implies the wave envelope estimate (Theorem 1.3) of \cite{locsmooth} with sharp cutoffs $\chi_U$. Indeed, for each $\tau$ and $U\|U_{\tau,R}$, we have
\begin{align*}
\sum_{U\|U_{\tau,R}}\|S_Uf\|_2^4&=\sum_{U\|U_{\tau,R}}\Big(\int\sum_{\theta\subset\tau}|f_\theta|^2W_U\Big)^2    =\sum_{U\|U_{\tau,R}}\Big(\sum_{U'\|U_{\tau,R}}\int_{U'}\sum_{\theta\subset\tau}|f_\theta|^2W_U\Big)^2  \\
&\le \sum_{U\|U_{\tau,R}}\Big(\sum_{U'\|U_{\tau,R}}\|W_U\|_{L^\infty(U')}\int_{U'}\sum_{\theta\subset\tau}|f_\theta|^2\Big)^2 \\
(\text{Cauchy-Schwarz) }&\lesssim \sum_{U\|U_{\tau,R}}\Big(\sum_{U'\|U_{\tau,R}}\|W_U\|_{L^\infty(U')}\Big)\sum_{U''\|U_{\tau,R}}\|W_U\|_{L^\infty(U'')}\Big(\int_{U''}\sum_{\theta\subset\tau}|f_\theta|^2\Big)^2\\
&\lesssim \sum_{U\|U_{\tau,R}}\sum_{U''\|U_{\tau,R}}\|W_U\|_{L^\infty(U'')}\Big(\int_{U''}\sum_{\theta\subset\tau}|f_\theta|^2\Big)^2\lesssim \sum_{U''\|U_{\tau,R}}\Big(\int_{U''}\sum_{\theta\subset\tau}|f_\theta|^2\Big)^2. 
\end{align*}

Now we describe the pruning process, which sorts between important and unimportant wave envelopes corresponding to $f$ on the set $U_\a$. Define the average integral notation $\fint_Ug$ to mean
\[ \fint_Ug= |U|^{-1}\int g W_U. \]
Let $C_0>0$ be a universal constant and $m_0>0$ be a constant depending on $\e$ which are both fixed in \textsection\ref{broadnarrow}. 
First identify a ``good" subset of $U\|U_{\theta,R}$ as follows. 
Set 
\[ U\in \mc{G}_\theta\quad\text{if}\quad (\log R)C_0^{2m_0}R^{2C_0\e}\fint_U|f_\theta|^2\ge \frac{\a^2}{(\#\theta)^2}  .\]
When we write $\#\theta$ (or $\#\tau_k$), we always mean the size of the set $\{\theta :f_{\theta}\not\equiv 0\}$ (or the size of $\{\tau_k:f_{\tau_k}\not\equiv 0\})$. 

\begin{definition}[Pruning with respect to $\tau_k$]\label{taukprune} For each $\theta$, define 
\[ f_{W_N,\theta}:=\sum_{\substack{U\|U_{\theta,R}\\U\in\mc{G}_\theta}}\s_Uf_\theta \qquad \text{and}\qquad f_{W_N}:=\sum_\theta f_{W_N,\theta}. \]
For each $k<N$ and each $\tau_k$, let 
\begin{align*} 
\mc{G}_{\tau_k}:=\{U&\|U_{\tau_k,R}:\log RC_0^{2m_0}R^{2C_0\e}\fint_U\sum_{\theta\subset\tau_k}|f_{W_{k+1},\theta}|^2\ge \frac{\a^2}{(\#\tau_k)^2}\},\\
f_{W_k,\theta}&:=\sum_{\substack{U\in\mc{G}_{\tau_k}}}\s_U f_{W_{k+1},\theta} \quad\text{where}\quad\theta\subset\tau_k, \\
\text{and}&\qquad f_{W_k}:=\sum_{\theta}f_{W_k,\theta}.
\end{align*}
Finally, write $f_{W_{m}}-f_{W_{m-1}}:=f_{W_m}^{\mc{B}}$.
\end{definition}

We record some of the important properties of this pruning process in the following lemma. 
\begin{lemma}\label{pruneprop} 
\begin{enumerate}
    \item\label{decomp}  $f_{W_N}=f_{W_1}+\sum_{m=2}^{N}f_{W_m}^{\mc{B}}$
    \item \label{ineqprune}For all $R^{\frac{1}{2}}\le W_{k+1}< W_k<R$, $|f_{W_{k+1}}|\le |f_{W_k}|\le |f|$ and for each $\theta$, $|f_{W_{k+1},\theta}|\le |f_{W_k,\theta}|\le |f_\theta|$. Also, $|f_{W_{k+1},\theta}^{\mc{B}}|\le |f_{W_k,\theta}^{\mc{B}}|\le |f_\theta|$.
    \item\label{supp} $ \text{supp} \widehat{f_{W_k,\theta}}\subset (N-k+2)\theta$ for all $\theta$. 
\end{enumerate} 
\end{lemma}

\begin{proof} To see property \eqref{decomp}, write 
\begin{align*} 
f_{W_N}    &=f_{W_{N-1}}+f_{W_N}-f_{W_{N-1}}=\cdots=f_{W_1}+\sum_{m=2}^{N}(f_{W_{m}}-f_{W_{m-1}})= f_{W_{1}}+\sum_{m=2}^{N}f_{W_{m-1}}^{\mc{B}}.
\end{align*} 
Property \eqref{ineqprune} follows directly from the definition of the pruning process. The claims about $f_{W_k}$ and $f_{W_k,\theta}^{\mc{B}}$ follow from the same reasoning. Property \eqref{supp} follows from the noting that the Fourier support of each $\s_{U}$, for $U\| U_{\tau_k,R}$, is contained in $\theta$. 
\end{proof}

\subsection{Lemmas about square functions} 
We will prove various results about square functions (squared) of the form $\sum_\tau|f_{W_k,\tau}|^2$ for various sizes of $\tau$. First we define some weight functions, which are useful when we invoke the locally constant property. By locally constant property, we mean generally that if a function $f$ has Fourier transform supported in a convex set $A$, then for a bump function $\p_A\equiv 1$ on $A$, $f=f*\widecheck{\p_A}$. Since $|\widecheck{\p_A}|$ is an $L^1$-normalized function which is positive on a set dual to $A$, $|f|*|\widecheck{\p_A}|$ is an averaged version of $|f|$ over a dual set $A^*$. After defining some weight functions, we record some of the specific locally constant properties we need in Lemma \ref{locconst}

For each $W_k$, let $w_{R/W_k}$ be the weight function
\begin{equation}\label{ballweight2} w_{R/W_k}(x)=\frac{c'}{(1+|x|^2/(R/W_k)^2)^{10}},\qquad\|w_{R/W_k}\|_1=1. \end{equation}

\begin{lemma}[Locally constant property]\label{locconst} For any $\theta$, 
\[ |f_\theta|^2(x)\lesssim |f_\theta|^2*|\widecheck{\rho}_\theta|\]
where $\rho_{\theta}$ be a smooth bump function equal to $1$ on $\theta$ and supported in $2\theta$. 
For each $k$ and $m$ and for any $x$, 
\[\sum_{\tau_m}|f_{W_k,\tau_m}|^2(x)\lesssim_\e \sum_{\tau_m}|f_{W_k,\tau_m}|^2*w_{R/W_m}(x) \]
in which $f_{W_k,\tau_m}=\sum_{\theta\subset\tau_m}f_{W_k,\theta}$.
\end{lemma}
The first property also holds for any $f_{W_k,\theta}$ in place of $f_\theta$ since each of these functions has essentially the same Fourier support. 
\begin{proof}[Proof of Lemma \ref{locconst}] Using Fourier inversion and H\"{o}lder's inequality, 
\[ |f_{W_k,\theta}(y)|^2=|f_{W_k,\theta}*\widecheck{\rho_{\theta}}(y)|^2\le\|\widecheck{\rho_{\theta}}\|_1 |f_{W_k,\theta}|^2*|\widecheck{\rho_{\theta}}|(y). \]
Since $\rho_{\theta}$ may be taken to be a standard bump function adapted to the unit ball precomposed with an affine transformation, $\|\widecheck{\rho_{\theta}}\|_1$ is a constant. The function $\widecheck{\rho_{\theta}}$ decays rapidly off of $\theta^*$, so
\[ |f_{W_k,\theta}(x)|^2\lesssim |\theta^*|^{-1}|f_{W_k,\theta}|^2*W_{\theta^*}(x), \]
which proves the first claim. For the second claim, note that by \eqref{supp} of Lemma \ref{pruneprop}, the Fourier support of $f_{W_k,\tau_m}$ is contained in $\cup_{\theta\subset\tau_m}(N-k+2)\theta$, which is contained in $C_\e B_{\tau_m}$ where $B_{\tau_m}$ is a ball of radius $2W_m/R$ which contains $\tau_m$. Let $\rho_{\tau_m}$ be a bump function equal to $1$ on $C_\e B_{\tau_m}$ and supported in $2C_\e B_{\tau_m}$. Again using Fourier inversion, we have
\[ |f_{W_k,\tau_m}(x)|^2=|f_{W_k,\tau_m}*\widecheck{\rho}_{\tau_m}(x)|^2\lesssim |f_{W_k,\tau_m}|^2*|\widecheck{\rho}_{\tau_m}|(x). \]
It remains to note that for each $\tau_m$,  $|\widecheck{\rho}_{\tau_m}|\lesssim_\e w_{R/W_m}$. 

\end{proof}

We use the locally constant property to show that it suffices to replace $f$ by $f_{W_N}$ in $U_\a$. 
\begin{lemma}\label{essent} We have the pointwise inequality
\[ |f(x)-f_{W_N}(x)|\lesssim \frac{1}{(\log R)^{1/2}C_0^{m_0}R^{C_0\e}}\a.\]
\end{lemma}
\begin{proof} Consider the difference of $f$ and $f_{W_N}$:
\begin{align*}
|f(x)-f_{W_N}(x)|&\le  \sum_\theta \sum_{U\notin\mc{G}_\theta}\s_U(x)|f_\theta(x)|. 
\end{align*}
Since $\s_Uf_{\theta}$ has Fourier transform supported in $2\theta$, we may use the same bump function $\rho_{2\theta}$ from the proof of Lemma \ref{locconst} to bound
\[ \sum_\theta \sum_{U\notin\mc{G}_\theta}\s_U(x)|f_\theta(x)|\le\sum_\theta \sum_{U\notin\mc{G}_\theta}\int \s_U(y)|f_\theta(y)||\widecheck{\rho}_{2\theta}|(x-y)dy . \]
By Cauchy-Schwarz, since $\|\widecheck{\rho}_{2\theta}\|_1\sim 1$, the right hand side is bounded by 
\[ \sum_\theta \sum_{U\notin\mc{G}_\theta}\Big(\int \s_U^2(y)|f_\theta(y)|^2|\widecheck{\rho}_{2\theta}|(x-y)dy\Big)^{\frac{1}{2}}. \]
Use the fact that for any $x$, $\sum_{U\not\in\mc{G}_\theta}\|\s_U(\cdot)|\widecheck{\rho}_{2\theta}(x-\cdot)\|_{\infty}\lesssim |U|^{-1}$ and that $\s_U(y)\le W_U(y)$ 
to get the upper bound
\[ \sum_\theta\max_{U\not\in\mc{G}_\theta}\Big(|U|^{-1}\int|f_\theta|^2W_U\Big)^{\frac{1}{2}} . \]
Finally, use the defining property of $\mc{G}_\theta$:
\[ \sum_\theta\max_{U\not\in\mc{G}_\theta}\Big(|U|^{-1}\int|f_\theta|^2W_U\Big)^{\frac{1}{2}}\le \sum_\theta\frac{1}{(\log R)^{1/2}C_0^{m_0}R^{C_0\e}}\frac{\a}{\#\theta}. \]

\end{proof}

Next we analyze high and low frequency portions of square functions. Define radial Littlewood-Paley-type functions that we will eventually use to localize to certain frequencies.  
\begin{definition}[Auxiliary functions] \label{aux}Let $\p(x):\R^2\to[0,\infty)$ be a radial, smooth bump function satisfying $\p(\xi)=1$ on $B_1$ and $\supp\p\subset B_2$. Then for $\xi\in B_2$, 
\begin{align*}
    1= \p(R^{\e(J+1)}\xi)+\sum_{j=-1}^J[\p(R^{\e j}\xi)-\p(R^{\e(j+1)}\xi)]
\end{align*}
where $J$ is defined by $R^{\e J }\le  R< R^{\e(J+1)}$. For each $r$, let
\[ \eta_{\le  r}=\p(r^{-1}\xi), \quad\eta_{> r}=\p(\xi)-\p(r^{-1}\xi),\quad \eta_{\sim r}(\xi) =\p(r^{-1}\xi)-\p(2 r^{-1}\xi) . \]
\end{definition}


In the following lemma, for each $\tau_m\subset\tau_{m-1}$, let 
\[ f_{W_m,\tau_m}^{\mc{B}}=\sum_{U\not\in\mc{G}_{\tau_{m-1}}}\s_U\sum_{\theta\subset\tau_{m}}f_{W_m,\theta}. \]
\begin{lemma}[Low lemma]\label{low} Fix $m$, $2\le m<N$. For any $2\le k<m$ and $r\le W_k/R$, 
\[ \sum_{\tau_m}|f_{W_m,\tau_{m-1}}^{\mc{B}}|^2*\widecheck{\eta}_{\le r}(x)=\sum_{\tau_k}\sum_{\substack{\tau_k'\\\tau_k\sim\tau_k'}} (f_{W_m,\tau_k}^{\mc{B}}\overline{f}_{W_m,\tau_k'}^{\mc{B}})*\widecheck{\eta}_{\le r}(x)\]
where $r_j=R^{-j\e}$ and $\tau_k\sim\tau_k'$ means that $\emph{dist}(\tau_k,\tau_k')\lesssim_\e W_k/R$. 
\end{lemma}
\begin{proof} By Plancherel's theorem, since $\overline{\widecheck{\eta}}_{\le r}={\widecheck{\eta}}_{\le r}$, we have for each ${\tau_m}$
\begin{align}
|f_{W_m,\tau_{m-1}}^{\mc{B}}|^2*\widecheck{\eta}_{\le r}(x)&= \int_{\R^2}|f_{W_m,\tau_{m-1}}^{\mc{B}}|^2(x-y)\widecheck{\eta}_{\le r}(y)dy \nonumber \\
&=  \int_{\R^2}\widehat{f}_{W_m,\tau_{m-1}}^{\mc{B}}*\widehat{\overline{f_{W_m,\tau_{m-1}}^{\mc{B}}}}(\xi)e^{-2\pi i x\cdot\xi}\eta_{\le r}(\xi)d\xi \nonumber \\
&=  \sum_{{\tau_k},{\tau_k}'\subset\tau_{m-1}}\int_{\R^2}e^{-2\pi i x\cdot\xi}\widehat{f}_{W_m,\tau_k}^{\mc{B}}*\widehat{\overline{f^{\mc{B}}_{W_m,\tau_k'}}}(\xi)\eta_{\le r}(\xi)d\xi .\label{dis}
\end{align}
The support of $\widehat{\overline{f_{W_m,\tau_k'}^{\mc{B}}}}(\xi)=\int e^{-2\pi ix\cdot\xi}\overline{f}_{W_m,\tau_k'}^{\mc{B}}(x)dx=\overline{\widehat{f}_{W_m,\tau_k'}^{\mc{B}}}(-\xi)$ is contained in $-(N-m+2)\tau_k'$ (where we used \eqref{supp} of Lemma \ref{pruneprop}). This means that the support of $\widehat{f}_{W_m,\tau_k}^{\mc{B}}*\widehat{\overline{f^{\mc{B}}_{W_m,\tau_k'}}}(\xi)$ is contained in $(N-m+2)(\tau_k-\tau_k')$. Since the support of $\eta_{\le r}(\xi)$ is contained in the ball of radius $2r$ (which is $\le 2W_k/R$), for each $\tau_k\subset\tau_{m-1}$, there are only $\lesssim_\e 1$ many adjacent $\tau_k'\subset\tau_{m-1}$ so that the integral in \eqref{dis} is nonzero. 
\end{proof}

Define the notation 
\[ f_{W_m,\theta}^{\mc{B}}=\sum_{U\not\in\mc{G}_{\tau_{m-1}}}\s_Uf_{W_m,\theta}\]
where $\theta\subset\tau_{m-1}$. 
\begin{lemma}[High lemma I]\label{high} For any $W_m,W_k$, 
\[ \int |\sum_{\theta}|f_{W_m,\theta}^{\mc{B}}|^2*\widecheck{\eta}_{\ge W_k^{-1}}|^2\lesssim_\e \sum_{\tau_k}\int |\sum_{\theta\subset\tau_k}|f_{W_m,\theta}^{\mc{B}}|^2*\widecheck{\eta}_{\ge W_k^{-1}}|^2. \]
\end{lemma}

\begin{proof} By Plancherel's theorem, we have
\begin{align*}
\int |\sum_{\theta}|f_{W_m,\theta}^{\mc{B}}|^2*\widecheck{\eta}_{\ge W_k^{-1}}|^2 &=    \int_{W_k^{-1}<|\xi|}|\sum_\theta \widehat{|f_{W_m,\theta}^{\mc{B}}|^2}(\xi)\eta_{\ge W_k^{-1}}(\xi)|^2 \\
    &=  \int_{W_k^{-1}<|\xi|}|\sum_{\tau_k}\sum_{\theta\subset\tau_k} \widehat{|f_{W_m,\theta}^{\mc{B}}|^2}(\xi)\eta_{\ge W_k^{-1}}(\xi)|^2   .
\end{align*}
The maximum number of nonzero summands corresponding to each $\tau_k$ in the final integrand is $\lesssim_\e 1$. This is because the maximum overlap of the sets $[(N+k-2) (\theta-\theta)]\cap\supp\eta_{\ge W_k^{-1}}$ happens on the circle of radius $W_k^{-1}$ where it is $\lesssim (N-k+2)W_k/R^{1/2}\sim_\e W_k/R^{1/2}$, coming from adjacent $\theta$. Therefore, by Cauchy-Schwarz, 
\begin{align*}
\int_{W_k^{-1}<|\xi|}|\sum_{\tau_k}\sum_{\theta\subset\tau_k} \widehat{|f_{W_m,\theta}^{\mc{B}}|^2}(\xi)\eta_{\ge W_k^{-1}}(\xi)|^2&\lesssim_\e\sum_{\tau_k}\int_{W_k^{-1}<|\xi|}|\sum_{\theta\subset\tau_k} \widehat{|f_{W_m,\theta}^{\mc{B}}|^2}(\xi)\eta_{\ge W_k^{-1}}(\xi)|^2\\
&\sim_\e \sum_{\tau_k}\int|\sum_{\theta\subset\tau_k} |f_{W_m,\theta}^{\mc{B}}|^2*\widecheck{\eta}_{\ge W_k^{-1}}|^2. 
\end{align*}
\end{proof}

\begin{lemma}[High lemma II]\label{high2} For any $W_k,W_m$,
\[ \int|\sum_{\tau_k}|f_{W_m,\tau_k}^{\mc{B}}|^2*\widecheck{\eta}_{\ge W_k/R}|^2\lesssim_\e  \sum_{\tau_k}\int|f_{W_m,\tau_{k}}^{\mc{B}}|^4.\]
\end{lemma}
\begin{proof} This follows by the same reasoning as in the proof of Lemma \ref{high}, except each $|f_{W_m,\tau_k}^{\mc{B}}|^2$ has Fourier support in $(N-k+2)(\tau_k-\tau_k)$, which overlap at most by $\lesssim_\e 1$ outside a ball of a radius $W_k/R$. 
\end{proof}

\begin{lemma}[High-domination of bad parts]\label{highdom} Let $R$ be sufficiently large depending on $\e$. If $m\in\{2,\ldots,N\}$ and $ \a\lesssim \e^{-1}C_0^{m_0}R^{C_0\e}|f_{W_m}^{\mc{B}}(x)|$,
then
\[ \sum_{\tau_{m-1}}|f_{W_m,\tau_{m-1}}^{\mc{B}}|^2*w_{R/W_{m-1}}(x)\le 2|\sum_{\tau_{m-1}}|f_{W_m,\tau_{m-1}}^{\mc{B}}|^2*w_{R/W_{m-1}}*\widecheck{\eta}_{> W_{m-1}^{-1}}(x)|. \]
\end{lemma}
\begin{proof} Recall that $f_{W_m}^{\mc{B}}=f_{W_m}-f_{W_{m-1}}$. Using the definition of the pruning process, we may also write $f_{W_m}^{\mc{B}}=\sum_{\tau_{m-1}}f_{W_m,\tau_{m-1}}^{\mc{B}}$ where  $f_{W_m,\tau_{m-1}}^{\mc{B}}:=\sum_{U\not\in\mc{G}_{\tau_{m-1}}}\s_U\sum_{\theta\subset\tau_{m-1}}f_{W_m,\theta}$. By Cauchy-Schwarz, we have 
\begin{align}
\nonumber \a&\lesssim \e^{-1}C_0^{m_0}R^{C_0\e}(\#\tau_{m-1})^{1/2}(\sum_{\tau_{m-1}}|f_{W_m,\tau_{m-1}}^{\mc{B}}|^2(x))^{1/2}\\
\label{CS} &\lesssim \e^{-1}C_0^{m_0}R^{C_0\e}(\#\tau_{m-1})^{1/2}(\sum_{\tau_{m-1}}|f_{W_m,\tau_{m-1}}^{\mc{B}}|^2*w_{R/W_{m-1}}(x))^{1/2}  \end{align} 
where we used Lemma \ref{locconst} in the last inequality. Now suppose that 
\begin{equation}\label{contra2} \sum_{\tau_{m-1}}|f_{W_m,\tau_{m-1}}^{\mc{B}}|^2*w_{R/W_{m-1}}(x)\le 2\big|\sum_{\tau_{m-1}}|f_{W_m,\tau_{m-1}}^{\mc{B}}|^2*w_{R/W_{m-1}}*\widecheck{\eta}_{\le W_{m-1}^{-1}}(x)\big|. \end{equation} 
Use Lemma \ref{low} with $r=W_{m-1}^{-1}$ (noting that $W_{m-1}^{-1}\le R^{-\frac{1}{2}}$) to get \[\sum_{\tau_{m-1}}|f_{W_m,\tau_{m-1}}^{\mc{B}}|^2*\widecheck{\eta}_{\le W_{m-1}^{-1}}(x)=\sum_{\theta}\sum_{\theta'\sim\theta} (f_{W_m,\theta}^{\mc{B}}\overline{f}_{W_m,\theta}^{\mc{B}})*\widecheck{\eta}_{\le W_{m-1}^{-1}}(x). \]
Then by the proof of Lemma \ref{locconst}, 
\[\sum_\theta |f_{W_m,\theta}^{\mc{B}}|^2*|\widecheck{\eta}_{\le W_{m-1}^{-1}}|(y)\lesssim_\e \sum_\theta |f_{W_m,\theta}^{\mc{B}}|^2*|\widecheck{\rho}_{\theta}|*|\widecheck{\eta}_{\le W_{m-1}^{-1}}|(y)  .\]
For each $\theta\subset\tau_{m-1}$, 
\begin{align*}
|f_{W_m,\theta}^{\mc{B}}|^2*|\widecheck{\rho}_\theta|&*|\widecheck{\eta}_{\le W_{m-1}^{-1}}|(y)\lesssim \int|\sum_{U\not\in\mc{G}_{\tau_{m-1}}} \s_U(y-z)f_{W_m,\theta}|^2(y-z)|\widecheck{\rho}_\theta|*|\widecheck{\eta}_{\le W_{m-1}^{-1}}|(y-z)dz\\
    (\text{Cauchy-Schwarz})\qquad \qquad   &\le \int\sum_{U\not\in\mc{G}_{\tau_{m-1}}} \s_U(z)|f_{W_m,\theta}|^2(z)\sum_{U'\not\in\mc{G}_{\tau_{m-1}}}\s_{U'}(z)|\widecheck{\rho}_\theta|*|\widecheck{\eta}_{\le W_{m-1}^{-1}}|(y-z)dz \\
(\sum_{U'\not\in\mc{G}_{\tau_{m-1}}}\s_{U'}\le 1)\qquad \qquad   &\le \int\sum_{U\not\in\mc{G}_{\tau_{m-1}}} \s_U(z)|f_{W_m,\theta}|^2(z)|\widecheck{\rho}_\theta|*|\widecheck{\eta}_{\le W_{m-1}^{-1}}|(y-z)dz \\
&\le \sum_{U\not\in\mc{G}_{\tau_{m-1}}}\|\s_U^{\frac{1}{2}}(\cdot)|\widecheck{\rho}_\theta|*|\widecheck{\eta}_{\le W_{m-1}^{-1}}|(y-\cdot)\|_{\infty}\int \s_U^{\frac{1}{2}}(z)|f_{W_m,\theta}(z)|^2dz. 
\end{align*}
Sum the above inequality over $\theta\subset\tau_{m-1}$ and use the condition $U\not\in \mc{G}_{\tau_{m-1}}$ to obtain
\begin{align*}
\sum_{\theta\subset\tau_{m-1}}|f_{W_m,\theta}^{\mc{B}}|^2*|\widecheck{\rho}_\theta|*|\widecheck{\eta}_{\le W_{m-1}^{-1}}|(y)    &\le \sum_{U\not\in\mc{G}_{\tau_{m-1}}}\|\s_U^{\frac{1}{2}}(\cdot)|\widecheck{\rho}_\theta|*|\widecheck{\eta}_{\le W_{m-1}^{-1}}|(y-\cdot)\|_{\infty}\frac{1}{(\log R)C_0^{m_0}R^{C_0\e}}\frac{\a^2}{(\#\tau_{m-1})^2}\\
&\lesssim \frac{1}{(\log R)C_0^{m_0}R^{C_0\e}}\frac{\a^2}{(\#\tau_{m-1})^2}. 
\end{align*}
Conclude in this case that \[\big|\sum_{\tau_m}|f_{W_m,\tau_m}^{\mc{B}}|^2*w_{R/W_m}(y)*\widecheck{\eta}_{\le W_{m-1}^{-1}}(y)\big|\lesssim_\e \frac{1}{(\log R)C_0^{m_0}R^{C_0\e}}\frac{\a^2}{\#\tau_{m-1}}. \]
Combined with \eqref{CS} and \eqref{contra2}, this is a contradiction, for $R$ sufficiently large depending on $\e$. Conclude that \eqref{contra2} is false and we have the desired high-frequency dominance. 
\end{proof}

\section{Proof of Theorem \ref{mainP} \label{mainPpf}}

By Lemma \ref{essent} and \eqref{decomp} of Lemma \ref{pruneprop}, 
\begin{equation}\label{sufficient} |U_\a|\le |\{x\in U_\a:|f(x)|\lesssim \e^{-1}|f_{W_1}(x)|\}|+\sum_{m=2}^{N}|\{x\in U_\a:|f(x)|\lesssim\e^{-1} |f_{W_m}^{\mc{B}}(x)|\}| .\end{equation} 
Since there are $\lesssim \e^{-1}$ many sets on the right hand side, it suffices to prove Theorem \ref{mainP} with $U_\a$ replaced by each of the sets on the right hand side above. The following proposition treats the easier case of $m=1$. 

\begin{prop}[$m=1$] \label{m0}
\[ \a^4|\{x\in U_\a:|f(x)|\lesssim \e^{-1}|f_{W_1}(x)|\}|\lesssim (\log R)^2C_0^{4m_0}R^{4C_0\e}\sum_{U\in\mc{G}_{\tau_1}}|U|\left(\fint_U\sum_{\theta}|f_\theta|^2\right)^2\]
where $\tau_1$ is $[-1,1]^2$. 
\end{prop}
\begin{proof} Recall that $U_\a\subset B_R$ for some fixed ball $B_R$. By definition of $f_{W_1}$ and assuming $|U_\a|>0$, there exists some $R$-cube $U\in\mc{G}_{\tau_1}$. Then using the definition of $\mc{G}_{\tau_1}$, we have
\[ \a^4|\{x\in U_\a:|f(x)|\lesssim\e^{-1}|f_{W_1}(x)|\}|\lesssim \Big((\log R)C_0^{2m_0}R^{2C_0\e} |U|^{-1}\int\sum_{\theta}|f_{W_1,\theta}|^2\s_U^{\frac{1}{2}}\Big)^2R^2. \]
By \eqref{ineqprune} of Lemma \ref{pruneprop} and noting that $\s_U^{\frac{1}{2}}\lesssim W_U$, this finishes the proof. 
\end{proof}

To bound the parts of $U_\a$ corresponding to $m=2,\ldots,N$, we first prove a \emph{broad} (or bilinear) version of Theorem \ref{mainP} in \textsection\ref{broad} and then reduce to the broad case in \textsection\ref{broadnarrow}. 

\subsection{The broad part of $U_\a$ \label{broad}}
The notation $\ell(\tau)=s$ means that $\tau$ is an approximate $s\times s^2$ block which is part of a partition of $\mc{N}_{s^2}(\P^1)$. For $m$, $2\le m\le N$ and two non-adjacent blocks $\tau,\tau'$ satisfying $\ell(\tau)=\ell(\tau')=R^{-\e}$, define the $m$th broad version of $U_\a$ to be
\begin{equation}\label{broad2} \text{Br}_\a^m(\tau,\tau')=\{x\in\R^2:\a\lesssim \e^{-1}|f_{W_m,\tau}^{\mc{B}}(x)f_{W_m,\tau'}^{\mc{B}}(x)|^{1/2}
\}. \end{equation}

We will use the following version of a local bilinear restriction theorem, which follows from a standard C\'{o}rdoba argument \cite{cordoba}. This version of the theorem is proved as Theorem 16 of \cite{lvlsets}.  
\begin{theorem}\label{bilrest} Let $S\ge 4$, $\frac{1}{2}\ge D\ge S^{-1/2}$, and $X\subset\R^2$ be any Lebesgue measurable set. Suppose that $\tau$ and $\tau'$ are $D$-separated subsets of $\mc{N}_{S^{-1}}(\P^1)$. Then for a partition $\{\theta_S\}$ of $\mc{N}_{S^{-1}}(\P^1)$ into $\sim S^{-1/2}\times S^{-1}$-blocks, we have
\[\int_{X}|f_{\tau}|^2(x)|f_{\tau'}|^2(x)dx\lesssim D^{-2}
\int_{\mc{N}_{S^{1/2}}(X)}|\sum_{\theta_S}|f_{\theta_S}|^2*w_{S^{1/2}}(x)|^2dx .\]
\end{theorem}

\begin{prop}\label{mainprop}If $m\in\{2,\ldots,N\}$
then
\[ \a^4|\text{Br}_\a^m(\tau,\tau')|\le C_\e R^\e\sum_{R^{1/2}<W_k<W_m}\sum_{\substack{\tau\\\ell(\tau)=W_k/R}}\sum_{U\in\mc{G}_\tau}|U|\left(\fint_U\sum_{\theta\subset\tau}|f_\theta|^2\right)^2. \]
\end{prop}
\begin{proof}
By bilinear restriction, here given by Theorem \ref{bilrest}, 
\begin{align*}
    \int_{\text{Br}_\a^m(\tau,\tau')}|f_{W_m}^{\mc{B}}|^4&\lesssim_\e R^{2\e} \int_{\mc{N}_{R/W_{m-1}}(\text{Br}_\a^m(\tau,\tau'))}|\sum_{\tau_{m-1}}|f_{W_m,\tau_{m-1}}^{\mc{B}}|^2*w_{R/W_{m-1}}|^2 .
\end{align*}
For each $x\in U_\a$, $\sum_{\tau_{m-1}}|f_{W_m,\tau_{m-1}}^{\mc{B}}|^2*w_{R/W_{m-1}}(x)\sim\sum_{\tau_{m-1}}|f_{W_m,\tau_{m-1}}^{\mc{B}}|^2*w_{R/W_{m-1}}(y)$ for any $y\in B_{R/W_{m-1}}(x)$. By Lemma \ref{highdom}, the last integral is bounded by 
\[\int_{\mc{N}_{R/W_{m}}(\text{Br}_\a^m(\tau,\tau'))}|\sum_{\tau_{m-1}}|f_{W_m,\tau_{m-1}}^{\mc{B}}|^2*w_{R/W_{m-1}}*\widecheck{\eta}_{>W_{m-1}^{-1}}|^2 .\]
We are done using properties of the restricted domain, so we now integrate over all of $\R^2$. Let $r$ be some dyadic value between $W_{m-1}^{-1}\le r\lesssim\e^{-1} W_{m-1}/R$ satisfying
\[ \int|\sum_{\tau_{m-1}}|f_{W_m,\tau_{m-1}}^{\mc{B}}|^2*w_{R/W_{m-1}}*\widecheck{\eta}_{>W_{m-1}^{-1}}|^2\le (\log R)^2 \int |\sum_{\tau_{m-1}}|f_{W_m,\tau_{m-1}}^{\mc{B}}|^2*w_{R/W_{m-1}}*\widecheck{\eta}_{>W_{m-1}^{-1}}*\widecheck{\eta}_{\sim r}|^2. \]
Use Young's convolution inequality to simplify the expression
\[ \int|\sum_{\tau_{m-1}}|f_{W_m,\tau_{m-1}}^{\mc{B}}|^2*w_{R/W_{m-1}}*\widecheck{\eta}_{>W_{m-1}^{-1}}*\widecheck{\eta}_{\sim r}|^2\lesssim  \int|\sum_{\tau_{m-1}}|f_{W_m,\tau_{m-1}}^{\mc{B}}|^2*\widecheck{\eta}_{\sim r}|^2. \]
The analysis of the integral on the right hand side splits into two further cases. 

\vspace{1mm}
\noindent\underline{Case 1: $r\le R^{-1/2}$.} 
Use Lemma \ref{low} to arrive at an expression with $\theta$:
\[ \int|\sum_{\tau_{m-1}}|f_{W_m,\tau_{m-1}}^{\mc{B}}|^2*\widecheck{\eta}_{\sim r}|^2=\int|\sum_{\theta}\sum_{\theta'\sim\theta}(f_{W_m,\theta}^{\mc{B}}\overline{f_{W_m,\theta'}^{\mc{B}}})*\widecheck{\eta}_{\sim r}|^2. \]
By Plancherel's theorem, 
\[ \int|\sum_{\theta}\sum_{\theta'\sim\theta}(f_{W_m,\theta}^{\mc{B}}\overline{f_{W_m,\theta'}^{\mc{B}}})*\widecheck{\eta}_{\sim r}|^2= \int|\sum_{\theta}\sum_{\theta'\sim\theta}(\widehat{f}_{W_m,\theta}^{\mc{B}}*\widehat{\overline{f_{W_m,\theta'}^{\mc{B}}}}){\eta}_{\sim r}|^2. \]
Let $k\ge m$ satisfy $R^\e r^{-1}\le W_k^{-1}\le R^{2\e}r^{-1}$. For each $\tau_k$, the intersection of a $2r$ ball with the support of $\sum_{\theta\subset\tau_k}\sum_{\theta'\sim\theta}(\widehat{f}_{W_m,\theta}^{\mc{B}}*\widehat{\overline{f}}_{W_m,\theta'}^{\mc{B}})$ is contained in $C_\e U_{\tau_k,R}^*$ where $U_{\tau_k,R}^*$ is the $W_k^{-1}\times R^{-1}$ rectangle dual to $U_{\tau_k,R}$. The sets $C_\e U_{\tau_k,R}^*$ are $\lesssim_\e R^\e$ overlapping on the support of $\eta_{\sim r}$, so 
\[ \int|\sum_{\theta}\sum_{\theta'\sim\theta}(\widehat{f}_{W_m,\theta}^{\mc{B}}*\widehat{\overline{f^{\mc{B}}_{W_m,\theta'}}}){\eta}_{\sim r}|^2\lesssim_\e R^\e \sum_{\tau_k}\int|\sum_{\theta\subset\tau_k}\sum_{\theta'\sim\theta}(\widehat{f}_{W_m,\theta}^{\mc{B}}*\widehat{\overline{f^{\mc{B}}_{W_m,\theta'}}}){\eta}_{\sim r}|^2.\]
Let $\rho_{\tau_k}$ be a Schwartz function satisfying $|\widecheck{\rho}_{\tau_k}|\lesssim\chi_{U_{\tau_k,R}}$ (where $\chi_{U_{\tau_k,R}}$ is the characteristic function), $\|\widecheck{\rho}_{\tau_k}\|_1\sim 1$, and $\rho_{\tau_k}\gtrsim 1$ on $C_\e U_{\tau_k,R}^*$. Then 
\begin{align*} 
\sum_{\tau_k}\int|\sum_{\theta\subset\tau_k}\sum_{\theta'\sim\theta}(\widehat{f}_{W_m,\theta}^{\mc{B}}*\widehat{\overline{f^{\mc{B}}_{W_m,\theta'}}}){\eta}_{\sim r}|^2&\lesssim\sum_{\tau_k} \int|\sum_{\theta\subset\tau_k}\sum_{\theta'\sim\theta}(\widehat{f}_{W_m,\theta}^{\mc{B}}*\widehat{\overline{f^{\mc{B}}_{W_m,\theta'}}})\rho_{\tau_k}|^2\\
&\lesssim_\e \sum_{\tau_k}\int|\sum_{\theta\subset\tau_k}|{f}_{W_m,\theta}^{\mc{B}}|^2*|\widecheck{\rho}_{\tau_k}||^2. 
\end{align*}
Since $W_k\le W_m$, $|f_{W_m,\theta}^{\mc{B}}|\le |f_{W_k,\theta}|\le |f_{W_{k+1},\theta}|\le |f_\theta|$ (see Lemma \ref{pruneprop}), so 
\begin{align*} 
\sum_{\tau_k}\int|\sum_{\theta\subset\tau_k}|{f}_{W_m,\theta}^{\mc{B}}|^2*|\widecheck{\rho}_{\tau_k}||^2&\le \sum_{\tau_k}\int|\sum_{\theta\subset\tau_k}|f_{W_k,\theta}|^2*|\widecheck{\rho}_{\tau_k}||^2 \\
&=  \sum_{\tau_k}\int|\sum_{\theta\subset\tau_k}\int |\sum_{U\in\mc{G}_{\tau_k}}\s_U(y)f_{W_{k+1},\theta}(y)|^2|\widecheck{\rho}_{\tau_k}|(x-y)dy|^2dx \\
&\lesssim \sum_{\tau_k}\int|\sum_{\theta\subset\tau_k}\int \sum_{U\in\mc{G}_{\tau_k}}\s_U(y)|f_{\theta}|^2(y)|\widecheck{\rho}_{\tau_k}|(x-y)dy|^2dx. 
\end{align*}
Note that for each $y$ and $x\in y+U_{\tau_k,R}$, 
$\s_U(y)
\le \max_{z\in x+U_{\tau_k,R}}|\s_U(z)|=: \tilde{\s}_U(x)$. 
Therefore, for each $\tau_k$, 
\begin{align*}
\int|\sum_{\theta\subset\tau_k}\int \sum_{U\in\mc{G}_{\tau_k}}\s_U(y)|f_{\theta}|^2(y)|\widecheck{\rho}_{\tau_k}|(x-y)dy|^2dx &\le \int|\sum_{\theta\subset\tau_k}\int \sum_{U\in\mc{G}_{\tau_k}}\tilde{\s}_U(x)|f_{\theta}|^2(y)|\widecheck{\rho}_{\tau_k}|(x-y)dy|^2dx\\
(\text{since}\,\,\sum_{U\in\mc{G}_{\tau_k}}\tilde{\s}_U\lesssim 1)\qquad\qquad\qquad&\lesssim \sum_{U\in\mc{G}_{\tau_k}}\int \tilde{\s}_U(x)\big|\int \sum_{\theta\subset\tau_k}|f_{\theta}|^2(y)|\widecheck{\rho}_{\tau_k}|(x-y)dy\big|^2dx \\
(\text{Minkowski's inequality})\qquad\qquad\quad&\lesssim  \sum_{U\in\mc{G}_{\tau_k}}\Big( \int \sum_{\theta\subset\tau_k}|f_{\theta}|^2(y)\big(\int \tilde{\s}_U(x)|\widecheck{\rho}_{\tau_k}|^2(x-y)dx\big)^{\frac{1}{2}}dy\Big)^2 \\
&\lesssim  \sum_{U\in\mc{G}_{\tau_k}}|U|^{-1}\Big( \int \sum_{\theta\subset\tau_k}|f_{\theta}|^2(y)W_U(y)dy\Big)^2 ,
\end{align*}
which finishes this case. 

\vspace{1mm}
\noindent\underline{Case 2: $r> R^{-1/2}$.} Let $W_k\le W_m$ satisfy $W_k/R\le r\le R^\e W_k/R$. By Lemma \ref{low},  
\begin{align*}
\int |\sum_{\tau_{m-1}}|f_{W_m,\tau_{m-1}}^{\mc{B}}|^2*\widecheck{\eta}_{\sim r}|^2&= \int|\sum_{\tau_k}\sum_{\tau_k'\sim\tau_k}(f_{W_m,\tau_k}^{\mc{B}}\overline{f_{W_m,\tau_k}^{\mc{B}}})*\widecheck{\eta}_{\sim r}|^2 .
\end{align*}
Then by an analogue to Lemma \ref{high2},
\[ \int|\sum_{\tau_k}\sum_{\tau_k'\sim\tau_k}(f_{W_m,\tau_k}^{\mc{B}}\overline{f_{W_m,\tau_k}^{\mc{B}}})*\widecheck{\eta}_{\sim r}|^2\lesssim_\e R^{\e} \sum_{\tau_k}\int|f_{W_m,\tau_k}^{\mc{B}}|^4 .\]  
Now apply the $L^4$ reverse square function inequality to get for each $\tau_k$,
\[ \int|f_{W_m,\tau_k}^{\mc{B}}|^4\lesssim_\e \int|\sum_{\theta\subset\tau_k}|f_{W_m,\theta}^{\mc{B}}|^2|^2. \]
\noindent\underline{Case 2a:} Suppose that 
\begin{equation}\label{Case2a}\sum_{\tau_k}\int|\sum_{\theta\subset\tau_k}|f_{W_m,\theta}^{\mc{B}}|^2|^2\lesssim \sum_{\tau_k}\int|\sum_{\theta\subset\tau_k}|f_{W_m,\theta}^{\mc{B}}|^2*\widecheck{\eta}_{\le W_m^{-1}}|^2 . \end{equation}
Then since $W_k\le W_m$ (so $\tau_k\subset\tau_m$), for the same function $\rho_{\tau_m}$ from Case 1, we have
\begin{align*} 
\sum_{\tau_k}\int|\sum_{\theta\subset\tau_k}|f_{W_m,\theta}^{\mc{B}}|^2*\widecheck{\eta}_{\le W_m^{-1}}|^2&\lesssim \sum_{\tau_k}\int|\sum_{\theta\subset\tau_k}|f_{W_m,\theta}^{\mc{B}}|^2*|\widecheck{\rho}_{\tau_m}||^2\\
&\lesssim \sum_{\tau_m}\int|\sum_{\theta\subset\tau_m}|f_{W_m,\theta}^{\mc{B}}|^2*|\widecheck{\rho}_{\tau_m}||^2\\
&\lesssim \sum_{\tau_m}\sum_{U\in\mc{G}_{\tau_m}}|U|\left(\fint_U\sum_{\theta\subset\tau_m}|f_\theta|^2\right)^2
\end{align*} 
where we used the analogous argument as in Case 1 to justify the final line.

\noindent\underline{Case 2b:} If \eqref{Case2a} does not hold, then 
\begin{equation}\label{Case2b}
\sum_{\tau_k}\int|\sum_{\theta\subset\tau_k}|f_{W_m,\theta}^{\mc{B}}|^2|^2\lesssim \sum_{\tau_k}\int|\sum_{\theta\subset\tau_k}|f_{W_m,\theta}^{\mc{B}}|^2*\widecheck{\eta}_{\ge W_m^{-1}}|^2 . \end{equation}
Since for each $\tau_k$, the Fourier transform of $\sum_{\theta\subset\tau_k}|f_{W_m,\theta}^{\mc{B}}|^2$ is contained in a ball of radius $\sim C_\e R^{-1/2}$, there is some dyadic $s$ in the range $W_m^{-1}\le s\lesssim_\e W_N^{-1}=R^{-1/2}$ satisfying 
\[\sum_{\tau_k}\int|\sum_{\theta\subset\tau_k}|f_{W_m,\theta}^{\mc{B}}|^2*\widecheck{\eta}_{\ge W_m^{-1}}|^2\lesssim_\e(\log R)^2 \sum_{\tau_k}\int|\sum_{\theta\subset\tau_k}|f_{W_m,\theta}^{\mc{B}}|^2*\widecheck{\eta}_{\sim s}|^2 . \]
From here, the analysis of each integral on the right hand side follows the argument of Case 1. If $s\le W_k^{-1}$, then for each $\tau_k$, 
\[\int|\sum_{\theta\subset\tau_k}|f_{W_m,\theta}^{\mc{B}}|^2*\widecheck{\eta}_{\sim s}|^2\lesssim \int|\sum_{\theta\subset\tau_k}|f_{W_m,\theta}^{\mc{B}}|^2*|\widecheck{\rho}_{\tau_k}||^2\lesssim\sum_{U\in\mc{G}_{\tau_k}}|U|\Big(\fint_U\sum_{\theta\subset\tau_k}|f_\theta|^2\Big)^2. \]
If $s>W_k^{-1}$, then for $W_k\le W_l\le W_m$ with $W_l$ closest to $s$, we have 
\begin{align*}
    \sum_{\tau_k}\int|\sum_{\theta\subset\tau_k}|f_{W_m,\theta}^{\mc{B}}|^2*\widecheck{\eta}_{\sim s}|^2&\lesssim   \sum_{\tau_k}\int|\sum_{\theta\subset\tau_k}|f_{W_m,\theta}^{\mc{B}}|^2*|\widecheck{\rho}_{\tau_l}||^2 \\
    &\lesssim  \sum_{\tau_l}\int|\sum_{\theta\subset\tau_l}|f_{W_k,\theta}^{\mc{B}}|^2*|\widecheck{\rho}_{\tau_l}||^2\lesssim \sum_{\tau_l}\sum_{U\in\mc{G}_{\tau_l}}|U|\Big(\fint\sum_{\theta\subset\tau_l}|f_\theta|^2\Big)^2, 
\end{align*}
which finishes the proof. 

\end{proof}

\subsection{Bilinear reduction \label{broadnarrow}}

First we will show why it suffices to consider $U_\a$ contained in an $R$-ball $B_R$. Combining this with a broad/narrow analysis, we will show that Propositions \ref{m0} and \ref{mainprop} imply Theorem \ref{mainP}. 

\begin{lemma}[Local version of Theorem \ref{mainP}]\label{local} Suppose that for each $R$-ball $B_R$, 
\[ \a^4|\{x\in B_R:\a\le|f(x)|\}|\lesssim_\e R^\e\sum_{R^{\frac{1}{2}}\le W\le R}\sum_{\ell(\tau)=W/R}\sum_{U\in\mc{G}_\tau}|U|^{-1}\Big(\int\sum_{\theta\subset\tau}|f_\theta|^2W_U\Big)\]
where $\mc{G}_\tau=\{U\|U_{\tau,R}:\frac{\a^2}{(\#\tau)^2}\le C_\e R^\e |U|^{-1}\int\sum_{\theta\subset\tau}|f_\theta|^2W_U\}$. Then Theorem \ref{mainP} is true. 
\end{lemma}
\begin{proof} Fix a finitely overlapping cover of $\R^2$ by $R$ balls $B_R$. Let $\phi_{B_R}\lesssim 1$ be a Schwartz function satisfying $\phi_{B_R}\gtrsim 1$ on $B_R$, $\supp\widehat{\phi}_{B_R}\subset\{\xi:|\xi|\le R^{-1}\}$, and rapid decay off of $B_R$. Each $\phi_{B_R}f_{\theta}$ has Fourier support inside of $2\theta$. We may use the triangle inequality to assume that $|\phi_{B_R}(x)f(x)|\lesssim |\sum_{\theta\in\Theta}\phi_{B_R}(x)f_\theta(x)|$ where the Fourier support of each $\phi_{B_R}f_\theta$ is disjoint for $\theta\in\Theta$. Then apply the hypothesis at scale $\sim R$ to get
\[ \a^4|\{x\in B_R:\a\le|f(x)|\}|\lesssim_\e R^\e\sum_{R^{\frac{1}{2}}\le W\le R}\sum_{\ell(\tau)=W/R}\sum_{U\in\mc{G}_\tau(B_R)}|U|^{-1}\Big(\int\sum_{\theta\subset\tau}|\phi_{B_R}f_\theta|^2W_U\Big)^2\]
where $\mc{G}_{\tau}(B_R)=\{U\|U_{\tau,R}:\frac{\a^2}{(\#\tau)^2}\le C_\e R^\e |U|^{-1}\int\sum_{\theta\subset\tau}|\phi_{B_R}f_\theta|^2W_U\}$. 
Note that each $\mc{G}_{\tau}(B_R)\subset\mc{G}_\tau$. Using the triangle inequality,
\[ \a^4|\{x\in \R^2:\a\le|f(x)|\}|\lesssim_\e R^\e \sum_{B_R}\sum_{R^{\frac{1}{2}}\le W\le R}\sum_{\ell(\tau)=W/R}\sum_{U\in\mc{G}_\tau}|U|^{-1}\Big(\int\sum_{\theta\subset\tau}|\phi_{B_R}f_\theta|^2W_U\Big)^2.\] 
Interchange the order of summation to first sum over $B_R$ and use $\|\cdot\|_{\ell^2}\le\|\cdot\|_{\ell^1}$, leading to 
\[ \sum_{B_R}\Big(\int\sum_{\theta\subset\tau}|\phi_{B_R}f_\theta|^2W_U\Big)^2\le \Big(\sum_{B_R}\int\sum_{\theta\subset\tau}|\phi_{B_R}f_\theta|^2W_U\Big)^2\lesssim \Big(\int\sum_{\theta\subset\tau}|f_\theta|^2W_U\Big)^2, \]
which finishes the proof.
\end{proof}

\begin{proof}[Propositions \ref{m0} and \ref{mainprop} imply Theorem \ref{mainP}]
By Lemma \ref{local}, to prove Theorem \ref{mainP}, it suffices to bound $\a^4|U_\a|$. By \eqref{sufficient}, it suffices to bound each set
\begin{equation}\label{Ualpham} \{x\in U_\a:|f(x)|\lesssim \e^{-1} |f_{W_m}^{\mc{B}}(x)|\} \end{equation}
for $m=2,\ldots,N$. In the following argument, for any $K\ge 4$, let $\ell(\tau)=K^{-1}$ index a partition of $\mc{N}_{K^{-2}}(\P^1)$ by approximate $K^{-1}\times K^{-2}$ blocks, where we take the dyadic value closest to $K^{-1}$ and use the definition from \eqref{blocks}. Let $f_{W_m,\tau}^{\mc{B}}=\sum_{\theta\subset\tau}f_{W_m,\theta}^{\mc{B}}$. Since $|f_{W_m}^{\mc{B}}|\le\sum_{\ell(\tau)=R^{-\e}}|f_{W_m,\tau}^{\mc{B}}(x)|$, there is a universal constant $C_0>0$ so that $|f_{W_m}^{\mc{B}}(x)|>R^{C_0\e}\max_{\substack{\ell(\tau)=\ell(\tau')=R^{-\e}\\\tau,\tau'\,\,\text{nonadj.}}}|f_{W_m,\tau}^{\mc{B}}(x)f_{W_m,\tau'}^{\mc{B}}(x)|^{1/2}$ implies $|f_{W_m}^{\mc{B}}(x)|\le C_0\max_{\ell(\tau)=R^{-\e}}|f_{W_m,\tau}^{\mc{B}}(x)|$. This gives the first step in the broad-narrow inequality
\begin{align*}
    |f_{W_m}^{\mc{B}}(x)|&\le C_0\max_{\ell(\tau)=R^{-\e}}|f_{W_m,\tau}^{\mc{B}}(x)|+R^{C_0\e}\max_{\substack{\ell(\tau')=\ell(\tau'')=R^{-\e}\\\tau',\tau''\,\text{nonadj.}}}|f_{W_m,\tau'}^{\mc{B}}(x)f_{W_m,\tau''}^{\mc{B}}(x)|^{1/2} .
\end{align*}
Iterate the inequality $m_0\sim\e^{-1}$ times (for the first term) where $R^{m_0\e}\sim R^{1/2}$ to bound $|f_{W_m}^{\mc{B}}(x)|$ by
\begin{align*}
    |f_{W_m}^{\mc{B}}(x)|&\lesssim  C_0^{m_0} \max_{\ell(\tau)=R^{-1/2}}|f_{W_m,\tau}^{\mc{B}}(x)|\\
    &\qquad+C_0^{m_0}R^{C_0\e}\sum_{\substack{R^{-1/2}<\Delta<1\\\Delta\in R^{\N\e}}} \max_{\substack{\ell(\tilde{\tau})= \Delta}}\max_{\substack{\ell(\tau')=\ell(\tau'')= R^{-\e}\Delta\\\tau',\tau''\subset\tilde{\tau},\,\,\text{nonadj.}}}|f_{W_m,\tau'}^{\mc{B}}(x)f_{W_m,\tau''}^{\mc{B}}(x)|^{1/2} .
\end{align*}
Recall that our goal is to bound the size of the set \eqref{Ualpham}. By the triangle inequality and using the notation $\theta$ for blocks $\tau$ with $\ell(\tau)=R^{-1/2}$,
\begin{align}
\label{broadnar1}    &|U_\a|\le |\{x\in \R^2:\a\lesssim \e^{-1}C_0^{m_0}\max_{\theta}|f_{W_m,\theta}^{\mc{B}}(x)|\}|
    +\sum_{\substack{R^{-1/2}<\Delta<1\\\Delta\in R^{\N\e}}} \sum_{\substack{\ell(\tilde{\tau})\sim \Delta\\\ell(\tau')=\ell(\tau'')\sim K^{-1}\Delta\\\tau',\tau''\subset\tilde{\tau},\,\,\text{nonadj.}}}|\text{Br}_\a^m(\tau',\tau'')|
\end{align}
where $\text{Br}_\a^m(\tau',\tau'')$ is the set
\[\{x\in \R^2:\a\lesssim \e^{-1}C_0^{m_0}K^{C_0}|f_{W_m,\tau'}^{\mc{B}}(x)f_{W_m,\tau''}^{\mc{B}}(x)|^{1/2}  \}. \]
The first term in the upper bound from \eqref{broadnar1} is bounded trivially:
\begin{align*}
\a^4|\{x\in \R^2:\a\lesssim \e^{-1}&C^{m_0}\max_{\theta}|f_{W_m,\theta}^{\mc{B}}(x)|\}| \lesssim_\e \int_{\R^2}\sum_\theta|f_{W_m,\theta}^{\mc{B}}|^4 \lesssim_\e \int_{\R^2}\sum_\theta|f_{W_N,\theta}|^4 
\end{align*}
where we used Lemma \ref{pruneprop} in the last line. It remains to note that for each $\theta$, 
\begin{align*}
\int |f_{W_N,\theta}|^4&= \int|\sum_{U\in\mc{G}_\theta}\s_Uf_\theta|^4 \le \sum_{U\in\mc{G}_\theta}\int|\s_U(x)|^2|f_\theta(x)|^4 dx\\
&\lesssim  \sum_{U\in\mc{G}_\theta}\int|\s_U(x)|^2|f_\theta*\widecheck{\rho}_\theta(x)|^4dx\\
(\text{H\"{o}lder's in the convolution})\qquad\qquad&\lesssim \sum_{U\in\mc{G}_\theta}\int|\s_U(x)|^2(|f_\theta|^2*|\widecheck{\rho}_\theta|(x))^2dx \\
(\text{Minkowski's inequality})\qquad\qquad&\lesssim \sum_{U\in\mc{G}_\theta}\Big(\int|f_\theta|^2(y)\Big(\int|\s_U(x)|^2|\widecheck{\rho}_{\theta}|^2(x-y)dy\Big)^{\frac{1}{2}}dy\Big)^2 \\
&\lesssim \sum_{U\in\mc{G}_\theta}|U|^{-1}\Big(\int|f_\theta|^2(y)W_U(y)dy\Big)^2 \end{align*}
where we used the definition of $W_U$ in the final line. 

Now we bound the remaining sets on the right hand side of \eqref{broadnar1}. Fix $\Delta\in R^{\N\e}$, $R^{-1/2}<\Delta<1$, $\tilde{\tau}$ satisfying $\ell(\tilde{\tau})=\Delta$, and $\tau',\tau''\subset\tilde{\tau}$ satisfying $\ell(\tau')=\ell(\tau'')= R^{-\e}\Delta$, $\tau',\tau''$ nonadjacent. By parabolic rescaling, there exists a linear transformation $T$ so that $|f_{W_m,\tau'}^{\mc{B}}\circ T|=|g_{\underline{\tau}'}|$ and $|f_{W_m,\tau''}^{\mc{B}}\circ T|=|g_{\underline{\tau}''}|$ where $\underline{\tau}'$ and $\underline{\tau}''$ are $\sim R^{-\e}$-separated blocks in $\mc{N}_{\Delta^{-2}R^{-1}}(\P^1)$. For each $\theta$, $|f_{W_m,\theta}^{\mc{B}}\circ T|=|g_{\underline{\theta}}|$ where $\underline{\theta}$ is a $\sim \Delta^{-1} R^{-\frac{1}{2}}\times \Delta^{-2}R^{-1}$ block. For $k\ge m$, $W_k\times R$ wave envelopes $U$ become $\Delta W_k\times \Delta^2R$ wave envelopes $\underline{U}$.


Next, we will note that a square function formed from $|g|=|f_{W_m}^{\mc{B}}\circ T|$ at scale $(\Delta W_m)/(\Delta^2 R)$ is high-dominated in the sense of Lemma \ref{highdom}. Let $\underline{\tau}_{m-1}$ be $\sim (\Delta W_{m-1})/(\Delta^{2}R)$ caps satisfying $|g_{\underline{\tau}_{m-1}}|=|f_{W_m,\tau_{m-1}}^{\mc{B}}\circ T|$. Consider the following abbreviated version of the proof of Lemma \ref{highdom}. By Cauchy-Schwarz,
\[ \a^2\lesssim\e^{-2}C_0^{2m_0}R^{2C_0\e}|g_{\underline{\tau}'}(x)||g_{\underline{\tau}''}(x)|\lesssim \e^{-2}C_0^{2m_0}R^{2C_0\e} \#\underline{\tau}_{m-1} \sum_{\underline{\tau}_{m-1}}|g_{\underline{\tau}_{m-1}}|^2(x)    \]
where $\#\underline{\tau_m}=\#\{\tau_m\subset(\tau'\cup\tau''):g_{\underline{\tau}_m}\not\equiv 0\}$. Assume that the square function is low-dominated: 
\begin{equation}\label{contra3} 
\sum_{\underline{\tau}_m}|g_{\underline{\tau}_m}|^2(x) \lesssim |\sum_{\underline{\tau}_m}|g_{\underline{\tau}_m}|^2*\widecheck{\eta}_{\le (\Delta W_{m-1})^{-1}}(x) |.
\end{equation} 
Then by analogues of Lemma \ref{low} and the proof of and Lemma \ref{high}, 
\[ \sum_{\underline{\tau}_{m-1}}|g_{\underline{\tau}_{m-1}}|^2(x) \lesssim\sum_{\underline{\tau}_{m-1}} \max_{\underline{U}\not\in\mc{G}_{\tau_{m-1}}}|\underline{U}|^{-1}\int \sum_{\underline{\theta}\subset\underline{\tau}_{m-1}}|g_{\underline{\theta}}|^2\s_{\underline{U}}^{\frac{1}{2}}. \]
Undoing the change of variables defining $g$, we see that 
\[ |\underline{U}|^{-1}\int \sum_{\underline{\theta}\subset\underline{\tau}_{m-1}}|g_{\underline{\theta}}|^2\s_{\underline{U}}^{\frac{1}{2}}=|{U}|^{-1}\int \sum_{{\theta}\subset{\tau}_{m-1}}|f_{W_m,{\theta}}^{\mc{B}}|^2\s_{{U}}^{\frac{1}{2}}. \]
By definition of $f_{W_m}^{\mc{B}}$, the right hand side is $\le \frac{1}{(\log R)C_0^{m_0}R^{C_0\e}}\frac{\a^2}{(\#\tau_{m-1})^2}$. As in the proof of Lemma \ref{high}, combined with \eqref{contra3}, this leads to a contradiction. Therefore, the function $g_{\underline{\tau}'}+g_{\underline{\tau}''}$ satisfies the same weak high-dominance property as $f^{\mc{B}}_{W_m}$ did in the bound for $\text{Br}_\a^m(\tau,\tau')$ in Proposition \ref{mainprop}. This means that we are in a position to repeat the proof of Proposition \ref{mainprop} at scale $\Delta^2R$, with the function $f_{W_m,\tilde{\tau}}^{\mc{B}}\circ T$, and using the broad caps $\tilde{\tau}',\tilde{\tau}''$, which yields 
\begin{align*} 
\a^4|\mc{J}(T)|^{-1}|\text{Br}_\a^m(\tau',\tau'')|&\lesssim_\e C_0^{4m_0}R^{100C_0\e} \sum_{ \Delta R^{1/2}<\Delta W_k<\Delta W_m}\sum_{\substack{\underline{\tau}_k\\\ell(\underline{\tau}_k)=\Delta^{-1}W_k/R}}\sum_{\underline{U}\in\mc{G}_{\underline{\tau}_k}}|\underline{U}|\left(|\underline{U}|^{-1}\int\sum_{\underline{\theta}\subset\underline{\tau}_k}|g_{\underline{\theta}}|^2W_{\underline{U}}\right)^2, 
\end{align*}
finally implying 
\begin{align*} 
\a^4|\text{Br}_\a^m(\tau',\tau'')|&\lesssim_\e C_0^{4m_0}R^{100C_0\e}  \sum_{  R^{1/2}< W_k<W_m}\sum_{\substack{{\tau}_k\subset\tilde{\tau}\\\ell({\tau}_k)=W_k/R}}\sum_{{U}\in\mc{G}_{{\tau}_k}}|U|\left(|U|^{-1}\int\sum_{{\theta}\subset{\tau}_k}|f_\theta|^2W_U\right)^2 ,
\end{align*}
as desired. 
\end{proof}

\subsection{Theorem \ref{mainP} for general curves \label{mainPgensec}}

In this section, we sketch how to generalize  Theorem \ref{mainP} from $\mb{P}^1$ to curves
\[ \mc{C}=\{(t,\g(t)):|t|\le\frac{1}{2}\} \]
where $\g:[-\frac{1}{2},\frac{1}{2}]\to\R$ is a $C^2$ function satisfying 
\begin{equation}\label{gencurve} \g(0)=0,\quad\g'(0)=0,\quad\frac{1}{2}\le \g''(t)\le 2\qquad\text{for all}\quad |t|\le\frac{1}{2} . \end{equation}
Define a canonical $\sim R^{-1/2}\times R^{-1}$ rectangle $\theta$ to be the portion of $\mc{N}(\mc{C})$ over $t$ in a dyadic interval of length $R^{-1/2}$. The proof of Proposition \ref{mainprop} holds uniformly for curves satisfying \eqref{gencurve}. In particular, the bilinear restriction inequality, Theorem \ref{bilrest}, cited as Theorem 16 from \cite{lvlsets} also holds uniformly over curves satisfying \eqref{gencurve}, by the proof in \cite{lvlsets}. Writing Lemma \ref{local} uniformly in ${\bf{C}}$ is straightforward. 

The biggest adaptation for general curves happens in the broad-narrow analysis which is used to prove that Propositions \ref{m0} and \ref{mainprop} prove Theorem \ref{mainP}. In particular, for $R^{-1/2}<\Delta<1$ and $\ell(\tau)=\Delta$, we no longer have a parabolic rescaling which maps $\tau\cap{\mc{C}}$ to ${\mc{C}}$. Instead, we define a more general map $L:\R^2\to\R^2$ as follows. Suppose that $\tau$ is the intersection of $\mc{N}_{\Delta^2}(\mc{C})$ with the strip $[(l-\frac{1}{2})\Delta,(l+\frac{1}{2})\Delta)\times\R$. Then let 
\[ L(x,y)=(\Delta^{-1}(x-l\Delta),\Delta^{-2}(y-\g'(l\Delta)x-\g(l\Delta)+\g'(l\Delta)l\Delta)). \]
Note that
\[ L({\mc{C}})=\{(s,\tilde{\g}(s)):|s|\le\frac{1}{2}\} \]
where $\tilde{\g}(s)=\Delta^{-2}[\g(\Delta(\w+l))-\g'(l\Delta)\Delta(\w+l)-\g(l\Delta )+\g'(l\Delta)l\Delta]$. The function $\tilde{\g}$ again satisfies \eqref{gencurve}. If $\ell(\theta)=R^{-1/2}$ is a canonical cap of $\mc{C}$, then $L(\theta)$ is a canonical cap of $L(\mc{C})$ with $\ell(L(\theta))=\Delta^{-1}R^{-1/2}$. Finally, we check how the wave envelopes transform under $L$. Let $A:\R^2\to\R^2$ be the linear part of $L$, so $A(x,y)=(\Delta^{-1}x,\Delta^{-2}(y-\g'(l\Delta)x))$. If $R^{-1/2}<s<\Delta$ and $\tilde{\tau}$, $\ell(\tilde{\tau})=\Delta^{-1}s$ is a canonical cap of $\tilde{\g}$ then the wave envelope $U_{\tilde{\tau},\Delta^2R}$ is 
\[ \{\lambda_1(1,\tilde{\g}'(\tilde{c}))+\lambda_2(-\tilde{\g}'(\tilde{c}),1):|\lambda_1|\le s\Delta R \quad\text{and}\quad |\lambda_2|\le  \Delta^2 R\}\]
where $\tilde{\tau}$ lies above the set where $|\w-\tilde{c}|\le\frac{s}{2}$. After applying Proposition \ref{mainprop} at scale $\Delta^2R$, we undo the change of variables to return to our original function. The wave envelope displayed above becomes $A^*(U_{\tilde{\tau},\Delta^2R})$, which is 
\[ \{\lambda_1(\Delta^{-1}-\Delta^{-2}\g'(l\Delta)\tilde{\g}'(\tilde{c}),\Delta^{-2}\tilde{\g}'(\tilde{c}))+\lambda_2(-\Delta^{-1}\tilde{\g}'(\tilde{c})-\Delta^{-2}\g'(l\Delta),\Delta^{-2}):|\lambda_1|\le s \Delta R \quad\text{and}\quad |\lambda_2|\le  \Delta^2 R\}. \]
Using the definition of $\tilde{\g}$, this simplifies to 
\[ \{\lambda_1(\Delta^{-1}-\Delta^{-2}\g'(l\Delta)\tilde{\g}'(\tilde{c}),\Delta^{-2}\tilde{\g}'(\tilde{c}))+\lambda_2(-\g'(\Delta(\tilde{c}+l)),1):|\lambda_1|\le s R \quad\text{and}\quad |\lambda_2|\le  R\}, \]
which is comparable to $U_{L^{-1}(\tilde{\tau}),R}$, as desired.

\section{Small cap decoupling for $\P^1$ from Theorem \ref{mainP} \label{smallcapP} }

Let $\b\in[\frac{1}{2},1]$ and let $\g$ be approximate $R^{-\b}\times R^{-1}$ boxes tiling the $R^{-1}$ neighborhood of $\P^1$. One way to precisely define the $\g$ is as follows: let $s$ be a dyadic number satisfying $R^{-\b}\le s<2R^{-\b}$. Then define the $\gamma$ to be 
\begin{equation}\label{smallblocks} 
\bigsqcup_{|l|\le RW_k^{-1}-2} \{(\xi_1,\xi_2)\in\mc{N}_{W_k^2/R^2}(\P^1):lW_kR^{-1}\le \xi_1<(l+1)W_kR^{-1} \}  \end{equation}
and the two end pieces
\[  \{(\xi_1,\xi_2)\in\mc{N}_{W_k^2/R^2}(\P^1):\xi_1<-1+W_kR^{-1}\} \sqcup  \{(\xi_1,\xi_2)\in\mc{N}_{W_k^2/R^2}(\P^1):1-W_kR^{-1}\le \xi_1\}. \]
Note that with this definition, the $\gamma$ further subdivide the $\theta$, $\ell(\theta)=R^{-1/2}$ defined in \textsection\ref{tools}. Suppose that we have pigeonholed the $f_\g$ so that
\[ \|f_\g\|_p^p\sim_p\|f_\g\|_2^2\]
(for an implicit constant that is uniform for $p$ in a compact set) and $\|f_\g\|_2^2$ is either comparable to a uniform number or equal to $0$, and $\|f_\g\|_\infty\sim 1$ or $0$. Reducing to this case is explained in an analogous context in Section 5 of \cite{gmw}. We wish to verify the following small cap theorem (Theorem 3.1 of \cite{smallcap}) using Theorem \ref{mainP}.

Recall the statement of Theorem \ref{smallcapthmB}:
\begin{theorem*}
Let $\b\in[\frac{1}{2},1]$. For any $p,q\ge 1$ satisfying $\frac{3}{p}+\frac{1}{q}\le 1$, 
\begin{equation}
\int_{B_R}|\sum_\g f_\g|^p\le C_\e R^\e\Big[1+R^{\b(\frac{1}{2}-\frac{1}{q})p}+R^{\b(p-\frac{p}{q}-1)-1}\Big]\big(\sum_\g\|f_\g\|_{L^p(\R^2)}^q\big)^{\frac{p}{q}} .\end{equation}
\end{theorem*}

\begin{proof}[Proof of Theorem \ref{smallcapthmB}]
It suffices to prove Theorem \ref{smallcapthmB} for $f$ satisfying additional hypotheses. By analogous argument as in Section 5 of \cite{gmw}, it suffices to consider $f=\sum_\g f_\g$ satisfying $\|f_\g\|_\infty\sim 1$ or $f_\g=0$ for all $\g$ and 
\[ \|f_\g\|_p^p\sim_p\|f_\g\|_2^2\]
for all $\g$ and all $p\ge 1$. By an argument analogous as in the beginning of Section 5.2 from \cite{lvlsets}, we may assume that for each $s$-arc $\w$ of $\P^1$, the number $\#\{\g:\g\cap\w\not=\emptyset\}$ is comparable to a fixed quantity or it is equal to $0$. The final pigeonholing step is to choose $\a>0$ so that 
\[\int_{B_R\cap\{|f|>R^{-1000}\}}|\sum_\g f_\g|^p\lesssim(\log R)\a^p|U_\a| \]
where $U_\a=\{x\in B_R:|f(x)|\sim\a\}$. The other portion  $\int_{B_R\cap\{|f|<R^{-1000}\}}|f|^p$ of the integral may be dealt with using trivial arguments. Thus our goal for the rest of the proof is to show
\[ \a^p|U_\a|\le C_\e R^\e(1+R^{\b(\frac{p}{2}-\frac{p}{q})}+R^{\b(p-\frac{p}{q}-1)-1})\#\g^{\frac{p}{q}}\|f_\g\|_p^p \]
where $\|f_\g\|_p^p$ is roughly uniform in $\g$. 

The restriction $\frac{3}{p}+\frac{1}{q}\le 1$ implies that $p\ge 3$. By Theorem \ref{mainP}, there is a scale $W_k$ for which 
\[ \a^{4}|U_\a|\le C_\e R^\e \sum_{\tau_k}\sum_{U\in\mc{G}_{\tau_k}(\a)}|U|^{-1}\|S_Uf\|_2^4. \]
Let $\g_k$ denote $\sim\max(R^{-\b},W_k^{-1})$-small caps, so either $\g=\g_k$ or $\g\subset\g_k$. By the definition of $\mc{G}_{\tau_k}(\a)$ and local $L^2$-orthogonality (the argument from the proof of Lemma \ref{low}), we have 
\[ \frac{\a^2}{(\#\tau_k)^2}\lesssim_\e R^\e |U|^{-1}\int\sum_{\theta\subset\tau_k}|f_\theta|^2W_U\lesssim_\e R^\e|U|^{-1}\int \sum_{\g_k\subset\tau_k}|f_{\g_k}|^2W_U. \]
Therefore, 
\[ \a^{\max(4,p)}|U_\a|\lesssim_\e R^{O(\e)}\sum_{\tau_k}\#\tau_k^{\max(4,p)-4}\sum_{U\in\mc{G}_{\tau_k}(\a)}|U|\Big(|U|^{-1}\int\sum_{\g_k\subset\tau_k}|f_{\g_k}|^2W_U\Big)^{\max(\frac{p}{2},2)} . \]
By removing $\max(\frac{p}{2},2)-1$ many factors of $\|\sum_{\g_k\subset\tau_k}|f_{\g_k}|^2\|_\infty\lesssim \#\g\subset\tau_k\#\g\subset\g_k$ and using $L^2$ orthogonality, further conclude that either $\a^2\le\#\g$ or
\begin{align}
\a^{\max(4,p)} |U_\a|&\le C_\e R^\e\sum_{\tau_k}\#\tau_k^{\max(4,p)-4}(\#\g\subset\tau_k\#\g\subset\g_k)^{\max(\frac{p}{2},2)-1}\sum_{\g\subset\tau_k}\|f_\g\|_2^2\nonumber \\
&\lesssim_pC_\e R^\e\#\g\#\tau_k^{\max(4,p)-4}\#\g\subset\tau_k^{\max(\frac{p}{2},2)-1}\#\g\subset\g_k^{\max(\frac{p}{2},2)-1}\|f_\g\|_p^p \label{fromhere}  .\end{align}

\noindent\fbox{Case 1: $3\le p\le 4$.} We claim that in this range,
\[ \a^p|U_\a|\lesssim R^{\b(\frac{p}{2}-\frac{p}{q})}\#\g^{\frac{p}{q}}\|f_\g\|_p^p. \]
Note that if $\a^2\le\#\g$, then 
\[ \a^{p}|U_\a|\le \#\g^{\frac{p}{2}}\|f_\g\|_2^2\lesssim_p \max(1,R^{\b(\frac{1}{2}-\frac{1}{q})p})\#\g^{\frac{p}{q}}\|f_\g\|_p^p, \]
so we assume without loss of generality that $\a^2>\#\g$.

Since $p\le 4$, $\frac{3}{p}+\frac{1}{q}\le 1$ implies that $q\ge 4$ and in particular, $\frac{p}{2}-\frac{p}{q}-1\ge 2-\frac{p}{2}\ge 0$. Now use \eqref{fromhere} with exponent $4$ to see that it suffices to verify that
\[ \frac{1}{\a^{4-p}}\#\g\#\g\subset\tau_k\#\g\subset\g_k \overset{?}{\le}R^{\b(\frac{p}{2}-\frac{p}{q})} \#\g^{\frac{p}{q}}. \]
Using the assumption that $\a^2\ge\#\g$, we check
\begin{align*}
\#\g^{\frac{p}{2}-\frac{p}{q}-1}\#\g\subset\tau_k\#\g\subset\g_k&\overset{?}{\le}R^{\b(\frac{p}{2}-\frac{p}{q})} \\
(\text{implied by})\qquad R^{\b(\frac{p}{2}-\frac{p}{q}-1)}(R^\b W_k/R)\max(1,R^\b W_k^{-1})&\overset{?}{\le}R^{\b(\frac{p}{2}-\frac{p}{q})} \\
    \max(1,R^\b W_k^{-1})&\overset{?}{\le}RW_k^{-1}\qquad(true).
\end{align*}

\noindent\fbox{Case 2: $4\le p\le 6$.} 
Picking up from \eqref{fromhere}, we verify
\begin{align*} 
\#\g\#\tau_k^{p-4}\#\g\subset\tau_k^{\frac{p}{2}-1}\#\g\subset\g_k^{\frac{p}{2}-1}&\overset{?}{\le}R^{\b(p-\frac{p}{q}-1)-1}\#\g^{\frac{p}{q}}\\
\#\g^{p-3}\#\g\subset\tau_k^{3-\frac{p}{2}}\#\g\subset\g_k^{\frac{p}{2}-1}&\overset{?}{\le} R^{\b(p-\frac{p}{q}-1)-1}\#\g^{\frac{p}{q}}\\
\#\g^{p-\frac{p}{q}-3}(R^{\b}W_k/R)^{3-\frac{p}{2}}(R^\b W_k^{-1})^{\frac{p}{2}-1}&\overset{?}{\le} R^{\b(p-\frac{p}{q}-1)-1} \\
\#\g^{p-\frac{p}{q}-3}R^{2\b-3+\frac{p}{2}}W_k^{4-p}&\overset{?}{\le} R^{\b(p-\frac{p}{q}-1)-1} \\
\#\g^{p-\frac{p}{q}-3}&\overset{?}{\le} R^{\b(p-\frac{p}{q}-3)}(W_k/R^{1/2})^{p-4} \qquad(true).
\end{align*}
To verify the final line, we used that $p-\frac{p}{q}-3\ge 0$, which is always true by hypothesis. 

\noindent\fbox{Case 3: $6\le p$.} 
Again, from \eqref{fromhere}, we verify
\begin{align*} 
\#\g \#\tau_k^{p-4}\#\g\subset\tau_k^{\frac{p}{2}-1}\#\g\subset\g_k^{\frac{p}{2}-1}&\overset{?}{\le} R^{\b(p-\frac{p}{q}-1)-1}\#\g^{\frac{p}{q}} \\
\#\tau_k^{\frac{p}{2}-3}\#\g^{\frac{p}{2}}\#\g\subset\g_k^{\frac{p}{2}-1}&\overset{?}{\le} R^{\b(p-\frac{p}{q}-1)-1}\#\g^{\frac{p}{q}} \\
 \#\tau_k^{\frac{p}{2}-3}\#\g^{\frac{p}{2}}\#\g\subset\g_k^{\frac{p}{2}-1}(\#\g_k^{\frac{p}{q}-\frac{p}{2}})&\overset{?}{\le} R^{\b(p-\frac{p}{q}-1)-1}\#\g^{\frac{p}{q}} (\#\g_k^{\frac{p}{q}-\frac{p}{2}})\\
  \#\tau_k^{\frac{p}{2}-3}\#\g\subset\g_k^{p-\frac{p}{q}-1}&\overset{?}{\le} R^{\b(p-\frac{p}{q}-1)-1}\#\g_k^{\frac{p}{q}-\frac{p}{2}}.
 \end{align*}
If $q\ge 2$, then we have
\begin{align*}
\#\g_k^{\frac{p}{2}-\frac{p}{q}}\#\tau_k^{\frac{p}{2}-3}\#\g\subset\g_k^{p-\frac{p}{q}-1}&\overset{?}{\le} R^{\b(p-\frac{p}{q}-1)-1}\\
(\text{implied by})\qquad (\min(R^\b,W_k))^{\frac{p}{2}-\frac{p}{q}}(R W_k^{-1})^{\frac{p}{2}-3}(R^\b\max(R^{-\b},W_k^{-1}))^{p-\frac{p}{q}-1}&\overset{?}{\le} R^{\b(p-\frac{p}{q}-1)-1} .
\end{align*}
If $R^\b\le W_k$, then this is true since $R^{\b(\frac{p}{2}-\frac{p}{q})}(R R^{-\b})^{\frac{p}{2}-3}\overset{?}{\le} R^{\b(p-\frac{p}{q}-1)-1}$ reduces to checking that $p\ge 4$. If $R^\b>W_k$, then it is true since 
$W_k^{\frac{p}{2}-\frac{p}{q}}(RW_k^{-1})^{\frac{p}{2}-3}(R^\b W_k^{-1})^{p-\frac{p}{q}-1}\overset{?}{\le}R^{\b(p-\frac{p}{q}-1)-1}$ reduces to verifying $R^{1/2}\le W_k$.

Finally, assume that $q<2$. The inequality we are trying to verify is 
\[\#\tau_k^{\frac{p}{2}-3}\#\g\subset\g_k^{p-\frac{p}{q}-1}\overset{?}{\le} R^{\b(p-\frac{p}{q}-1)-1}\#\g_k^{\frac{p}{q}-\frac{p}{2}}. \]
Note that $\#\tau_k\le \#\g_k$, so it suffices to verify
\[ \#\tau_k^{p-\frac{p}{q}-3}\#\g\subset\g_k^{p-\frac{p}{q}-1}\overset{?}{\le} R^{\b(p-\frac{p}{q}-1)-1}. \]
Since all of the powers on the left hand side are nonnegative, it now suffices to check
\[ (R W_k^{-1})^{p-\frac{p}{q}-3}(R^\b\max(R^{-\b},W_k^{-1}))^{p-\frac{p}{q}-1}\overset{?}{\le} R^{\b(p-\frac{p}{q}-1)-1}. \]
If $R^\b\le W_k$, this is true because it reduces to verifying $(RR^{-\b})^{p-\frac{p}{q}-3}\overset{?}{\le}R^{\b(p-\frac{p}{q}-1)-1} $, which simplifies to $1\le R^{(2\b-1)(p-\frac{p}{q}-2)}$ (true). If $R^\b>W_k$, this is true since it reduces to verifying $R^{p-\frac{p}{q}-2} W_k^{2p-2\frac{p}{q}-4}$, which clearly follows from $R\le W_k^2$.

$R^{\b(\frac{p}{2}-\frac{p}{q})}$ since $\#\tau_k^{\frac{p}{2}-3}\le\#\g_k^{\frac{p}{2}-1}$. 
\end{proof}

\section{The proof of the main theorem for the cone for the cone \label{mainCpf}}

Theorem \ref{mainC} builds on the proof of Theorem 1.3 from \cite{locsmooth}. The functions we will consider are Schwartz functions $f:\R^3\to\C$ with Fourier transform supported in an $R^{-1}$ neighborhood of 
\[ \Gamma=\{(\xi_1,\xi_2,\xi_3)\in\R^3:\xi_1^2+\xi_2^2=\xi_3^2,\quad\frac{1}{2}\le|\xi_3|\le 2\} ,\]
which we denote $\mc{N}_{R^{-1}}(\Gamma)$. Although many choices of conical caps work in the argument that follows, we will choose a specific way to partition $\mc{N}_{R^{-1}}(\Gamma)$ for concreteness. Let $R\in 4^\N$. Split $[0,2\pi)$ into $R^{1/2}$ many intervals $I_\theta$ of length $2\pi/R^{1/2}$. Writing $(\xi_1,\xi_2,\xi_3)=(\rho\cos\w,\rho\sin\w,z)$ in cylindrical coordinates, each canonical $R^{-1/2}$-cap $\theta$ for the cone is 
\begin{equation}\label{theta} \theta:=\mc{N}_{R^{-1}}(\Gamma)\cap \{\w\in I_\theta\} \end{equation}
and we may write $\mc{N}_{R^{-1}}(\Gamma)=\sqcup\theta$. For another dyadic value $s\ge R^{-1/2}$, an $s$-cap $\tau$ satisfies either $\theta\cap \tau=\emptyset$ or $\theta\subset\tau$. Note that each $s$-cap $\tau$ is approximately a plank of dimension $1\times s\times s^2$, meaning that there is some rectangular plank $P$ of dimensions $1\times s\times s^2$ so that $cP\subset \tau\subset CP$ for absolute constants $c,C>0$, where the $cP$ and $CP$ are dilations of $P$ with respect to its centroid. Define the Fourier projections $f_\tau$ by 
\[ f_\tau(x)=\int_\tau \widehat{f}(\xi)e^{2\pi i x\cdot\xi}d\xi . \]

The parameter $\a>0$ describes the size of functions in the superlevel sets 
\[ U_\a=\{x\in B_R:|f(x)|\ge \a\},\]
where $B_R$ is any $R$-ball in $\R^3$. Next we summarize the relevant notation, occasionally with superficial differences, from Section 3 of \cite{locsmooth} so that we may easily reference certain parts of their argument. 
\begin{enumerate}
    \item Let ${\bf{S}}_s$ denote the collection of $s$-caps $\tau$. Our $s$-caps $\tau$ are comparable to the caps $\tau(s,\xi)$ defined at the beginning of Section 3 of \cite{locsmooth}, which we show in \textsection\ref{conesmallcap}. 
    \item For $\tau\in{\bf{S}}_s$, let the associated $I_\tau$ in the definition of $\tau$ have left endpoint $\w_\tau$. Then define the frame ${\bf{c}}(\w_\tau)=(\cos\w_{\tau},\sin\w_\tau,1)$, ${\bf{n}}(\w_\tau)=(\cos\w_{\tau},\sin\w_{\tau},-1)$, and ${\bf{t}}(\w_{\tau})=(-\sin\w_{\tau},\cos\w_\tau,0)$. Define $U_{\tau,R}$ by
    \[ U_{\tau,R}=\{x\in\R^3:|{\bf{c}}(\w_{\tau})\cdot x|\le Rs^2,\quad|{\bf{n}}(\w_\tau)\cdot x|\le R,\quad |{\bf{t}}(\w_\tau)\cdot x|\le sR\}.\]
    Write $U\| U_{\tau,R}$ to index a tiling of $\R^3$ by wave envelopes $U$ which are translates of $U_{\tau,R}$.  
    \item For $U\| U_{\tau,R}$, define $W_U$ analogously as in \eqref{weightdef}. The average integral notation $\fint_Ug$ here means
    \[ \fint_Ug= |U|^{-1}\int g W_U. \]
    \item Given any $U\|U_{\tau,R}$
    , define the square function 
    \[ S_Uf(x)=\big(\sum_{\substack{\theta\in\mc{S}_{R^{-1/2}}\\ \theta\subset \tau }}|f_\theta|^2(x)W_U(x)\big)^{1/2} .\]
    These partial square functions are larger than the analogous partial square functions from \cite{locsmooth}, which is fine because we do not invoke any results from \cite{locsmooth} specifically regarding the partial square functions in our argument. Note that for any $r$-cube $B_r$, 
    \[ S_{B_r}f(x)=\big(\sum_{\tau\in{\bf{S}}_{r^{-1/2}}}|f_\tau|^2(x)W_{B_r}(x)\big)^{1/2}. \]
    \item Define the average $L^p$ norm of the square function $S_Uf$ by 
    \[ \|S_Uf\|_{L^p_{avg}}=\big(|U|^{-1}\int |S_Uf|^p\big)^{1/p} . \]
\end{enumerate}

Recall the statement of the theorem we need to prove. 
\begin{theorem*}\label{mainC*} For $\e>0$, there exists $C_\e>0$ such that the following inequality holds for $R$ sufficiently large. For any $\a>0$,
\begin{equation}\label{mainCeqn} \alpha^4 |U_\alpha| \le C_\e R^\e \sum_{\substack{R^{-1/2}<s<1\\ s\quad\text{dyadic}}}\sum_{\tau\in{\bf{S}}_s} \sum_{\substack{U\|U_{\tau,R}\\
C_\e R^\e\|S_Uf\|_{L^2_{avg}}\ge \frac{\a}{\#\tau}}} |U|\|S_Uf\|_{L^2_{avg}}^4  \end{equation}
where $f:\R^3\to\C$ is a Schwartz function with Fourier transform supported in $\mc{N}_{R^{-1}}(\Gamma)$.
\end{theorem*}
The notation $\#\tau$ means the size of the set $\{\tau\in{\bf{S}}_s: f_\tau\not=0\}$. In an analogous fashion to \cite{locsmooth}, Theorem \ref{mainC} is proved by analyzing an auxiliary constant $S(r,R)$ defined presently. To define $S(r,R)$ adapted to our set-up, we need an expression for $\#\tau$ which behaves well in rescaling arguments. Let $K=K(\e)$ be a constant. For $s\in [R^{-1/2},1]$, define $\mu_f(s)$ to be the supremum over partitions $s\le s_1\le \cdots\le s_k\le K^{-1}$, with $s_i\le K^{-1}s_{i+1}$ for each $i$, of the quantity 
\begin{equation}\label{mudef}
    \#\{\tau_{s_k}\in{\bf{S}}_{s_k}:f_{\tau_{s_k}}\not\equiv 0\}\prod_{i=1}^{k-1}\max_{\tau_{s_{i+1}}\in{\bf{S}}_{s_{i+1}}}\#\{\tau_{s_i}\in{\bf{S}}_{s_i}:\tau_{s_i}\subset\tau_{s_{i+1}},\quad f_{\tau_{s_i}}\not\equiv 0\}. 
\end{equation} 
If $s>K^{-1}$, then let $\mu_f(s)=\{\tau_s\in{\bf{S}}_{s}:f_{\tau_s}\not\equiv 0\}$. Define $S(r,R)$ to be the infimum of the set of $A>0$ satisfying \begin{equation}\label{constdef}
\sum_{\substack{B_r\subset\R^3\\\|S_{B_r}f\|_{L^2_{avg}}\ge \a }}|B_r|\|S_{B_r}f\|_{L^2_{avg}}^4 \le A \sum_{R^{-1/2}\le s\le 1} \sum_{\tau\in{\bf{S}}_s} \sum_{\substack{U \parallel U_{\tau,R}\\ A\|S_Uf\|_{L^2_{avg}}\ge \frac{\a}{\mu_f(s)}}} |U|\|S_Uf\|_{L^2_{avg}}^4  \end{equation}
for all $\a>0$ and all Schwartz functions $f:\R^3\to\C$ with Fourier transform supported in $\mc{N}_{R^{-1}}(\Gamma)$. Note that $S(r,R)<\infty$ since for $R$-cubes $B_R$ tiling $\R^3$ and any $B_r\subset B_R$,  
\begin{align*}
\|S_{B_r}f\|_{L^2_{avg}}^2&=\fint_{B_r}\sum_{d(\tau)=r^{-1/2}}|f_\tau|^2\\
    \text{(Cauchy-Schwarz)}\qquad   &\lesssim \fint_{B_r}\sum_{d(\theta)=R^{-1/2}}|f_\theta|^2(\#\theta\subset\tau) \lesssim \Big(\frac{R}{r}\Big)^{7/2} \|S_{B_R}f\|_{L^2_{avg}}^2 .
\end{align*}
So in particular,
\begin{align*}
\sum_{\substack{B_r\subset\R^3\\\|S_{B_r}f\|_{L^2_{avg}}\ge \a }}|B_r|\|S_{B_r}f\|_{L^2_{avg}}^4 \le (R/r)^{10} \sum_{\substack{B_R \\ (R/r)^{\frac{7}{4}}\|S_Uf\|_{L^2_{avg}}\ge \frac{\a}{\mu_f(1)} }}|B_R|\|S_{B_R}f\|_{L^2_{avg}}^4   .
\end{align*}
Noticing that the right hand side corresponds to the $s=1$ term of \eqref{constdef}, meaning that \eqref{constdef} is satisfied with $A\sim (R/r)^{10}$.

We will state the following three lemmas which have analogues in \cite{locsmooth} and will be used to analyze $S(r,R)$.  
The first lemma is proved via a Kakeya type estimate (Proposition \ref{kakcone} below). 
\begin{lemma}[Analogue of Lemma 3.1 of \cite{locsmooth}] \label{kaklem} There is some $C>0$ so that for any $r\ge 10$ and $r_1\in[r,r^2]$,
\[ S(r_1,r^2)\le C\big(\log \frac{r^2}{r_1}\big)^{10}. \]
\end{lemma}

We use the observation of Bourgain and Demeter \cite{BD} that a neighborhood of a truncated piece of $\Gamma$ is comparable to a cylindrical neighborhood of a parabola. This allows Theorem \ref{mainP} to come into play and establish the base case of the induction-on-scales argument for bounding $S(r,R)$. Let $K=K(\e)$ be the same constant from the definition of $\mu_f(s)$ which we will now use to measure the truncation. We replace $S(r,R)$ by $S_K(r,R)$, which has the same defining inequality for $S(r,R)$ but restricted to functions $f$ with Fourier transform supported in a $K^{-1}$-truncation of $\Gamma$ (defined precisely in Section \ref{trunc}). Since $K$ is a constant, splitting $\Gamma$ into $O(K)$ many truncated regions and decoupling each region suffices to bound $S(r,R)$. See Section \ref{trunc} for a detailed definition of $S_K$. 
\begin{lemma}[Analogue of Lemma 3.2 of \cite{locsmooth}]\label{basecase} For any $K\ge 10$, any $1\le r\le R\le K$, and any $\delta>0$, $S_K(r,R)\le C_\delta K^\delta$.
\end{lemma}

The following lemma is our analogue of Lemma 3.3 from \cite{locsmooth}, which plays an analogous role as parabolic rescaling does in decoupling. For technical reasons, we require an extra factor $D_\d(K,r_2)$ which, after choosing parameters, will be $\lessapprox 1$ in an appropriate sense. 
\begin{lemma} \label{rescaling} For any $0<\d<1$ and $10\le r_1< r_2\le r_3$,
\[S_K(r_1,r_3)\le D_\d(K,r_2)K^{3\d}\log r_2\cdot S_K(r_1,r_2)\max_{r_2^{-1/2}\le s\le 1}S_K(s^2r_2,s^2r_3), \]
where $D_\d(K,r_2)$ is from Lemma \ref{pig}. 
\end{lemma}

Finally, assuming the above three lemmas, we may bound $S_K(r,R)$ (and therefore bound $S(r,R)$).  
\begin{prop}[Analogue of Proposition 3.4 of \cite{locsmooth}]\label{induct} For any $\e>0$, there exists $K=K(\e)$ so that for any $1\le r\le R$, we have
\[ S_K(r,R)\le \tilde{C}_\e (R/r)^{\e} .\]
\end{prop}

\begin{proof} We proceed analogously as in the proof of Proposition 3.4 in \cite{locsmooth}, written here for convenience. 
First note that if $r>R^{1/2}$, then Lemma \ref{kaklem} says that $S(r,R)\le C [\log (R/r)]^{10}$. Therefore, the conclusion holds in this case. 

The constant $K=K(\e)>10$ will be chosen below based on Lemmas \ref{basecase} and \ref{rescaling}. 

\noindent \underline{Base case:} $R/r\le \sqrt{K}$. Without loss of generality (based on the first note in this proof), assume $r\le R^{1/2}$. Combined with $R/r\le\sqrt{K}$, this means $R\le K$. Since $K$ is a constant depending only on $\e$, the inequality $S(r,R)\le \tilde{C}_K=\tilde{C}_\e$ is easily verified using a trivial argument like the one following \eqref{constdef}.

\noindent{\underline{Inductive hypothesis:}} Suppose that given $(r,R)$, for any pair $(r',R')$ satisfying $R'/r'\le R/{2r}$, we have $S_K(r',R')\le \tilde{C}_\e (R'/r')^\e$.

\noindent\underline{Case 1: $r\le K^{1/2}$.} By Lemma \ref{rescaling}, for each $0<\d<1$,
\[S_K(r,R)\le (C\d \log K)^{\d^{-1}\log (K^{3/2})/\log K}K^{3\d}\log K\cdot S_K(r,K^{1/2}r)\max_{(K^{1/2}r)^{-1/2}\le s\le 1}S_K(s^2K^{1/2}r,s^2R) . \]
The first factor of $S_K$ is bounded by Lemma \ref{basecase} and the second factor of $S_K$ is bounded by induction, yielding the inequality
\[ S_K(r,R)\le (C\d \log K)^{\d^{-1}\log (K^{3/2})/\log K}K^{3\d}\log KC_\d K^\d \tilde{C}_\e (R/K^{1/2}r)^\e. \]
Choose $\d=\e/20$ and let $K=K(\e)$ be large enough so 
\begin{align*} 
(C(\e/20) \log K)^{(20/\e)\log (K^{3/2})/\log K}&K^{3\e/20}\log KC_{\e/20}K^{\e/20}K^{-\e/2}\\
&\le \log K\cdot K^{30\e^{-1}\log (C\e\log K)/\log K}C_{\e/20}K^{-3\e/10}\le 1,\end{align*}
which closes the induction in this case.

\noindent\underline{Case 2: $r>K^{1/2}$.} Again by Lemma \ref{rescaling}, we have that for each $0<\d<1$,
\[ S_K(r,R)\le (C\d \log K)^{\d^{-1}\log (K^2r)/\log K}K^{3\d} \log r\cdot S_K(r,r^2) \max_{r^{-1}\le s\le 1}S_K(s^2r^2,s^2R). \]
Take $\d=\frac{\e}{20}$ and use Lemma \ref{kaklem} for the first factor of $S_K$ and induction for the second factor, yielding
\[ S(r,R)\le (C(\e/20) \log K)^{20\e^{-1}\log (K^2r)/\log K}K^{3\e/20}\log r C(\log r)^{10}\tilde{C}_\e (R/r^2)^\e   . \]
Choose $K$ large enough so that $r\ge K^{1/2}$ implies that 
\begin{align*}
(C(\e/20) \log K)^{20\e^{-1}\log (K^2r)/\log K}&K^{3\e/20} \log r C(\log r)^{10}r^{-\e}\\
&\le (K^2r)^{20\e^{-1}\log(C\e \log K)/\log K}K^{3\e/20} \log r C(\log r)^{10}r^{-\e}\le 1,
\end{align*}
which closes the induction. 
\end{proof}

\section{Theorem \ref{mainC} from Proposition \ref{induct}\label{mainthmpf}}

Theorem \ref{mainC} follows from Proposition \ref{induct} (with $r=1$ and $S$ in place of $S_K$) after an initial pigeonholing step. 
\begin{proof}[Theorem \ref{mainC} from Proposition \ref{induct}]

Begin with a pigeonholing process to guarantee that the distribution of the Fourier support of $f$ is somewhat regular. This will allow us to show that $\mu_f(s)\lessapprox \#\tau$, when $\tau\in{\bf{S}}_s$, in an appropriate sense. For the initial step, let $K^{-N\e}\le R^{-1/2}K^{-\e}\le K^{-N\e+\e}$ and write 
\[ \{\tau_{N}\in{\bf{S}}_{K^{-N\e}}:\exists\theta\in{\bf{S}}_{R^{-1/2}}\,\text{s.t.}\,f_\theta\not=0,\,\,\theta\subset\tau_{N}\}=\sum_{1\le \lambda\le K^\e}\Lambda_{N}(\lambda) \]
where $\lambda$ is a dyadic number, $\Lambda_N(\lambda)=\{\tau_{N}\in{\bf{S}}_{K^{-N\e}}:\#\theta\subset\tau_{N}\sim \lambda\}$, $\#\theta\subset\tau_{N}$ means $\#\{\theta\subset\tau_{N}:f_\theta\not\equiv0\}$, and $\#\theta\subset\tau_{N}\sim\lambda$ means $\lambda\le \#\theta\subset\tau_{N}<2\lambda$. Note that we use the convention that ${\bf{S}}_{K^{-N\d}}$ actually means ${\bf{S}}_{s}$ where $s$ is the largest dyadic number satisfying $s\le K^{-N\e}$. Since there are $\lesssim \e\log K$ many $\lambda$ in the sum, there exists some $\lambda_{N}$ such that
\[ |\{x:|f(x)|>\a\}|\le C(\e\log K)|\{x:C(\log K)|\sum_{\tau_{N}\in\Lambda_{N}(\lambda_{N})}f_{\tau_{N}}(x)|>\a\}|. \]
Write $f^N=\sum_{\tau_N\in\Lambda_N(\lambda_N)}f_{\tau_N}$. Continuing in this manner, we have
\[ \{\tau_k\in{\bf{S}}_{K^{-k\e}}:\exists \tau_{k+1}\in\Lambda_{k+1}(\lambda_{k+1})\,\,\text{s.t.}\,\,f_{\tau_{k+1}}^{k+1}\not\equiv 0,\tau_{k+1}\subset\tau_k\}=\sum_{1\le \lambda\le K^\e}\Lambda_k(\lambda) \]
where $\Lambda_k(\lambda)=\{\tau_k\in{\bf{S}}_{K^{-k\e}}:\#\{\tau_{k+1}\in\Lambda_{k+1}(\lambda_{k+1}):f^{k+1}_{\tau_{k+1}}\not\equiv0,\,\tau_{k+1}\subset\tau_k\}\sim\lambda\}$ and for some $\lambda_k$, 
\begin{align*} 
|\{x:(C\e(\log K))^{N-k}|&f^{k+1}(x)|\ge |f(x)|>\a\}|\\
&\qquad\qquad \le C(\e\log K)|\{x:(C(\e\log K))^{N-k+1}|f^{k}(x)|\ge |f(x)|>\a\}|  
\end{align*} 
with $f^k=\sum_{\tau_k\in\Lambda_k(\lambda_k)}f^{k+1}_{\tau_k}$. Continue this process until we have found $f^1$, $\lambda_1$ so that
\begin{align*} 
|\{x:|f(x)|>\a\}|\le (C\e \log K)^{N}|\{x:C^{\e^{-1}}(\log R)^{N}|f^1(x)|>\a\}|  ,
\end{align*} 
noting that $N\sim \e^{-1}\log R/\log K$, so for an appropriate constant $c>0$, $(\e\log K)^{N}\le R^{\frac{c\e^{-1}\log\log K}{\log K}}$ (which, since $K$ may be arbitrarily large depending on $\e$, is an acceptable loss factor). The function $f^1$ now satisfies the property that $\mu_{f^1}(s)$ may be controlled by $\#\{\tau\in{\bf{S}}_s:f^1_\tau\not\equiv 0\}$. Indeed, let $s\le s_1\le\cdots s_k\le 1$ be a partition satisfying $s_i\le K^{-1}s_{i+1}$. The associated quantity in the definition of $\mu_{f^1}(s)$ is 
\begin{equation}\label{potquantity} \#\{\tau_{s_k}\in{\bf{S}}_{s_k}:f^1_{\tau_{s_k}}\not\equiv 0\}\prod_{i=1}^{k-1}\max_{\tau_{s_{i+1}}\in{\bf{S}}_{s_{i+1}}}\#\{\tau_{s_i}\in{\bf{S}}_{s_i}:\tau_{s_i}\subset\tau_{s_{i+1}},\quad f_{\tau_{s_i}}^1\not\equiv 0\}.\end{equation}
For each $s_i$, use the notation $\lceil s_i\rceil$ to denote the smallest element $b\in K^{\e\N}$ satisfying $s_i\le b$ and $\lfloor s_i\rfloor$ to denote the largest element $a\in K^{-\e \N}$ satisfying $a\le s_i$. Then note that the first term and the $i=k-1$ term from \eqref{potquantity} are bounded by 
\begin{align*} 
&\#\{\tau_{\lfloor s_k\rfloor}\in{\bf{S}}_{\lfloor s_k\rfloor}: f_{\tau_{\lfloor s_k\rfloor }}^1\not\equiv 0\}\max_{\tau_{\lceil s_{k}\rceil}\in{\bf{S}}_{\lceil s_{k}\rceil}}\#\{\tau_{\lfloor s_{k-1}\rfloor}\in{\bf{S}}_{\lfloor s_{k-1}\rfloor}:\tau_{\lfloor s_{k-1}\rfloor }\subset\tau_{\lceil s_{k}\rceil },\quad f_{\tau_{\lfloor s_{k-1}\rfloor }}^1\not\equiv 0\} \\
    &\qquad \le CK^\e\#\{\tau_{\lceil s_k\rceil}\in{\bf{S}}_{\lceil s_k\rceil}: f_{\tau_{\lceil s_k\rceil }}^1\not\equiv 0\}\max_{\tau_{\lceil s_{k}\rceil}\in{\bf{S}}_{\lceil s_{k}\rceil}}\#\{\tau_{\lfloor s_{k-1}\rfloor}\in{\bf{S}}_{\lfloor s_{k-1}\rfloor}:\tau_{\lfloor s_{k-1}\rfloor }\subset\tau_{\lceil s_{k}\rceil },\quad f_{\tau_{\lfloor s_{k-1}\rfloor }}^1\not\equiv 0\} \\
    &\qquad \le 2CK^\e \#\{\tau_{\lfloor s_{k-1}\rfloor}\in{\bf{S}}_{\lfloor s_{k-1}\rfloor}: f_{\tau_{\lfloor s_{k-1}\rfloor }}^1\not\equiv 0\}
\end{align*}
where we used our construction in the final inequality. Continuing in this manner with successive terms from \eqref{potquantity}, we obtain the upper bound 
\[ (2CK^\e)^{k-1}\#\{\tau_{\lfloor s_{1}\rfloor}\in{\bf{S}}_{\lfloor s_{1}\rfloor}: f_{\tau_{\lfloor s_{1}\rfloor }}^1\not\equiv 0\}.\]
By the spacing condition $s_i\le K^{-1}s_{i+1}$ and the constraint $R^{-\frac{1}{2}}\le s$, we have $k\le c\ln R/\ln K$ for an appropriate constant $c>0$. Therefore, the build-up of constants is bounded by $(2CK^\e)^{c\ln R/\ln K} \le R^{2\e} $. The other factor in the above displayed expression is clearly bounded by $\{\tau_s\in{\bf{S}}_s:f_{\tau_s}^1\not\equiv 0\}$. By our construction of $f^1$, this is bounded by $\{\tau_s\in{\bf{S}}_s:f_{\tau_s}\not\equiv 0\}$, as desired.

It remains to show that 
\begin{equation}\label{remains} \a^4||\{x\in B_R:(C\e\log K)^N|{f}^1(x)|>\a\}|\le   \sum_{\substack{B_1\subset\R^3\\C_\e R^{\e}\|S_{B_1}{f}^1\|_{L^2_{avg}}\ge \a }}|B_1|\|S_{B_1}{f}^1\|_{L^2_{avg}}^4 \end{equation}
and then invoke Proposition \ref{induct} with $r=1$ and $R$, since $|S_Uf^1|\le |S_Uf|$ for any wave envelope $U$. Let $B_1$ be a $1$-cube in a tiling of $\R^3$ which has nonempty intersection with $U_\a$. If $\rho$ is a smooth bump function equal to $1$ on $[-3,3]^3$ and supported in $[-6,6]^3$, then by Fourier inversion, ${f^1}={f^1}*\widecheck{\rho}$ pointwise. Then if $x\in U_\a\cap B_1$, 
\begin{align*} 
\frac{\a^2}{(C\e \log K)^{2N}}&< |{f^1}(x)|^2=|{f}^1*\widecheck{\rho}(x)|^2\lesssim |f^1|^2*|\widecheck{\rho}|(x)\lesssim \|S_{B_1}f^1\|_{L^2_{avg}}^2.  
\end{align*}
The inequality \eqref{remains} is an immediate consequence of the inequality displayed above, which concludes the proof.

\end{proof}

\section{A refined Kakeya-type estimate (Proof of Lemma \ref{kaklem})\label{kak} }

Unwinding the definitions in Lemma \ref{kaklem}, it suffices to prove the following proposition, which is a refinement of Lemma 1.4 in \cite{locsmooth}. 
\begin{prop}\label{kakcone} There exists $C>0$ large enough so that for any $10\le r$, $r\le r_1\le r^2$, and $\a>0$, 
\begin{equation}\label{ineq} \sum_{\substack{B_{r_1}\subset\R^3\\\|S_{B_{r_1}}f\|_{L^2_{avg}}\ge \a}}|B_{r_1}|\|S_{B_{r_1}}f\|_{L^2_{avg}}^4\le C(\log \frac{r^2}{r_1})^{10}\sum_{r^{-1}\le s\le 1}\sum_{\tau\in{\bf{S}}_s}\sum_{\substack{U\|U_{\tau,r^2}\\C(\log (r^2/r_1))^3\|S_Uf\|_{L^2_{avg}}^2\ge \frac{\a^2}{{\#\tau}}}}|U|\|S_Uf\|_{L^2_{avg}}^4.  \end{equation}
\end{prop}
As before, $\#\tau=\#\{\tau\in{\bf{S}}_s:f_\tau\not=0\}$. The proof of Proposition \ref{kakcone} will follow the proof of Lemma 1.4 of \cite{locsmooth} after an additional pruning process that has a similar flavor as \textsection\ref{pruning} for the parabola. We will use notation that is consistent with the proof of Lemma 1.4 from \cite{locsmooth}.

\begin{enumerate}
    \item The parameters $\sigma_i$ vary over dyadic values in the range $\sigma_1=r^{-1}\le \sigma_n\le \sigma_N=1$. 
    \item For each $\sigma_n$, fix a collection $\{\tau_n\}={\bf{S}}_{\sigma_n^{-1}r^{-1}}$ of $\sim 1\times\sigma_n^{-1}r^{-1}\times \sigma_n^{-2}r^{-2}$ planks tiling $\mc{N}_{\sigma_n^{-2}r^{-2}}(\Gamma)$.
    \item Distinguish the finest frequency blocks by $\tau_N=\theta\in{\bf{S}}_{r^{-1}}$. 
    \item Let $\#\tau_n$ denote
 \[ \#\tau_n=\#\{\tau_n\in{\bf{S}}_{\sigma_n^{-1}r^{-1}}:f_{\tau_n}\not=0\}. \] 
    \item For each $\tau_n$, let $U_{\tau_n,r^2}$ be a dual envelope of dimensions $ \sigma_n^{-2}\times\sigma_n^{-1}r\times r^2$ defined at the beginning of \textsection\ref{mainCpf}. 
    \item Let $1=\sum_{U\pp U_{\tau_n,r^2}}\s_{U}$ be a partition of unity associated to the tiling $U\|U_{\tau,r^2}$, defined analogously as in \textsection\ref{pruning}. 
    \item Distinguish the parameter $\sigma_{N_0}$ satisfying $\sigma_{N_0}=(r_1/r)^{-1}$. 
\end{enumerate}

Now we define the pruning process. The constant $C_0$ used in the definition of the good sets $\mc{G}_{\tau_n}$ will be chosen in the proof of Lemma \ref{weakhi}. 

\begin{definition}[Pruning with respect to $\tau_n$ for the cone]\label{conepruning} For $\tau_{N_0}\in{\bf{S}}_{\sigma_{N_0}^{-1}r^{-1}}$, define 
\begin{align*} 
\mc{G}_{\tau_{N_0}}&:=\{U\|U_{\tau_{N_0},r^2}:C_0(\log (r^2/r_1))^5|U|^{-1}\int\sum_{\substack{\theta\in{\bf{S}}_{r^{-1}}\\\theta\subset\tau_{N_0}}}|f_\theta|^2W_U\ge \frac{\a^2}{\#\tau_{N_0}}\}, \\
f_{\sigma_{N_0},\theta}&:=\sum_{\substack{U\in\mc{G}_{\tau_{N_0}}}}\s_Uf_\theta \qquad \text{and}\qquad f_{\sigma_{N_0}}:=\sum_{\theta\in{\bf{S}}_{r^{-1}}} f_{\sigma_{N_0},\theta}. 
\end{align*} 
For each $n<N_0$ and each $\tau_n\in{\bf{S}}_{\sigma_n^{-1}r^{-1}}$, let 
\begin{align*} 
\mc{G}_{\tau_n}:=\{U&\|U_{\tau_n,r^2}:C_0(\log( r^2/r_1))^{5}|U|^{-1}\int\sum_{\theta\subset\tau_n}|f_{\sigma_{n+1},\theta}|^2W_U \ge \frac{\a^2}{\#\tau_n}\}, \\
f_{\sigma_n,\theta}&:=\sum_{\substack{U\in\mc{G}_{\tau_n}}}\s_U f_{\sigma_{n+1},\theta} \quad\text{where}\quad\theta\subset\tau_n, \qquad\text{and}\qquad f_{\sigma_n}:=\sum_{\theta}f_{\sigma_n,\theta}.
\end{align*}
Finally, write $f_{\sigma_{n+1}}-f_{\sigma_n}:=f_{\sigma_n}^{\mc{B}}$ and $f_{\sigma_{n},\theta}^{\mc{B}}=\sum_{\substack{U\|U_{\tau_n}\\ U\not\in\mc{G}_{\tau_n}}}\s_Uf_{\sigma_{n+1},\theta}$. 
\end{definition}

\begin{lemma}\label{conepruneprops}
\begin{enumerate}
    \item $\supp f_{\sigma_n,\theta}\subset \underline{C}(\log (r^2/r_1))\tilde{\theta}$ and $\supp f_{\sigma_n,\theta}^{\mc{B}}\subset \underline{C}(\log (r^2/r_1))\tilde{\theta}$. 
    \item $|f_{\sigma_n,\theta}^{\mc{B}}|\le |f_{\sigma_{n+1},\theta}|\le |f_{\sigma_k,\theta}|\le |\sum_{\substack{U\|U_{\tau_k,r^2}\\ U\in\mc{G}_{\tau_k}}}\s_Uf_\theta|\le |f_\theta|$ for any $n\le k\le N_0-1$.  
    \item $f_{\sigma_n,\theta}=f_{\sigma_{n-1},\theta}^{\mc{B}}+\cdots+f_{\sigma_1,\theta}^{\mc{B}}+f_{\sigma_1}$ for all $1<n\le N_0$.
\end{enumerate}

\end{lemma}
\begin{proof} Fix $\theta\in{\bf{S}}_{r^{-1}}$ and let $\tau_n\in{\bf{S}}_{\sigma_n^{-1}r^{-1}}$ satisfy $\theta\subset\tau_n$. The first property follows from the identity
\begin{align*}
    f_{\sigma_n,\theta}=\sum_{U_n\in\mc{G}_{\tau_n}}\s_{U_n}f_{\sigma_{n+1},\theta}=\cdots&=\sum_{U_n\in\mc{G}_{\tau_n}}\s_{U_n}\cdots\sum_{U_N\in\mc{G}_{\tau_{N_0}}}\s_{U_{N_0}}f_\theta \\
    &= \sum_{U_n\in\mc{G}_{\tau_n}}\cdots\sum_{U_{N_0}\in\mc{G}_{\tau_{N_0}}}\prod_{k=n}^{N_0} \s_{U_k}f_\theta
\end{align*}
where $f_{\sigma_{N_0+1},\theta}=f_\theta$. Let $U_k^*=\{\xi:|\xi\cdot x|\le 1\quad\forall x\in U_{\tau_k,r^2}\}$. Them each summand $\prod_{k=n}^{N_0} \s_{U_k}f_\theta$ has Fourier transform supported in  
\[ 2U_n^*+\cdots+2U_{N_0}^*+\theta.  \]
Since each $U_k^*\subset\theta-\theta$ and the number of scales is $N_0\lesssim \log (r^2/r)$, $\supp f_{\sigma_n,\theta}\subset \underline{C}(\log (r^2/r_1))\theta$ for a large enough constant $\underline{C}$. The argument for $f_{\sigma_n,\theta}^{\mc{B}}$ is the same. The second property follows from the fact that $\sum_{U\|U_{\tau_n,r^2}}\s_U\le 1$ for all choices of $\tau_n$, for all scales $n=1,\ldots,N_0$. Property (3) is clear from the definition of the pruning process, particularly the relation $f_{\sigma_{n+1}}=f_{\sigma_n}^{\mc{B}}+f_{\sigma_n}$. 

\end{proof}

Next we define some auxiliary functions which will be used to approximate weight functions $W_{B_{r_1}}$ by a version $\widetilde{W}_{B_{r_1}}$ with controlled Fourier support. 
\begin{definition}\label{weight} 
Let $\p(x):\R^2\to[0,\infty)$ be a radial, smooth bump function satisfying $\p(x)=1$ when $|x|\le 1$ and $\p(x)=0$ if $|x|>2$. For each $k\in\Z_{\ge 0}$, let $\chi_{A_k}=\p(\frac{x}{2^k r_1})-\p(\frac{x}{2^{k-1} r_1})$ if $k\ge 1$ and let $\chi_{A_0}=\p(\frac{x}{r_1})$. Let $H:\R^3\to[0,\infty)$ be a Schwartz function satisfying $H\gtrsim 1$ on $B_{r_1}$ and $\widehat{H}$ is supported where $|\xi|\le r_1^{-1}$.  Define  
\begin{align}\label{weight2}
    \widetilde{W}_{B_{r_1}}(x)=\sum_{k=0}^\infty \frac{1}{(2^k)^{10}}\int H(x-y)\chi_{A_k}(y)dy. 
\end{align}
Note that $W_{B_{r_1}}\sim \widetilde{W}_{B_{r_1}}$ and $\widehat{\widetilde{W}}_{B_{r_1}}$ is supported in $|\xi|\le r_1^{-1}$. 
\end{definition}

Recall some more notation from \cite{locsmooth}. Call $\xi\in\R^3$ an orientation if $\xi_1^2+\xi_2^2=\xi_3^2$ and $\xi_3=1$. For each orientation $\xi$, define the core line direction ${\bf{c}}(\xi)=(\xi_1,\xi_2,1)$, the normal direction  ${\bf{n}}(\xi)=(\xi_1,\xi_2,-1)$, and the tangent direction ${\bf{t}}(\xi)=(-\xi_2,\xi_1,0)$. For each $\theta\in{\bf{S}}_{r^{-1}}$, write $\tilde{\theta}$ for the set $\theta-\theta$. Note that each $\theta$, $\tilde{\theta}$ has an associated orientation $\xi$ so that $\tilde{\theta}(\xi)$ is contained in (and comparable to)
\begin{equation}\label{thetadef}  \{\w:|{\bf{c}}(\xi)\cdot\w|\le2,\quad |{\bf{n}}(\xi)\cdot\w|\le 2r^{-2},\quad |{\bf{t}}(\xi)\cdot\w|\le 2r^{-1}\} .
\end{equation}    
For each $\sigma_n$, ${\bf{CP}}_{\sigma_n}$ is the union of $2\sigma_n^2\times 2r^{-1}\sigma_n\times 2r^{-2}$ planks $\Theta$ centered at the origin and tangent to the light cone. Define $\Theta\in{\bf{CP}}_{\sigma_n}$ explicitly by taking orientations $\xi$ from a $r^{-1}$-separated set and letting
    \begin{equation}\label{Thetadef} \Theta=\Theta(\sigma_n,\xi)=\{\w:|{\bf{c}}(\xi)\cdot\w|\le\sigma_n^2,\quad |{\bf{n}}(\xi)\cdot\w|\le 2r^{-2},\quad |{\bf{t}}(\xi)\cdot\w|\le 2r^{-1}\sigma_n\} .\end{equation} 
For each $\sigma_n$, $\tau_n\in{\bf{S}}_{\sigma_n^{-1}r^{-1}}$, and each $\theta\in{\bf{S}}_{r^{-1}}$ with $\theta\subset\tau_n$, associate a single $\Theta_{\tau_n}\in{\bf{CP}}_{\sigma_n}$ so that $\Theta_{\tau_n}=\Theta_{\tau_n}(\sigma_n,\xi')$, $\theta=\theta(\xi)$, and $|\xi-\xi'|\le  \sigma_n^{-1}r^{-1}$. Let $\Omega_{\le \sigma_n}=\cup_{\Theta\in{\bf{CP}}_{\sigma_n}}\Theta$ and $\Omega_{\sigma_n}=\Omega_{\sigma_n}\setminus\Omega_{\le \sigma_n/2}$. Partition $\Omega:=\cup_{\theta\in{\bf{S}}_{r^{-1}}}\tilde{\theta}$ into sets $\Omega_{\le r^{-1}}\cup\big(\cup_{r^{-1}<\sigma_n\le 1}\Omega_{\sigma_n}\big)$.

The next lemma will be used to show that if $\|S_{B_{r_1}}f\|_{L^2_{avg}}\ge \a$, then $f_{\sigma_{N_0}}$ satisfies a similar property. 

\begin{lemma}\label{prune1} For each $r_1$-ball $B_{r_1}\subset\R^3$, 
\[|B_{r_1}|^{-1}\int\sum_{\theta\in{\bf{S}}_{r^{-1}}}|f_\theta-f_{\sigma_{N_0},\theta}|^2W_{B_{r_1}}\lesssim  \frac{1}{C_0}\a^2.  \]
\end{lemma}

\begin{proof} 
Using that $|\sum_{\substack{U\|U_{\tau_{N_0},r^2}\\U\not\in{\mc{G}}_{\tau_{N_0}}}}\s_U|\le 1$, we have  
\begin{align*}
\sum_\theta|f_\theta-f_{\sigma_{N_0},\theta}|^2(x)&=     \sum_{\theta\in{\bf{S}}_{r^{-1}}}|\sum_{\substack{U\|U_{\tau_{N_0},r^2}\\ U\not\in \mc{G}_{\tau_{N_0}}
}} \s_{U} f_\theta|^2 (x) \\
&\le  \sum_{\tau_{N_0}\in{\bf{S}}_{\sigma_{N_0}^{-1}r^{-1}}}\sum_{\substack{\theta\in{\bf{S}}_{r^{-1}}\\\theta\subset\tau_{N_0}}} \sum_{\substack{U\|U_{\tau_{N_0},r^2}\\ U\not\in \mc{G}_{\tau_{N_0}}
}} \s_{U} |f_\theta|^2 (x).
\end{align*}
Then using Definition \ref{weight}, we have
\begin{align*}
\int\sum_{\theta\in{\bf{S}}_{r^{-1}}}|f_\theta-f_{\sigma_{N_0},\theta}|^2W_{B_{r_1}} &\lesssim     \sum_{\tau_{N_0}\in{\bf{S}}_{\sigma_{N_0}^{-1}r^{-1}}}\sum_{\substack{\theta\in{\bf{S}}_{r^{-1}}\\\theta\subset\tau_{N_0}}} \sum_{\substack{U\|U_{\tau_{N_0},r^2}\\ U\not\in \mc{G}_{\tau_{N_0}}
}}   \int \s_{U}|f_\theta|^2 \widetilde{W}_{B_{r_1}} .
\end{align*}
Next, analyze each integral. By Plancherel's theorem, 
\[ \int \s_{U}|f_\theta|^2 \widetilde{W}_{B_{r_1}}=\int \widehat{\s}_{U}*\widehat{|f_\theta|^2} \overline{\widehat{\widetilde{W}}}_{B_{r_1}}. \]
The support of $\widehat{\s}_U*\widehat{|f_\theta|^2}$ is contained in $U^*+\tilde{\theta}$, which is contained in $10\tilde{\theta}$. The support of $\widehat{\widetilde{W}}_{B_{r_1}}$ is contained in a ball $B_{r_1^{-1}}$centered at the origin of radius $r_1^{-1}$. 

Recall that $\sigma_{N_0}= (r_1/r)^{-1}$, $\tau_{N_0}\in{\bf{S}}_{r_1/r^2}$, and each $\Theta(\sigma_{N_0},\xi)\in{\bf{CP}}_{\sigma_{N_0}}$ is defined by \eqref{Thetadef}. If $\Theta_{\tau_{N_0}}\in{\bf{CP}}_{\sigma_{N_0}}$ is associated to $\theta\subset\tau_{N_0}$, then it is clear from the definition of $\Theta_{\tau_{N_0}}$ that $10\tilde{\theta}\cap B_{r_1^{-1}}$ is contained in $C \Theta_{\tau_{N_0}}$ for an appropriate constant $C>0$. 
Let $\rho_{\Theta_{\tau_{N_0}}}$ be a smooth bump function equal to $1$ on $C\Theta_{\tau_{N_0}}$ and equal to $0$ outside of $2C\Theta_{\tau_{N_0}}$. Then 
\[ \int \widehat{\s}_{U}*\widehat{|f_\theta|^2} \overline{\widehat{\widetilde{W}}}_{B_{r_1}}=\int \widehat{\s}_{U}*\widehat{|f_\theta|^2}\rho_{\Theta_{\tau_{N_0}}} \overline{\widehat{\widetilde{W}}}_{B_{r_1}}=\int ({\s}_{U}{|f_\theta|^2})*\widecheck{\rho}_{\Theta_{\tau_{N_0}}}{{\widetilde{W}}}_{B_{r_1}} \]
and so
\begin{align} \label{B}\int\sum_{\theta\in{\bf{S}}_{r^{-1}}}|f_\theta-f_{\sigma_{N_0},\theta}|^2W_{B_{r_1}} &\lesssim \sum_{\tau_{N_0}\in{\bf{S}}_{\sigma_{N_0}^{-1}r^{-1}}}\sum_{\substack{\theta\in{\bf{S}}_{r^{-1}}\\\theta\subset\tau_{N_0}}} \sum_{\substack{U\|U_{\tau_{N_0},r^2}\\ U\not\in \mc{G}_{\tau_{N_0}}
}}    \int (\s_{U}|f_\theta|^2)*|\widecheck{\rho}_{\Theta_{\tau_{N_0}}}| W_{B_{r_1}} . 
\end{align} 
Next we analyze the integrand pointwise, without the $W_{B_{r_1}}$ function. Note that $\s_U(x-y)|\widecheck{\rho}_{\Theta_{\tau_{N_0}}}|(y)\lesssim |U|^{-1}W_U(x)W_U(x-y)$, so for each $x\in\R^3$,
\begin{align*}
\int \s_{U}(x-y)|f_\theta|^2(x-y)|\widecheck{\rho}_{\Theta_{\tau_{N_0}}}|(y)dy \lesssim   W_U(x)  |U|^{-1}\int |f_\theta|^2 W_U.  
\end{align*}
By the definition of $\mc{G}_{\tau_{N_0}}$, this means we may bound the right hand side of \eqref{B} by  
\[ \sum_{\tau_{N_0}\in{\bf{S}}_{\sigma_{N_0}^    {-1}r^{-1}}}                     \sum_{\substack{U\|U_{\tau_{N_0},r^2}\\ U\not\in \mc{G}_{\tau_{N_0}}
}}  \int W_U W_{B_{r_1}}\left( \frac{1}{C_0\log (r^2/r_1)}\frac{\a^2}{\#\tau_{N_0}}\right) \lesssim  |B_{r_1}| \frac{1}{C_0(\log (r^2/r_1))^3}\a^2 , \]
which proves the claim. 
\end{proof}


We will define a high/low frequency cutoff now related to the notation for the frequency decomposition in \cite{locsmooth}. 
\begin{definition}[Frequency cut-off] \label{freqcutoff}
For $\theta\in{\bf{S}}_{r^{-1}}$ and $\tau_n\in{\bf{S}}_{\sigma_n^{-1}r^{-1}}$, let $\Theta_{\tau_n}$ be the $\Theta\in{\bf{CP}}_{\sigma_n}$ that is associated to $\theta$. Then let $\eta_{\Theta_{\tau_n}}$ be a bump function equal to $1$ on $4\underline{C}(\log (r^2/r_1))\Theta_{\tau_n}$ (where $\underline{C}$ is the same as in property (1) from Lemma \ref{conepruning}) and equal to $0$ off of $8\underline{C}(\log (r^2/r_1))\Theta_{\tau_n}$. Also define $\eta_{\Theta_{\tau_n}^c}=\chi_{B_2}-\eta_{\Theta_{\tau_n}}$ where $\chi_{B_2}$ is a fixed smooth bump function equal to $1$ when $|\xi|\le 2\underline{C}(\log (r^2/r_1))$ and supported in $|\xi|\le 4\underline{C}(\log (r^2/r_1))$.
\end{definition}

\begin{lemma}[Weak high-dominance] \label{weakhi} For each $n$, $1\le n\le N_0-1$, if 
\begin{equation}\label{contra}
|B_{r_1}|\a^2\lesssim (\log(r^2/r_1))^2\int\sum_{\theta\in{\bf{S}}_{r^{-1}}}|f_{\sigma_n,\theta}^{\mc{B}}|^2 \widetilde{W}_{B_{r_1}} ,\end{equation}
then 
\[ \int\sum_{\theta\in{\bf{S}}_{r^{-1}}}|f_{\sigma_n,\theta}^{\mc{B}}|^2\widetilde{W}_{B_{r_1}}\lesssim  \Bigl|\sum_{\tau_n\in{\bf{S}}_{\sigma_n^{-1}r^{-1}}}\int \big(\sum_{\substack{\theta\in{\bf{S}}_{r^{-1}}\\\theta\subset\tau_n}}|f_{\sigma_n,\theta}^{\mc{B}}|^2*\widecheck{\eta}_{\Theta_{\tau_n}^c}\big)\widetilde{W}_{B_{r_1}}\Bigr| . \]
\end{lemma}

\begin{proof} Suppose for a contradiction that
\[ \int\sum_{\theta\in{\bf{S}}_{r^{-1}}}|f_{\sigma_n,\theta}^{\mc{B}}|^2 \widetilde{W}_{B_{r_1}}\le 2\Big|\int\big(\sum_{\theta\in{\bf{S}}_{r^{-1}}}|f_{\sigma_n,\theta}^{\mc{B}}|^2*\widecheck{\eta}_{\Theta_{\tau_n}}\big) \widetilde{W}_{B_{r_1}} \Big| .\]
Analyze the integrand pointwise, without the weight function $\widetilde{W}_{B_{r_1}}$. Using the definition of the pruning process, we have
\begin{align*}
\Bigl| \sum_{\theta\in{\bf{S}}_{r^{-1}}}|f_{\sigma_n,\theta}^{\mc{B}}|^2*\widecheck{\eta}_{\Theta_{\tau_n}}(x)\Bigr|&= \Bigl|\sum_{\tau_n\in{\bf{S}}_{\sigma_n^{-1}r^{-1}}}\sum_{\substack{\theta\in{\bf{S}}_{r^{-1}}\\\theta\subset\tau_n}}|\sum_{\substack{U\|U_{\tau_n,r^2}\\ U\not\in \mc{G}_{\tau_n}}} \s_{U} f_{\sigma_{n+1},\theta}|^2*\widecheck{\eta}_{\Theta_{\tau_n}}(x)   \Bigr| \\
&\le \sum_{\tau_n\in{\bf{S}}_{\sigma_n^{-1}r^{-1}}}\sum_{\substack{\theta\in{\bf{S}}_{r^{-1}}\\\theta\subset\tau_n}}\sum_{\substack{U\|U_{\tau_n,r^2}\\ U\not\in \mc{G}_{\tau_n}}} \int \s_U(y)|f_{\sigma_{n+1},\theta}|^2(y)|\widecheck{\eta}_{\Theta_{\tau_n}}|(x-y)dy. 
\end{align*}
Noting that $\s_U(y)|\widecheck{\eta}_{\Theta_{\tau_n}}|(x-y)\lesssim (\log (r^2/r_1))^3\s_U^{1/2}(x)|U|^{-1}{W}_U(y)$, we have
\begin{align*}
\sum_{\tau_n\in{\bf{S}}_{\sigma_n^{-1}r^{-1}}}\sum_{\substack{\theta\in{\bf{S}}_{r^{-1}}\\\theta\subset\tau_n}}&\sum_{\substack{U\|U_{\tau_n,r^2}\\ U\not\in \mc{G}_{\tau_n}}} \int \s_U(y)|f_{\sigma_{n+1},\theta}|^2(y)|\widecheck{\eta}_{\Theta_{\tau_n}}|(x-y)dy\\
&\lesssim (\log (r^2/r_1))^3\sum_{\tau_n\in{\bf{S}}_{\sigma_n^{-1}r^{-1}}}\sum_{\substack{U\|U_{\tau_n,r^2}\\ U\not\in \mc{G}_{\tau_n}}} \s_U^{1/2}(x)|U|^{-1}\int \sum_{\substack{\theta\in{\bf{S}}_{r^{-1}}\\\theta\subset\tau_n}} |f_{\sigma_{n+1},\theta}|^2|W_U \\
&\lesssim (\log (r^2/r_1))^3\sum_{\tau_n\in{\bf{S}}_{\sigma_n^{-1}r^{-1}}}\sum_{\substack{U\|U_{\tau_n,r^2}\\ U\not\in \mc{G}_{\tau_n}}} \s_U^{1/2}(x)\frac{1}{C_0(\log (r^2/r_1))^5}\frac{\a^2}{\#\tau_n}\lesssim \frac{\a^2}{C_0(\log (r^2/r_1))^2}. 
\end{align*}
This is a contradiction to \eqref{contra}, if $C_0$ is sufficiently large, which leads to the desired conclusion. 
\end{proof}

\begin{lemma}[Cone high lemma] \label{middom} For each $n$, $1\le n\le N_0-1$,
\[\int_{|\xi|\le 2r_1^{-1}}|\sum_{d(\theta)=r^{-1}}\widehat{|f_{\sigma_n,\theta}^{\mc{B}}|^2}{\eta}_{\Theta_{\tau_n}^c} |^2\lesssim (\log \big(\frac{r^2}{r_1}\big))^6\sum_{\sigma_n^{-1}r^{-1}\le s\le 1}\sum_{d(\tau)=s}\sum_{\substack{U\|U_{\tau,r^2}\\C_0(\log (\frac{r^2}{r_1}))^3\|S_Uf\|_{L^2_{avg}}^2\ge \frac{\a^2}{{\#\tau}}}}|U|\|S_Uf\|_{L^2_{avg}}^4.  \]
\end{lemma}

\begin{proof} As in the proof of Lemma 1.4 in  \cite{locsmooth}, we may write the Fourier support of the function $\sum_{\theta\in{\bf{S}}_{r^{-1}}}|f_\theta|^2$ as $\Omega=\big(\sqcup_{k=n+1}^N \Omega_{\sigma_k}\big)\sqcup \Omega_{\le\sigma_n}$. Since $\sum_{\theta\in{\bf{S}}_{r^{-1}}}|f_{\sigma_n,\theta}^{\mc{B}}|^2$ has Fourier support in the $\underline{C}\log (r^2/r_1)$ dilation of the Fourier support of $\sum_\theta|f_\theta|^2$ (by property (1) of Lemma \ref{conepruneprops}), its Fourier support is partitioned into $\underline{C}\log (r^2/r_1)$ dilations of the sets $\Omega_{\sigma_k}$ and $\Omega_{\le\sigma_n}$. 

Note that by the proof of Lemma 4.1 from \cite{locsmooth}, $\tilde{\theta}\cap\Omega_{\le\sigma_n}\subset 4 \Theta_{\tau_n}$ where  $\Theta_{\tau_n}\in{\bf{CP}}_{\sigma_n}$ associated to $\theta$. It follows that $\supp\widehat{|f_{\sigma_n,\theta}^{\mc{B}}|^2}\cap\underline{C}\log (r^2/r_1)\Omega_{\le\sigma_n}$ is contained in $4\underline{C}\log (r^2/r_1)\Theta_{\tau_n}$. Since $\supp|\widehat{f_{\sigma_n,\theta}^{\mc{B}}}|^2{\eta}_{\Theta_{\tau_n}^c}\subset \underline{C}\log (r^2/r_1)\tilde{\theta}\setminus 4\underline{C}\log (r^2/r_1)\Theta_{\tau_n}$, we must have 
\begin{equation}\label{spose}
\int_{|\xi|\le 2r_1^{-1}}|\sum_{\theta\in{\bf{S}}_{r^{-1}}}\widehat{|f_{\sigma_n,\theta}^{\mc{B}}|^2}{\eta}_{\Theta_{\tau_n}^c} |^2\lesssim \sum_{n<k\le N_0}\int_{\underline{C}(\log (r^2/r_1))\Omega_{\sigma_k}}|\sum_{\theta\in{\bf{S}}_{r^{-1}}}\widehat{|f_{\sigma_n,\theta}^{\mc{B}}|^2}{\eta}_{\Theta_{\tau_n}^c} |^2 . 
\end{equation}
Note that we also omitted the sets $\Omega_{\sigma_{N_0+1}},\ldots,\Omega_{\sigma_N}$ since they (and their dilations by $\underline{C}\log(r^2/r_1)$) do not intersect the ball $|\xi|\le 2r_1^{-1}$. Then by Lemmas 4.1 and 4.2 of \cite{locsmooth} and Cauchy-Schwarz, for each $k$, $n< k\le N_0$, 
\begin{align*}
\int_{\underline{C}(\log (r^2/r_1))\Omega_{\sigma_k}}|\sum_{\theta\in{\bf{S}}_{r^{-1}}}\widehat{|f_{\sigma_n,\theta}^{\mc{B}}|^2}{\eta}_{\Theta_{\tau_n}^c} |^2&\lesssim \log (r^2/r_1)\sum_{\tau_k\in{\bf{S}}_{\sigma_k^{-1}r^{-1}}}\int_{\underline{C}(\log (r^2/r_1))\Omega_{\sigma_k}}|\sum_{\substack{\theta\in{\bf{S}}_{r^{-1}}\\ \theta\subset\tau_k}} \widehat{|f_{\sigma_n,\theta}^{\mc{B}}|^2}\eta_{\Theta_{\tau_k}}{\eta}_{\Theta_{\tau_n}^c} |^2\\
&\lesssim \log (r^2/r_1)\sum_{\tau_k\in{\bf{S}}_{\sigma_k^{-1}r^{-1}}}\int |\sum_{\substack{\theta\in{\bf{S}}_{r^{-1}}\\ \theta\subset\tau_k}}  \widehat{|f_{\sigma_n,\theta}^{\mc{B}}|^2}{\eta}_{\Theta_{\tau_k}} |^2 . 
\end{align*}
It remains to bound each integral. We have
\begin{align*}
\int |\sum_{\substack{\theta\in{\bf{S}}_{r^{-1}}\\ \theta\subset\tau_k}}  \widehat{|f_{\sigma_n,\theta}^{\mc{B}}|^2}{\eta}_{\Theta_{\tau_k}} |^2&= \int |\sum_{\substack{\theta\in{\bf{S}}_{r^{-1}}\\ \theta\subset\tau_k}}|f_{\sigma_n,\theta}^{\mc{B}}|^2*\widecheck{\eta}_{\Theta_{\tau_k}} |^2 =\sum_{U\|U_{\tau_k,r^2}}  \int_U  |\sum_{\substack{\theta\in{\bf{S}}_{r^{-1}}\\ \theta\subset\tau_k}}|f_{\sigma_n,\theta}^{\mc{B}}|^2*\widecheck{\eta}_{\Theta_{\tau_k}} |^2  . 
\end{align*}
Note that  
\begin{align*}
\sup_{x\in U} \int \sum_{\substack{\theta\in{\bf{S}}_{r^{-1}}\\ \theta\subset\tau_k}}|f_{\sigma_n,\theta}^{\mc{B}}|^2(y)|\widecheck{\eta}_{\Theta_{\tau_k}} |(x-y)dy &\le      \int \sum_{\substack{\theta\in{\bf{S}}_{r^{-1}}\\ \theta\subset\tau_k}}|f_{\sigma_n,\theta}^{\mc{B}}|^2(y)  \sup_{x\in U}|\widecheck{\eta}_{\Theta_{\tau_k}} |(x-y)dy \\
&\lesssim (\log (r^2/r_1))^3|U|^{-1}\int \sum_{\substack{\theta\in{\bf{S}}_{r^{-1}}\\ \theta\subset\tau_k}}|f_{\sigma_n,\theta}^{\mc{B}}|^2(y) W_U(y)dy,
\end{align*}
so 
\begin{align*}
\sum_{U\|U_{\tau_k,r^2}}  \int_U  |\sum_{\substack{\theta\in{\bf{S}}_{r^{-1}}\\ \theta\subset\tau_k}}|f_{\sigma_n,\theta}^{\mc{B}}|^2*\widecheck{\eta}_{\Theta_{\tau_k}} |^2&\lesssim (\log (r^2/r_1))^6\sum_{U\|U_{\tau_k,r^2}}  |U|^{-1}  \Bigl(\int \sum_{\substack{\theta\in{\bf{S}}_{r^{-1}}\\ \theta\subset\tau_k}}|f_{\sigma_n,\theta}^{\mc{B}}|^2 W_U \Bigr)^2. 
\end{align*}
By property (2) of Lemma \ref{conepruneprops}, since $n<k\le N_0$,  $|f_{\sigma_n,\theta}^{\mc{B}}|\le |f_{\sigma_k,\theta}|\le |\sum_{U\in\mc{G}_{\tau_k}}\s_Uf_\theta|$. Using this and property (2) of Lemma \ref{conepruneprops}, it remains to observe
\begin{align}
\sum_{U\|U_{\tau_k,r^2}}  |U|^{-1}  \Bigl(\int \sum_{\substack{\theta\in{\bf{S}}_{r^{-1}}\\ \theta\subset\tau_k}}|f_{\sigma_n,\theta}^{\mc{B}}|^2 W_U \Bigr)^2&\le \sum_{U\|U_{\tau_k,r^2}}  |U|^{-1}  \Bigl(\int \sum_{\substack{\theta\in{\bf{S}}_{r^{-1}}\\ \theta\subset\tau_k}}|\sum_{U'\in\mc{G}_{\tau_k}}\s_{U'} f_{\theta}|^2 W_U \Bigr)^2    \label{beginninghere}\\
&\le \sum_{\substack{\substack{U\|U_{\tau_k,r^2}}}}|U|^{-1}\Bigl(\sum_{\substack{U'
    \|U_{\tau_k,r^2} \\ U'\in\mc{G}_{\tau_k}}}\|\s^{1/2}_{U'}W_U^{1/2}\|_\infty \int\sum_{\theta\subset\tau_k}|f_{\theta}|^2 \s_{U'}^{1/2}W_U^{1/2}\Bigr)^2 \nonumber\\
\text{(Cauchy-Schwarz)} \qquad    &\le|U|^{-1}\sum_{\substack{U\|U_{\tau_k,r^2}}} \Bigl( \sum_{\substack{U'
    \|U_{\tau_k,r^2}\\ U'\in\mc{G}_{\tau_k}}}\|\s_{U'}W_U\|_\infty\Bigr) \sum_{\substack{U'
    \|U_{\tau_k,r^2}\\ U'\in\mc{G}_{\tau_k}}}\Bigl(\int\sum_{\theta\subset\tau_k}|f_{\theta}|^2 \s_{U'}W_U\Bigr)^2 \nonumber\\
(\|\cdot\|_{\ell^2}\le \|\cdot\|_{\ell^1})\qquad     &\lesssim |U|^{-1}\sum_{\substack{U'
    \|U_{\tau_k,r^2}\\ U'\in\mc{G}_{\tau_k}}}\Bigl( \sum_{U\|U_{\tau_k,r^2}}\int\sum_{\theta\subset\tau_k}|f_{\theta}|^2 \s_{U'}^{1/2}W_U^{1/2}\Bigr)^2 \nonumber\\
&\lesssim \sum_{\substack{U'
    \|U_{\tau_k,r^2}\\ U'\in\mc{G}_{\tau_k}}}|U'|^{-1}\Bigl( \int\sum_{\theta\subset\tau_k}|f_{\theta}|^2 W_{U'}\Bigr)^2=\sum_{\substack{U'
    \|U_{\tau_k,r^2}\\ U'\in\mc{G}_{\tau_k}}}|U'|\|S_{U'}f\|_{L^2_{avg}}^4 \nonumber.
\end{align}

\end{proof}

\begin{lemma}[Cone low lemma]\label{lowdom}
\[\int_{|\xi|\le 2r_1^{-1}}|\sum_{\theta\in{\bf{S}}_{r^{-1}}}\widehat{|f_{\sigma_1,\theta}|^2} |^2\lesssim C(\log \big(\frac{r^2}{r_1}\big))^6\sum_{r^{-1}\le s\le 1}\sum_{d(\tau)=s}\sum_{\substack{U\|U_{\tau,r^2}\\C_0(\log (\frac{r^2}{r_1}))^3\|S_Uf\|_{L^2_{avg}}^2\ge \frac{\a^2}{{\#\tau}}}}|U|\|S_Uf\|_{L^2_{avg}}^4.  \]

\end{lemma}

\begin{proof} Use the decomposition of the Fourier support $\underline{C}\log (r^2/r_1)\Omega$ (where $\Omega$ is from the proof of Lemma 1.4 in \cite{locsmooth}):
\begin{align*} \int_{|\xi|\le \underline{C}\log(\frac{r^2}{r_1})}|\sum_{\theta\in{\bf{S}}_{r^{-1}}}\widehat{|f_{\sigma_1,\theta}|^2} |^2 &\le \sum_{k=2}^{N_0}\int_{\underline{C}\log (\frac{r^2}{r_1})\Omega_{\sigma_k}}|\sum_{\theta\in{\bf{S}}_{r^{-1}}}\widehat{|f_{\sigma_1,\theta}|^2} |^2  +\int_{\underline{C}\log (\frac{r^2}{r_1})\Omega_{\le \sigma_1}}|\sum_{\theta\in{\bf{S}}_{r^{-1}}}\widehat{|f_{\sigma_1,\theta}|^2} |^2. 
\end{align*} 
Note that the sets $\underline{C}\log (\frac{r^2}{r_1})\Omega_{\sigma_n}$ for $N_0< n\le N$ do not contain the ball $|\xi|\le \underline{C}\log(\frac{r^2}{r_1})$, which is why the sum on the right hand side above ends at $N_0$. For the terms corresponding to $k=2,\ldots,N_0$, repeat the argument beginning at \eqref{spose} in the proof of Lemma \ref{middom} to bound each integral. For the final integral, $\underline{C}(\log (\frac{r^2}{r_1})\Omega_{\le \sigma_1}$ is contained in a ball of radius $\sim\underline{C}(\log \frac{r^2}{r_1})$ centered at the origin. It follows that for a bump function $\eta_{\sigma_1}$ equal to $1$ on that ball,
\begin{align*}
\int_{\underline{C}(\log (\frac{r^2}{r_1}))\Omega_{\le \sigma_1}}|\sum_{\theta\in{\bf{S}}_{r^{-1}}}\widehat{|f_{\sigma_1,\theta}|^2} |^2&\le      \int|\sum_{\theta\in{\bf{S}}_{r^{-1}}}\widehat{|f_{\sigma_1,\theta}|^2} \eta_{\sigma_1}|^2\\
&\lesssim \int|\sum_{\theta\in{\bf{S}}_{r^{-1}}}{|f_{\sigma_1,\theta}|^2} *\widecheck{\eta}_{\sigma_1}|^2 =
\sum_{U\|U_{\tau_1,r^2}}\int_{U_{\tau_1,r^2}}|\sum_{\theta\in{\bf{S}}_{r^{-1}}}{|f_{\sigma_1,\theta}|^2} *\widecheck{\eta}_{\sigma_1}|^2
\end{align*}
where $\tau_1$ is just a $1\times 1\times 1$ plank. Then by the same reasoning as that beginning in \eqref{beginninghere} above, conclude that
\[ \sum_{U\|U_{\tau_1,r^2}}\int_{U_{\tau_1,r^2}}|\sum_{\theta\in{\bf{S}}_{r^{-1}}}{|f_{\sigma_1,\theta}|^2} *\widecheck{\eta}_{\sigma_1}|^2 \lesssim (\log (\frac{r^2}{r_1}))^6 \sum_{\substack{U\|U_{\tau_1,r^2}\\ U\in\mc{G}_{\tau_1}}} |U|\|S_Uf\|_{L^2_{avg}}^4.\]
which finishes the proof. 
\end{proof}

With the pruning process and related lemmas in hand, we may prove Proposition \ref{kakcone}. 
\begin{proof}[Proof of Proposition \ref{kakcone}] 
The indexing set for the sum on the left hand side of \eqref{ineq} is a finitely overlapping collection of $r_1$ balls $B_{r_1}\subset\R^3$ for which $\a^2\le \|S_{B_{r_1}}f\|_{L^2_{avg}}^2$. Now we do a standard local $L^2$ orthogonality argument to relate $\|S_{B_{r_1}}f\|_{L^2_{avg}}$ to an expression involving $\sum_{\theta\in{\bf{S}}_{r^{-1}}}|f_\theta|^2$. Using Definition \ref{weight}, we have 
\begin{align*}
\|S_{B_{r_1}}f\|_{L^2_{avg}}^2&\lesssim |B_{r_1}|^{-1}\int\sum_{d(\tau)=r^{-1/2}}|f_\tau|^2\widetilde{W}_{B_{r_1}}\\
&= |B_{r_1}|^{-1}\sum_{d(\tau)=r^{-1/2}} \sum_{\substack{\theta\in{\bf{S}}_{r^{-1}}\\\theta,\theta'\subset\tau}}\int f_\theta\overline{f}_{\theta'}\widetilde{W}_{B_{r_1}} \\
&= |B_{r_1}|^{-1}\sum_{d(\tau)=r^{-1/2}} \sum_{\substack{\theta\in{\bf{S}}_{r^{-1}}\\\theta,\theta'\subset\tau}}\int \widehat{f_\theta}\overline{\widehat{f_{\theta'}}*\widehat{\widetilde{W}}}_{B_{r_1}}. 
\end{align*} 
By Lemma The support of $\widehat{f_{\theta'}}*\widehat{\widetilde{W}}_{B_{r_1}}$ is contained in the $r_1^{-1}$-neighborhood of $\theta'$, which intersects $\theta$ only if $\theta$ and $\theta'$ are neighbors. The conclusion is that 
\[ \|S_{B_{r_1}}f\|_{L^2_{avg}}^2\lesssim |B_{r_1}|^{-1}\int\sum_{\theta\in{\bf{S}}_{r^{-1}}}|f_\theta|^2\widetilde{W}_{B_{r_1}}\sim \fint_{B_{r_1}}\sum_{\theta\in{\bf{S}}_{r^{-1}}}|f_\theta|^2 .  \]
Since we only consider balls $B_{r_1}$ on which
\[ \a^2\lesssim\fint \sum_{\theta\in{\bf{S}}_{r^{-1}}}|f_\theta|^2,\]
by Lemma \ref{prune1}, we also have
\[ \a^2\lesssim\fint \sum_{\theta\in{\bf{S}}_{r^{-1}}}|f_\theta|^2   \lesssim \fint \sum_{\theta\in{\bf{S}}_{r^{-1}}}|f_{\sigma_{N_0},\theta}|^2    .    \]

Next, we use property (3) of Lemma \ref{conepruneprops} to write $f_{\sigma_{N_0}}=f_{\sigma_{N_0-1},\theta}^{\mc{B}}+\cdots+f_{\sigma_1,\theta}^{\mc{B}}+f_{\sigma_1,\theta}$ and also 
\begin{align*}
    \a\lesssim\sum_{k=1}^{N_0-1} \left(|B_{r_1}|^{-1}\int \sum_{\theta\in{\bf{S}}_{r^{-1}}}|f_{\sigma_{k},\theta}^{\mc{B}}|^2\widetilde{W}_{B_{r_1}}\right)^{1/2}+\left(|B_{r_1}|^{-1}\int \sum_{\theta\in{\bf{S}}_{r^{-1}}}|f_{\sigma_{1},\theta}|^2\widetilde{W}_{B_{r_1}}\right)^{1/2}. 
\end{align*}
Since there are $N_0\lesssim \log (\frac{r^2}{r_1})$ many $\sigma_n$, we may assume that either 
\begin{equation}\label{case1} \sum_{\substack{B_{r_1}\subset\R^3\\\|S_{B_{r_1}}f\|_{L^2_{avg}}\ge \a}}|B_{r_1}|\|S_{B_{r_1}}f\|_{L^2_{avg}}^4\lesssim (\log \frac{r^2}{r_1})^4\sum_{\substack{B_{r_1}\subset\R^3\\ (\log \frac{r^2}{r_1})^2\fint_{B_{r_1}}\underset{\theta\in{\bf{S}}_{r^{-1}}}{\sum}|f_{\sigma_n,\theta}^{\mc{B}}|^2 \gtrsim \a^2}}|B_{r_1}|\Bigl(\fint_{B_{r_1}}\sum_{\theta\in{\bf{S}}_{r^{-1}}}|f_{\sigma_n,\theta}^{\mc{B}}|^2\Bigr)^2 \end{equation}
for some $1\le n\le N_0-1$ or 
\begin{equation}\label{case2} \sum_{\substack{B_{r_1}\subset\R^3\\\|S_{B_{r_1}}f\|_{L^2_{avg}}\ge \a}}|B_{r_1}|\|S_{B_{r_1}}f\|_{L^2_{avg}}^4\lesssim (\log \frac{r^2}{r_1})^4\sum_{\substack{B_{r_1}\subset\R^3\\(\log \frac{r^2}{r_1})^2\fint_{B_{r_1}}\underset{\theta\in{\bf{S}}_{r^{-1}}}{\sum}|f_{\sigma_1,\theta}|^2\gtrsim \a^2}}|B_{r_1}|\Bigl(\fint_{B_{r_1}}\sum_{\theta\in{\bf{S}}_{r^{-1}}}|f_{\sigma_1,\theta}|^2\Bigr)^2 . \end{equation} 
Suppose that we are in the latter case \eqref{case2}. Let $\widetilde{W}$ denote the weight function from Definition \ref{weight} associated to the $r_1$-ball centered at the origin. Note that for all $x\in B_{r_1}$, 
\[ \underset{\theta\in{\bf{S}}_{r^{-1}}}{\sum}|f_{\sigma_1,\theta}|^2*\widetilde{W}(x)\sim \int \underset{\theta\in{\bf{S}}_{r^{-1}}}{\sum}|f_{\sigma_1,\theta}|^2W_{B_{r_1}}. \]
Therefore, 
\begin{align*}
\sum_{\substack{B_{r_1}\subset\R^3\\(\log \frac{r^2}{r_1})^2\fint_{B_{r_1}}\underset{\theta\in{\bf{S}}_{r^{-1}}}{\sum}|f_{\sigma_1,\theta}|^2\gtrsim \a^2}}|B_{r_1}|\Bigl(\fint_{B_{r_1}}\sum_{\theta\in{\bf{S}}_{r^{-1}}}|f_{\sigma_1,\theta}|^2\Bigr)^2 &\lesssim |B_{r_1}|^{-2}\int|\sum_{\theta\in{\bf{S}}_{r^{-1}}}|f_{\sigma_1,\theta}|^2*\widetilde{W} |^2\\
&= |B_{r_1}|^{-2}\int|\sum_{\theta\in{\bf{S}}_{r^{-1}}}\widehat{|f_{\sigma_1,\theta}|^2}\widehat{\widetilde{W}} |^2.  
\end{align*}
Since $\widehat{\widetilde{W}}$ is supported in $|\xi|\le r_1^{-1}$ and we may assume that $|\widehat{\widetilde{W}}|\lesssim |B_{r_1}|$, we may bound the final expression above by
\[ \int_{|\xi|\le r_1^{-1}}|\sum_{\theta\in{\bf{S}}_{r^{-1}}}\widehat{|f_{\sigma_1,\theta}|^2} |^2 \]
and then use Lemma \ref{lowdom} to bound the right hand side. 

The remaining case is that for some $n$, $1\le n\le N_0-1$, \eqref{case1} holds. Then using Lemma \ref{weakhi}, for each $B_{r_1}$ in the indexing set on the right hand side of \eqref{case1}, we have
\[ \fint_{B_{r_1}}\sum_{\theta\in{\bf{S}}_{r^{-1}}}|f_{\sigma_n,\theta}^{\mc{B}}|^2\sim |B_{r_1}|^{-1}\int \sum_{\theta\in{\bf{S}}_{r^{-1}}}|f_{\sigma_n,\theta}^{\mc{B}}|^2\widetilde{W}_{B_{r_1}}\lesssim|B_{r_1}|^{-1}\left|\int \sum_{\theta\in{\bf{S}}_{r^{-1}}}|f_{\sigma_n,\theta}^{\mc{B}}|^2*\widecheck{\eta}_{\Theta_{\tau_n}^c}\widetilde{W}_{B_{r_1}} \right|. \]
Let $\eta_{r_1}$ be a bump function equal to $1$ on $|\xi|\le r_1^{-1}$ and supported in $|\xi|\le 2r_1^{-1}$. Then 
\[\int \sum_{\theta\in{\bf{S}}_{r^{-1}}}|f_{\sigma_n,\theta}^{\mc{B}}|^2*\widecheck{\eta}_{\Theta_{\tau_n}^c}\widetilde{W}_{B_{r_1}}=\int \sum_{\theta\in{\bf{S}}_{r^{-1}}}|f_{\sigma_n,\theta}^{\mc{B}}|^2*\widecheck{\eta}_{\Theta_{\tau_n}^c}*\widecheck{\eta}_{r_1}\widetilde{W}_{B_{r_1}} \]
and so using Cauchy-Schwarz, we may bound 
\begin{align*} 
\sum_{\substack{B_{r_1}\subset\R^3\\ (\log \frac{r^2}{r_1})^2\fint_{B_{r_1}}\underset{\theta\in{\bf{S}}_{r^{-1}}}{\sum}|f_{\sigma_n,\theta}^{\mc{B}}|^2 \gtrsim \a^2}}|B_{r_1}|&\Bigl(\fint_{B_{r_1}}\sum_{\theta\in{\bf{S}}_{r^{-1}}}|f_{\sigma_n,\theta}^{\mc{B}}|^2\Bigr)^2 \\ 
&\lesssim  \sum_{\substack{B_{r_1}\subset\R^3\\ (\log \frac{r^2}{r_1})^2\fint_{B_{r_1}}\underset{\theta\in{\bf{S}}_{r^{-1}}}{\sum}|f_{\sigma_n,\theta}^{\mc{B}}|^2 \gtrsim \a^2}}\int|\sum_{\theta\in{\bf{S}}_{r^{-1}}}|f_{\sigma_n,\theta}^{\mc{B}}|^2*\widecheck{\eta}_{\Theta_{\tau_n}^c}*\widecheck{\eta}_{r_1}|^2\widetilde{W}_{B_{r_1}} \\
    &\lesssim \int|\sum_{\theta\in{\bf{S}}_{r^{-1}}}|f_{\sigma_n,\theta}^{\mc{B}}|^2*\widecheck{\eta}_{\Theta_{\tau_n}^c}*\widecheck{\eta}_{r_1}|^2\lesssim \int_{|\xi|\le 2r_1^{-1}}|\sum_{\theta\in{\bf{S}}_{r^{-1}}}\widecheck{|f_{\sigma_n,\theta}^{\mc{B}}|^2}{\eta}_{\Theta_{\tau_n}^c}|^2. 
\end{align*} 
Invoking Lemma \ref{middom} finishes the proof.

\end{proof}

\section{Base case of the parabola (Proof of Lemma \ref{basecase}) \label{trunc}}

We record the Lorentz transformations recorded in \textsection 5 of \cite{locsmooth}. As discussed just after Lemma 3.1 in \cite{locsmooth}, it suffices to replace the truncated light cone $\Gamma$ with $\Gamma_{\frac{1}{K}}$ where $K$ will be a constant depending on $\e$ that we choose later. The truncated cone is
\[ \Gamma= \{(\xi_1,\xi_2,\xi_3):\xi_1^2+\xi_2^2=\xi_3^2,\quad \frac{1}{2}\le \xi_1^2+\xi_2^2\le 2 \}.   \]
Using the coordinates  $(\nu_1,\nu_2,\nu_3)$ defined by $\nu_2=\xi_1$, $\nu_1=\frac{\xi_3-\xi_2}{\sqrt{2}}$, $\nu_3=\frac{\xi_3+\xi_2}{\sqrt{2}}$, let 
\begin{align*} 
\Gamma_{\frac{1}{K}}&:= \Big\{(\nu_1,\nu_2,\nu_3):2\nu_1\nu_3=\nu_2^2,\qquad 1-\frac{1}{K}\le \nu_3\le 1,\quad\Big|\frac{\nu_2}{\nu_3}\Big|\le 1 \Big\}. 
\end{align*} 

Then planks $\tau$, $d(\tau)=s$ are approximately the convex hull of sets
\[\Lambda=\Lambda(\eta,s)=\Big\{(\nu_1,\nu_2,\nu_3)\in\Gamma_{\frac{1}{K}}:\Big|\frac{\nu_2}{\nu_3}-\eta\Big|<s\Big\} \]
for a real number $\eta=\eta(\tau)$, $|\eta|<1$ and $0<s<1$ satisfying $-1\le \eta\pm s\le 1$.

In their proof of decoupling for the cone \cite{BD}, Bourgain and Demeter showed that a piece of the cone like $\Gamma_{\frac{1}{K}}$ is contained in a certain neighborhood of a cylinder over a parabola, so that decoupling for the parabola can be used at a certain scale. We will use this idea to prove Lemma \ref{basecase}. 

In particular, note that the $\frac{1}{K}$-neighborhood of $\Gamma_{\frac{1}{K}}$, denoted $\mc{N}_{\frac{1}{K}}(\Gamma_{\frac{1}{K}})$ is approximately the $\frac{1}{K}$ neighborhood of 
\begin{equation}\label{cone} \Big\{(\nu_1,\nu_2,\nu_3): \nu_1=\frac{1}{2}\nu_2^2,\quad 1-\frac{1}{K}\le \nu_3\le 1,\quad \Big|\frac{\nu_2}{\nu_3}\Big|\le 1\Big\} \end{equation} 
since if $(\nu_1,\nu_2,\nu_3)\in\Gamma_{\frac{1}{K}}$, then $|\nu_1-\frac{1}{2}\nu_2^2|=|\frac{1}{2\nu_3}\nu_2^2-\frac{1}{2}\nu_2^2|\le K^{-1}$. The set \eqref{cone} is a portion of a cylinder over the parabola $\{(\frac{1}{2}\nu_2^2,\nu_2,0):|\nu_2|\le 1\}$. For any dyadic $R\ge 1$, the caps $\theta\in{\bf{S}}_{R^{-1/2}}$ defined in \eqref{theta} will now be $\theta\cap \{1-K^{-1}\le \nu_3\le 1\}$, which we will continue to denote by $\theta$. 

For $s$, $K^{-1/2}<s<1$, the convex hull of $\Lambda(\eta,s)$ is comparable to some cap $\tau\in{\bf{S}}_{s}$. The projection $\pi_3(\tau)=\{(\nu_1,\nu_2):\exists\nu_3\,\text{with}\,(\nu_1,\nu_2,\nu_3)\in\tau\}$ is equal to a cap $\w$ with $\ell(\w)=s$ of the $s^2$-neighborhood of the parabola. The \emph{conic} wave envelope $U_{\tau,R}$ associated to $\tau$ has dimensions $s^2R\times sR\times R$. We will use the fact that for the $sR\times R$ \emph{parabolic} wave envelope $V_{\w,R}$ associated to $\w$, $U_{\tau,R}$ is equivalent to $V_{\w,K}\times[-s^2R,s^2R]$, oriented according to the axes of $U_{\tau,R}$. 

Unwinding the definitions in Lemma \ref{basecase}, it suffices to prove the following proposition. 
\begin{prop}\label{basecaseprop} Let $K$ be a sufficiently large constant. For any $1\le r\le R\le K$, $R/r\ge K^{1/2}$, 
\begin{equation}\label{RHS}\sum_{\substack{B_r\subset\R^3\\\|S_{B_r}f\|_{L^2_{avg}}^2\ge \a^2 }}|B_r|\|S_{B_r}f\|_{L^2_{avg}}^4 \le C_\delta K^{\delta} \sum_{R^{-1/2}\le s\le 1}\sum_{\tau\in{\bf{S}}_s}\sum_{\substack{U \parallel U_{\tau,R}\\ C_\d K^\d\|S_Uf\|_{L^2_{avg}}^2\ge \frac{\a^2}{(\#\tau)^2}}}|U|\|S_Uf\|_{L^2_{avg}}^4  \end{equation} 
for any $\a>0$ and any Schwartz function $f:\R^3\to\C$ with $\supp\widehat{f}\subset\mc{N}_{\frac{1}{K}}(\Gamma_{\frac{1}{K}})$. 
\end{prop}
\begin{proof} We re-use notation from the proof of Proposition 6.6 in \cite{locsmooth}. Fix the notation $\tau$ for a plank in ${\bf{S}}_{r^{-1/2}}$. Also fix the notation $\theta$ for a plank in ${\bf{S}}_{R^{-1/2}}$. Let $A_1,\ldots,A_{1000}$ be disjoint collections of $\theta$, $f_j=\sum_{\theta\in  A_j}f_\theta$, $f_{j,\tau}=\sum_{\substack{\theta\subset\tau,\\ \theta\in A_j}}f_\theta$ so that distinct planks $\theta,\theta'\in\ A_j$ are $\ge 999R^{-1/2}$-separated. As in the proof of Proposition 6.6 in \cite{locsmooth}, for each $B_{r}\subset\R^3$, 
\begin{align*}
    |B_r|\|S_{B_r}f\|_{L^2_{avg}}^4&= |B_r|^{-1}\left(\int\sum_{\tau}|f_\tau|^2 W_{B_r}\right)^2\\
    &\lesssim \sum_{j=1}^{1000}|B_r|^{-1}\left(\int\sum_{\tau}|f_{j,\tau}|^2 W_{B_r}\right)^2 \\
    &=\sum_{j=1}^{1000}|B_r|^{-1}\left(\int|f_{j}|^2 W_{B_r}\right)^2 . 
\end{align*}
For each $B_r$ satisfying $\|S_{B_r}f\|_{L^2_{avg}}\ge \a$, there is some $j$, $1\le j\le 1000$, for which 
\begin{align} \label{C}
    \a^2&\lesssim |B_r|^{-1}\int|f_j|^2W_{B_r}
    .
\end{align}
Suppose $1\le j\le 1000$ satisfies 
\begin{align*}
\sum_{\substack{B_r\subset\R^3\\\|S_{B_r}f\|_{L^2_{avg}}\ge \a }}|B_r|\|S_{B_r}f\|_{L^2_{avg}}^4&\lesssim     \sum_{\substack{B_r\subset\R^3\\\fint_{B_r}|f_j|^2\gtrsim  \a^2 }}|B_r|^{-1}\left(\int |f_j|^2W_{B_r}\right)^2
\end{align*}
Let $U_\a=\{x\in \R^3:|f_j(x)|\ge c\a \}$ where $c>0$ will be chosen presently. If $\chi_{U_\a}$ is the characteristic function of $U_\a$, then for $c>0$ sufficiently small, 
\[ \a^2\le \|S_{B_r}f\|_{L^2_{avg}}^2\lesssim \fint_{B_r}|f_j|^2\lesssim \fint_{B_r}\chi_{U_\a}| f_j|^2. \] 
It follows using Cauchy-Schwarz that 
\begin{align*}
\sum_{\substack{B_r\subset\R^3\\\|S_{B_r}f\|_{L^2_{avg}}\ge \a }}|B_r|\|S_{B_r}f\|_{L^2_{avg}}^4&\lesssim \sum_{\substack{B_r\subset\R^3\\\fint_{B_r}|f_j|^2\ge \a^2 }}|B_r|^{-1}\left(\int\chi_{U_\a}|f_j|^2W_{B_r}\right)^2\\
    &\lesssim \sum_{\substack{B_r\subset\R^3\\\fint_{B_r}|f_j|^2\ge \a^2 }}\int\chi_{U_\a}|f_j|^4W_{B_r} \lesssim \a^{4}|U_\a|. 
\end{align*}
Next we argue that we can apply a cylindrical version of Theorem \ref{mainP} at scale $R$. Analyze each $f_{j,\theta}$ using the product structure of cylinders. Let $I_{K^{-1}}=[1-\frac{1}{K},1]$ and let $\pi_3(\theta)=\{(\nu_1,\nu_2):\exists\nu_3\quad\text{with}\quad(\nu_1,\nu_2,\nu_3)\in\theta\}$. Writing $\nu'=(\nu_1,\nu_2)$ and $x'=(x_1,x_2)$, we may express $f_{j,\theta}(x)$ by 
\begin{align*}
    f_{j,\theta}(x',x_3)=\int_{I_{K^{-1}}}\int_{\pi_3(\theta)}\widehat{f}_{j,\theta}(\nu',\nu_3)e^{2\pi ix'\cdot\xi'}e^{2\pi i x_3\xi_3}d\nu'd\nu_3 .
\end{align*}
For each $x_3\in\R^3$, the above formula defines a function on $\R^2$ with Fourier transform supported in $\mc{N}_{R^{-1}}(\P^1)$. Although it is not a Schwartz function, we may approximate $f_{j,\theta}$ arbitrarily closely with Schwartz functions which converge pointwise to $f_{j,\theta}$ by taking smooth cutoffs for $\theta$, which suffices for our purposes. With this view, we may apply Theorem \ref{mainP} to the function $\sum_{\theta\in A_j}f_{j,\theta}(x',x_3)$ in the coordinate $x'$, which we do as follows. By Fubini's theorem, for almost every $x_3$, $U_\a^{x_3}=\{x':(x',x_3)\in U_\a\}$ is a Lebesgue measurable subset of $\R^2$. By Theorem \ref{mainP}, we have
\begin{align*}
\a^4|U_\a^{x_3}|\le C_\delta R^\delta  \sum_{R^{-1/2}<s<1}\sum_{\ell(\w)=s} \sum_{V\in\mc{G}_{\w}}|V|\left(\fint_{V}\sum_{\pi_3(\theta)\subset\w}|f_{j,\theta}(x',x_3)|^2 dx'\right)^2   
\end{align*}
where $\ell(\w)=s$ indexes $s\times s^2$ caps $\w$ of $\mc{N}_{s^2}(\P^1)$, $V$ is an $sR\times R$ (parabolic) wave envelope associated to $\w$, and $V\in\mc{G}_{\w}$ if 
\[ C_\delta R^\delta |V|^{-1}\int \sum_{\pi_3(\theta)\subset\w}|f_{j,\theta}(x',x_3)|^2W_{V}(x')dx'\ge \frac{\a^2}{(\#\w)^2}.  \]
As in Definition \ref{weight}, we may replace $W_V$ by the weight $\widetilde{W}_V$ which has Fourier transform supported in an $(sR)^{-1}\times R^{-1}$ rectangle $V^*$ centered at the origin which is dual to $V$. Using Plancherel in $x'$, Fourier inversion in $x_3$, and the Fourier support of $|f_{j,\theta}|^2$ (which is $(\theta-\theta)\cap \{-K^{-1}\le\nu_3\le K^{-1}\}$), we have
\[\int \sum_{\pi_3(\theta)\subset\w}|f_{j,\theta}(x',x_3)|^2\widetilde{W}_{V}(x')dx'=\int_{[-K^{-1},K^{-1}]}e^{2\pi i x_3\nu_3}\int_{V^*} \sum_{\pi_3(\theta)\subset\w}\widehat{|f_{j,\theta}|^2}(\nu)\overline{\widehat{\widetilde{W}}_{V}}(\nu')d\nu'd\nu_3\]
where $\nu'=(\nu_1,\nu_2)$. Choose $\tau'\in{\bf{S}}_s$ so the set of $\theta$ with $\pi_3(\theta)\subset\w$ is equivalent to $\theta\subset\tau'$. The set $\{(\nu',\nu_3):\nu'\in V^*,\quad\nu_3\in[-K^{-1},K^{-1}]\}$ is contained in $10\Theta_{\tau'}$ (using the notation from Section \ref{kak}), which is dual to the conical envelope $U_{\tau',R}$ of dimensions $(s^2R)^{-1}\times (sR)^{-1}\times R^{-1}$. Therefore, \begin{align*}
    \int \sum_{\pi_3(\theta)\subset\w}|f_{j,\theta}(x',x_3)|^2\widetilde{W}_{V}(x')dx'&=\int_{[-K^{-1},K^{-1}]}e^{2\pi i x_3\nu_3}\int_{V^*} \sum_{\pi_3(\theta)\subset\w}\widehat{|f_{j,\theta}|^2}(\nu)\eta_{\Theta_{\tau'}}(\nu)\overline{\widehat{\widetilde{W}}_{V}}(\nu')d\nu'd\nu_3 \\
    &= \int\int \sum_{\pi_3(\theta)\subset\w}|f_{j,\theta}|^2(y)\widecheck{\eta}_{\Theta_{\tau'}}(x'-y'-z',x_3-y_3)\widetilde{W}_V(z')dz'dy
\end{align*}
where $\eta_{\Theta_{\tau'}}$ is a bump function equal to $1$ on $10\Theta_{\tau'}$ and equal to $0$ off $20\Theta_{\tau'}$. Noting that 
\[ |V|^{-1}\int |\widecheck{\eta}_{\Theta_{\tau'}}|(x'-y'-z',x_3-y_3)\widetilde{W}_V(z')dz'\lesssim W_{U_{\tau',R}}(x-y), \]
conclude that
\[ \b^4|U_\a^{x_3}|\le C_\delta R^\delta  \sum_{R^{-1/2}<s<1}\sum_{\ell(\w)=s} \sum_{V\in\mc{G}_{\w}}|V|\left(\fint_{U}\sum_{\pi_3(\theta)\subset\w}|f_{j,\theta}|^2 \right)^2   . \]
where $U$ is a translate of $U_{\tau',R}$ that intersects $V\times\{x_3\}$. Since there are $O(1)$ many choices of $U$ for each $x_3$ in an $s^2R$-interval $I_s$, we have
\begin{align*}
\int \a^4|U_\a^{x_3}|dx_3&\le C_\delta R^\delta  \sum_{R^{-1/2}<s<1}\sum_{\ell(\w)=s} \sum_{I_s}\sum_{V\in\mc{G}_{\w}}\int_{I_s}|V|\left(\fint_{U}\sum_{\pi_3(\theta)\subset\w}|f_{j,\theta}|^2 \right)^2 dx_3\\
&\lesssim 
C_\delta R^\delta  \sum_{R^{-1/2}<s<1}\sum_{\tau'\in{\bf{S}}_s} \sum_{\substack{U\|U_{\tau',R} \\ C_\delta R^\delta \|S_{U}f\|_{L^2_{avg}}^4\ge  \frac{\a^2}{(\#\tau)^2}}} |U|\left(\fint_{U}\sum_{\theta\subset\tau'}|f_{j,\theta}|^2 \right)^2 ,
\end{align*}
which finishes the proof. 
\end{proof}

\section{Lorentz rescaling (Proof of Lemma \ref{rescaling})\label{lorsec}}

Begin with a reduction which regularizes the Fourier support of $f$, analogous to the proof in \textsection\ref{mainthmpf}. 
\begin{lemma} \label{pig} Let $1>\d>0$ and $10\le r_1\le r_2\le r_3$. Write $D_\d(K,r_2)=(C\d \log K)^{\d^{-1}\log (K^2r_2^{1/2})/\log K}$ where $C$ is an absolute constant. For any $\a>0$ and Schwartz function $f:\R^3\to\C$ with Fourier transform supported in $\mc{N}_{r_3^{-1}}(\Gamma)$, there is a refined version $\tilde{f}$ of $f$ which satisfies 
\begin{align} 
\sum_{\substack{B_{r_1}\subset\R^3\\\|S_{B_{r_1}}f\|_{L^2_{avg}}}\ge \a}|B_{r_1}|\|S_{B_{r_1}}f\|_{L^2_{avg}}^4&\le D_\d(K,r_2)^4\sum_{\substack{B_{r_1}\subset\R^3\\D_\d(K,r_2)\|S_{B_{r_1}}\tilde f\|_{L^2_{avg}}}\ge  \a}|B_{r_1}|\|S_{B_{r_1}}\tilde{f}\|_{L^2_{avg}}^4, \\
    \text{for any }\tau\in{\bf{S}}_{\max(r_3^{-1/2},K^{-2}r_2^{-1/2})},\quad\text{either}&\quad\tilde{f}_{\tau}=f_\tau\quad\text{or}\quad\tilde{f}_\tau\equiv 0,
\end{align}
and for each $l,\quad \max(r_3^{-1/2},K^{-2}r_2^{-1/2})\le K^{-\d l}\le 1$, there exists $\lambda_l>0$ so that 
\begin{align}\label{density}
\#\{\tau_{K^{-\d l}}\in{\bf{S}}_{K^{-\d l}}:\tau_{K^{-\d l}}\subset \tau_{K^{-\d l+\d}},\quad \tilde{f}_{\tau_{K^{-\d l}}}\not\equiv 0\}\sim \lambda_l\quad\text{or}\quad 0 
\end{align}
for any $\tau_{K^{-\d l+\d}}\in{\bf{S}}_{K^{-\d l+\d}}$. 

\end{lemma}
\begin{proof} We write an abbreviated version of the proof in \textsection\ref{mainthmpf} adapted to this set-up. Let $N$ be defined by $K^{-\d N-\d}\le \max(r_3^{-1/2},K^{-2}r_2^{-1/2})\le K^{-\d N}$. Step $k$ of the algorithm takes as an input a collection $\Lambda_{k+1}(\lambda_{k+1})\subset{\bf{S}}_{K^{-\d (N-k)}}$, a function $f^{k+1}$ satisfying \eqref{density} for $l=k+1$, and an inequality 
\begin{align*} \sum_{\substack{B_{r_1}\subset\R^3\\\|S_{B_{r_1}}f\|_{L^2_{avg}}}\ge \a}|B_{r_1}|\|S_{B_{r_1}}f\|_{L^2_{avg}}^4&\le (C\d \log K)^{4(N-k)}\sum_{\substack{B_{r_1}\subset\R^3\\(C\d \log K)^{N-k}\|S_{B_{r_1}} f^{k+1}\|_{L^2_{avg}}}\ge  \a}|B_{r_1}|\|S_{B_{r_1}}{f^{k+1}}\|_{L^2_{avg}}^4.
\end{align*} 
To obtain $f^k$, begin by writing 
\[ f^k=\sum_{\tau_k\in{\bf{S}}_{K^{-k\d}}}f_{\tau_k}^{k+1}=\sum_{1\le \lambda\le K^{\d}}\sum_{\tau_k\in\Lambda_k(\lambda)}f_{\tau_k}^{k+1}\]
in which $\lambda$ is dyadic and $\Lambda_k(\lambda)=\{\tau_k\in{\bf{S}}_{K^{-k\e}}:\#\{\tau_{k+1}\in\Lambda_{k+1}(\lambda_{k+1}):f^{k+1}_{\tau_{k+1}}\not\equiv0,\,\tau_{k+1}\subset\tau_k\}\sim\lambda\}$. Since for each $B_{r_1}\subset\R^3$, $S_{B_{r_1}}f^{k+1}\le \sum_{\lambda}S_{B_{r_1}}(\sum_{\tau_{k}\in\Lambda_k(\lambda)}f_{\tau_{k}}^{k+1})$, by dyadic pigeonholing, there is some $\lambda_k$ for which 
\begin{align*} \sum_{\substack{B_{r_1}\subset\R^3\\(C\d \log K)^{N-k}\|S_{B_{r_1}} f^{k+1}\|_{L^2_{avg}}}\ge  \a}|B_{r_1}|\|S_{B_{r_1}}{f^{k+1}}\|_{L^2_{avg}}^4\le (C\d\log K)^4\sum_{\substack{B_{r_1}\subset\R^3\\(C\d \log K)^{N-k+1}\|S_{B_{r_1}} f^{k}\|_{L^2_{avg}}}\ge  \a}|B_{r_1}|\|S_{B_{r_1}}{f^{k}}\|_{L^2_{avg}}^4.
\end{align*} 
where $f^k=\sum_{\tau_k\in\Lambda_k(\lambda_k)}f^{k+1}_{\tau_k}$. Continue this process until we have found $f^1$, which we set equal to $\tilde{f}$. Note that the accumulated constant satisfies $(C\d \log K)^{4N}\le (C\d \log K)^{\d^{-1}\log (Kr_2^{1/2})/\log K}$.

\end{proof}

We re-state Lemma \ref{rescaling} for the reader's convenience. 

\begin{lemma}[Analogue of Lemma 3.3 of \cite{locsmooth}] For any $0<\d<1$ and $10\le r_1< r_2\le r_3$,
\[S_K(r_1,r_3)\le D_\d(K,r_2)^4K^{3\d}\log r_2\cdot S_K(r_1,r_2)\max_{r_2^{-1/2}\le s\le 1}S_K(s^2r_2,s^2r_3), \]
where $D_\d(K,r_2)$ is from Lemma \ref{pig}. \end{lemma}
\begin{proof} By Lemma \ref{pig}, it suffices to consider $f=\tilde{f}$ which satisfies the properties in that lemma. To bound $S_K(r_1,r_3)$, we need to show that for 
\[A=K^{10\d}\log r_2\cdot S_K(r_1,r_2)\max_{r_2^{-1/2}\le s\le 1}S_K(s^{2}r_2,s^{2}r_3) \]
and any $\a>0$, we have
\begin{align*} 
\sum_{\substack{B_{r_1}\subset\R^3\\\|S_{B_{r_1}}f\|_{L^2_{avg}}}\ge \a}|B_{r_1}|\|S_{B_{r_1}}f\|_{L^2_{avg}}^4&\le A  \sum_{r_3^{-1/2}\le s\le 1}\sum_{\tau\in{\bf{S}}_s}\sum_{\substack{U\parallel U_{\tau,r_3}\\ A\|S_Uf\|_{L^2_{avg}(U)}\ge \frac{\a}{\mu_f(s)}}}|U|\|S_Uf\|_{L^2_{avg}}^4. 
\end{align*} 
Use the definition of $S_K(r_1,r_2)$ to write
\[ \sum_{\substack{B_{r_1}\subset\R^3\\\|S_{B_{r_1}}f\|_{L^2_{avg}}}\ge \a}|B_{r_1}|\|S_{B_{r_1}}f\|_{L^2_{avg}}^4 \le S_K(r_1,r_2)\sum_{r_2^{-1/2}\le s\le 1}\sum_{\tau\in{\bf{S}}_s}\sum_{\substack{U_1 \parallel U_{\tau,r_2}\\S_K(r_1,r_2) \|S_{U_1}f\|_{L^2_{avg}}\ge \frac{\a}{\mu_f(s)}}}|U_1|\|S_{U_1}f\|_{L^2_{avg}}^4. \]
Next, we show that for each sector $\tau\in{\bf{S}}_s$, $r_2^{-1/2}\le s\le 1$,
\begin{align}
\label{aftlor}
\sum_{\substack{U_1 \parallel U_{\tau,r_2}\\S_K(r_1,r_2) \|S_{U_1}f\|_{L^2_{avg}}\ge \frac{\a}{\mu_f(s)}}}&|U_1|\|S_{U_1}f\|_{L^2_{avg}}^4\le S_K(s^2r_2,s^2r_3)\\
    &\times\sum_{r_3^{-1/2}\le s'\le s}\sum_{\substack{\tau'\subset\tau\\ \tau'\in{\bf{S}}_{s'}}}\sum_{\substack{U\parallel U_{\tau',r_3}\\ S_K(r_1,r_2)S_K(s^2r_2,s^2r_3)\|S_Uf\|_{L^2_{avg}}\ge \frac{\a}{\mu_f(s){\mu}_{f_{\tau}}(s')}}}|U|\|S_Uf\|_{L^2_{avg}}^4 \nonumber.  
\end{align} 
The factor $\log r_2$ appears in Lemma \ref{rescaling} since the number of $\tau\in{\bf{S}}_s$ which intersect a $\tau'\in{\bf{S}}_{s'}$ is bounded by $\log r_2$. Let $\b=[S_K(r_1,r_2)]^{-1}\frac{\a}{\mu_f(s)}$. By the definition of $S_K(s^2,r_2,s^2r_3)$, if $\widehat{h}$ is supported on $\mc{N}_{s^{-2}r_3^{-1}}(\Gamma_{\frac{1}{K}})$, then 
\begin{align*}
\sum_{\substack{B_{s^2r_2}\subset\R^3\\\|S_{B_{s^2r_2}}h\|_{L^2_{avg}}}\ge s^{-3}\b}|B_{s^2r_2}|&\|S_{B_{s^2r_2}}f\|_{L^2_{avg}}^4 \le S_K(s^2r_2,s^2r_3)\\
    &\times\sum_{s^{-1}r_3^{-1/2}\le s''\le 1}\sum_{\tau''\in{\bf{S}}_{s''}}\sum_{\substack{U'' \parallel U_{\tau'',s^2r_3}\\S_K(s^2r_2,s^2r_3)\|S_{U''}h\|_{L^2_{avg}(U'')}\ge \frac{s^{-3}\b}{\mu_{h}(s'')}}}|U''|\|S_{U''}h\|_{L^2_{avg}}^4.
\end{align*}
The inequality \eqref{aftlor} follows from the above defining inequality of $S(s^2r_2,s^2r_3)$ after a Lorentz transformation. Let $T$ be the Lorentz transformation from the proof of Lemma 3.3 in \cite{locsmooth}. Then we have the following relations, which are demonstrated thoroughly in \cite{locsmooth}:
\begin{enumerate}
    \item $T$ is a linear map taking $\tau$ to $[-1,1]^3$ (roughly), so $\det T\sim s^{-3}$.
    \item Define $\widehat{h}=\widehat{f_\tau}(T^{-1}(\cdot))$, so $h(x)=(\det T)f_\tau(T^*x)$. Also, $\widehat{h}_{\tau''}=\widehat{f}_{\tau'}(T^{-1}(\cdot))$ where $T(\tau')=\tau''$. 
    \item $U_1=T^*(B_{s^2r_2})$ and $\|S_{U_1}f\|_{L^2}^2=(\det T)^{-1}\|S_{B_{s^2r_2}}h\|_{L^2}^2$, so $\|S_{U_1}f\|_{L^2_{avg}}^2=(\det T)^{-2}\|S_{B_{s^2r_2}}h\|_{L^2_{avg}}^2$. 
    \item $U=T^*(U'')$ and $\|S_{U}f\|_{L^2}^2=(\det T)^{-1}\|S_{U''}h\|_{L^2}^2$, so $\|S_{U}f\|_{L^2_{avg}}^2=(\det T)^{-2}\|S_{U''}h\|_{L^2_{avg}}^2$.
    \item \[ \sum_{\substack{U_1\|U_{\tau,r_2}\\\|S_{U_1}f\|_{L^2_{avg}}\ge \b}}\longleftrightarrow \sum_{\substack{B_{s^2r_2}\subset\R^3\\\|S_{B_{s^2r_2}}h\|_{L^2_{avg}}\ge s^{-3}\b}} \]
    \item \begin{align*} \sum_{s^{-1}r_3^{-1/2}\le s''\le 1}\sum_{\tau''\in{\bf{S}}_{s''}}\sum_{\substack{U'' \parallel U_{\tau'',s^2r_3}\\S_K(s^2r_2,s^2r_3)\|S_{U''}h\|_{L^2_{avg}}\ge \frac{s^{-3}\b}{\mu_h(s'')}}} &\longleftrightarrow \sum_{r_3^{-1/2}\le s''s\le s}\sum_{\tau'\in{\bf{S}}_{s''s}}\sum_{\substack{U \parallel U_{\tau',r_3}\\S_K(s^2r_2,s^2r_3)\|S_{U}f\|_{L^2_{avg}}\ge \frac{\b}{\tilde{\mu}_{f_\tau}(s''s)}}} .\end{align*} 
    In the above relation, we relabel $\mu_h(s'')$ by $\tilde{\mu}_{f_\tau}(s''s)$. 

\end{enumerate}
We would be done if we could verify that for each $r_2^{-1/2}\le s\le 1$, $\tau\in{\bf{S}}_s$, and $r_3^{-1/2}\le s''s\le s$, 
\[ \mu_f(s)\tilde{\mu}_{f_\tau}(s''s)\lessapprox  \mu_f(s''s). \]
We unwind the definition of $\mu_f(s)$ and $\tilde{\mu}_{f_\tau}(s''s)$ using the two cases which go into both of their definitions. 

\noindent\fbox{$s> K^{-1}$ and $s''> K^{-1}$.} Then 
\[ \mu_f(s)\tilde{\mu}_{f_\tau}(s''s)=\#\{\tau_s\in{\bf{S}}_s:f_{\tau_s}\not\equiv 0\}\#\{\tau_{s''s}\in{\bf{S}}_{s''s}:f_{\tau_{s''s}}\not\equiv 0,\tau_{s''s}\subset\tau\}. \]
By Lemma \ref{pig}, since $\max(r_3^{-1/2},K^{-2}r_2^{-1/2})\le s''s$, this is bounded by 
\[ 2K^{3\d}\#\{\tau_{s''s}\in{\bf{S}}_{s''s}:f_{\tau_{s''s}}\not\equiv 0\}\le 2K^{2\d}\mu_f(s''s).  \]

\noindent\fbox{$s\le K^{-1}$ and $s''> K^{-1}$.} Then 
\[ \mu_f(s)\tilde{\mu}_{f_\tau}(s''s)=\mu_f(s)\#\{\tau_{s''s}\in{\bf{S}}_{s''s}:f_{\tau_{s''s}}\not\equiv 0,\tau_{s''s}\subset\tau\}. \]
Let $s\le s_1\le\cdots s_k\le K^{-1}$ be an admissible partition from the definition of $\mu_f(s)$. By Lemma \ref{pig}, since $\max(r_3^{-1/2},K^{-2}r_2^{-1/2})\le s''s$ and $r_2^{-1/2}\le s_1\le s_2$, we have that 
\[ \max_{\tau_{s_2}\in{\bf{S}}_{s_2}}\{\tau_{s_1}\in{\bf{S}}_{s_1}:\tau_{s_1}\subset\tau_{s_2},\quad f_{\tau_{s_1}}\not\equiv 0\}\#\{\tau_{s''s}\in{\bf{S}}_{s''s}:f_{\tau_{s''s}}\not\equiv 0,\tau_{s''s}\subset\tau\}\]
is bounded by 
\[ 2K^{3\d} \max_{\tau_{s_2}\in{\bf{S}}_{s_2}}\{\tau_{ s''s}\in{\bf{S}}_{ s''s}:\tau_{ s''s}\subset\tau_{s_2},\quad f_{\tau_{ s''s}}\not\equiv 0\} .\]
Therefore, we have taken an admissible partition $s\le s_1<\cdots< s_k\le K^{-1}$ from the definition of $\mu_f(s)$ and transferred it to an admissible partition $s''s\le \tilde{s}_1<\cdots< s_k\le K^{-1}$, with $\tilde{s}_1=s''s$, from the definition of $\mu_f(s''s)$. Thus $\mu_f(s)\tilde{\mu}_{f_\tau}(s''s)\le 2K^{3\d}\mu_f(s''s)$. 

\noindent\fbox{$s> K^{-1}$ and $s''\le K^{-1}$.} In this case, 
\[ \mu_f(s)\tilde{\mu}_{f_\tau}(s''s)=\#\{\tau_s\in{\bf{S}}_s:f_{\tau_s}\not\equiv 0\}\tilde{\mu}_{f_\tau}(s''s). \]
Let $s''s\le s_1s<\cdots< s_ks\le K^{-1}s$ be an admissible partition from the definition of $\tilde{\mu}_{f_\tau}(s''s)$. It is no loss of generality (by possibly adding $s_{k+1}=K^{-1}s$ to the partition) to assume that $K^{-2}s\le s_ks\le K^{-1}s$, which we do now. Then $\max(r_3^{-1/2},K^{-2}r_2^{-1/2})\le s_ks$, so by Lemma \ref{pig}, 
\[  \#\{\tau_s\in{\bf{S}}_s:f_{\tau_s}\not\equiv 0\}\#\{\tau_{s_{k}s}\in{\bf{S}}_{\tau_{s_{k}s}}:f_{\tau_{s_{k}s}}\not\equiv 0,\quad \tau_{s_{k}s}\subset\tau\}\le 2K^{3\d}\#\{\tau_{s_ks}\in{\bf{S}}_{s_ks}:f_{\tau_{s_ks}}\not\equiv 0\}.  \]
Therefore, for each admissible partition $s''s\le s_1s<\cdots<s_ks\le K^{-1}s$ used to define the expression $\mu_f(s)\tilde{\mu}_{f_\tau}(s''s)$, we have obtained an expression from the definition of $\mu_f(s''s)$, which concludes this case. 

\noindent\fbox{$s\le K^{-1}$ and $s''\le K^{-1}$.} In this final case, $\mu_f(s)$ and $\tilde{\mu}_{f_\tau}(s''s)$ are both determined by partitions. Since the partitions in the definition of $\tilde{\mu}_{f_\tau}(s''s)$ are bounded above by $K^{-1}s$, the partition formed by concatenating partitions from $\mu_f(s)$ and $\tilde{\mu}_{f_\tau}(s''s)$ are admissible for $\mu_f(s''s)$, which concludes the argument. 
\end{proof}

\section{Small cap decoupling for $\Gamma$ from Theorem \ref{mainC} \label{conesmallcap}}

Let $\b_1\in[\frac{1}{2},1]$, $\b_2\in[0,1]$, and let $\g$ be approximate $R^{-\b_1}\times R^{-\b_2}\times  R^{-1}$ boxes tiling $\Gamma$ and which subdivide canonical $\theta\in{\bf{S}}_{R^{-1/2}}$.


Recall the statement of Theorem \ref{smallcapthm}:
\begin{theorem*}
Let $\b_1\in[\frac{1}{2},1]$ and $\b_2\in[0,1]$. For any $p\ge 2$, $q\ge 1$, 
\begin{equation}\label{smallcaplpLp} \int_{B_R}|\sum_\g f_\g|^p\le C_\e R^\e\Big[R^{(\b_1+\b_2)(\frac{p}{2}-1)}+R^{(\b_1+\b_2)(p-2)-1}+R^{(\b_1+\b_2-\frac{1}{2})(p-2)}\Big]\sum_\g\|f_\g\|_{L^p(\R^2)}^p .\end{equation}
\end{theorem*}

\begin{proof}[Proof of Theorem \ref{smallcapthm}]
By an analogous explanation as in the proof of Theorem \ref{smallcapthmB}, it suffices to prove \[\a^p|U_\a|\le C_\e R^\e[R^{(\b_1+\b_2)(\frac{p}{2}-1)}+R^{(\b_1+\b_2)(p-2)-1}+R^{(\b_1+\b_2-\frac{1}{2})(p-2)}]\sum_\g\|f_\g\|_p^p \]
for Schwartz functions $f=\sum_\g f_\g$ satisfying the additional hypotheses that $\|f_\g\|_\infty\sim 1$ or $f_\g=0$ for all $\g$ and $\|f_\g\|_p^p\sim_p\|f_\g\|_2^2$ for all $\g$ and all $p\ge 2$. 

By Theorem \ref{mainC}, there is a scale $\sigma$, $R^{-1/2}\le \sigma\le 1$, for which 
\[ \a^{4}|U_\a|\le C_\e R^\e \sum_{\tau_\sigma}\sum_{U\in\mc{G}_{\tau_\sigma}(\a)}|U|^{-1}\|S_Uf\|_{2}^4. \]
Here $\tau_\sigma\in{\bf{S}}_{\sigma^{-1}R^{-1/2}}$ and $S_{\tau_{\sigma},R}f(x)^2=\sum_{\theta\subset\tau_\sigma}|f_\theta|^2(x) W_U(x)$ where $W_{U}$ is an $L^\infty$ normalized weight function localized to $U\|U_{\tau_\sigma,R}$, which is an $\sigma^{-1}R^{\frac{1}{2}}\times \sigma^{-2}\times R$ wave envelope. 

Let $\g_\sigma$ denote  $\max(R^{-\b_2},\sigma^2)\times\max(R^{-\b_1},\sigma R^{-\frac{1}{2}})\times R^{-1}$-small caps, so either $\g=\g_\sigma$ or $\g\subset\g_\sigma$. From the proof of Lemma 1.4 in \cite{locsmooth}, the dual set $U^*$ may be denoted by $\Theta\in{\bf{CP}}_{\sigma}$, using their notation. $\Theta$ is either comparable to or contained in a small cap plank $\g_\sigma\subset\tau_\sigma$. Since the set of $\g_\sigma'+\g_\sigma$, varying over $\g_\sigma'\subset\tau_\sigma$, are finitely overlapping, we may use local $L^2$ orthogonality. By local $L^2$-orthogonality, we have the inequality
\[ \|S_Uf\|_2^2=\int\sum_{\theta\subset\tau_\sigma}|f_\theta|^2W_U\lesssim \int\sum_{\g_\sigma\subset\tau_\sigma}|f_{\g_\sigma}|^2W_U. \]
Combined with the defining property of $U\in\mc{G}_{\tau_\sigma}(\a)$, this implies the upper bound
\[ \a^2\le C_\e R^\e (\#\tau)^2|U|^{-1}\|S_Uf\|_2^2\lesssim C_\e R^\e (\sigma R^{\frac{1}{2}})^2|U|^{-1}\|\sum_{\g_\sigma\subset\tau_\sigma}|f_{\g_\sigma}|^2W_U\|_2^2.\]
Therefore, for any $q\ge 4$, 
\[ \a^{q}|U_\a|\le C_\e R^\e\sum_{\tau_\sigma}\sum_{U\in\mc{G}_{\tau_\sigma}(\a)}(\sigma R^{\frac{1}{2}})^{q-4}|U|^{1-\frac{q}{2}}\Big(\int\sum_{\g_\sigma\subset\tau_\sigma}|f_{\g_\sigma}|^2W_U\Big)^{\frac{q}{2}} . \]
Use the assumption $\|f_\g\|_\infty\lesssim 1$ to bound
\begin{align*}
    |U|^{1-\frac{q}{2}}\Big(\int\sum_{\g_\sigma\subset\tau_\sigma}|f_{\g_\sigma}|^2W_U\Big)^{\frac{q}{2}-1}&\lesssim \|\int\sum_{\g_\sigma\subset\tau_\sigma}|f_{\g_\sigma}|^2\|_\infty^{\frac{q}{2}-1}\\
    &\lesssim \#\g\subset\tau_\sigma \#\g\subset\g_\sigma\lesssim \sigma^{-1}R^{\b_1+\b_2-\frac{1}{2}}\max(1,R^{\b_1}\sigma R^{-\frac{1}{2}},R^{\b_2}\sigma^2,\sigma^3 R^{\b_1+\b_2-\frac{1}{2}}). 
\end{align*}
This simplifies the inequality to 
\[\a^{q}|U_\a|\le C_\e R^\e\sum_{\tau_\sigma}\sum_{U\in\mc{G}_{\tau_\sigma}(\a)}(\sigma R^{\frac{1}{2}})^{q-4}\int\sum_{\g_\sigma\subset\tau_\sigma}|f_{\g_\sigma}|^2W_U\sigma^{-1}R^{\b_1+\b_2-\frac{1}{2}}\max(1,R^{\b_1}\sigma R^{-\frac{1}{2}},R^{\b_2}\sigma^2,\sigma^3 R^{\b_1+\b_2-\frac{1}{2}})  \]
We are done using special properties of the wave envelopes $U$, so within the displayed expression, we bound
\[\sum_{\tau_\sigma}\sum_{U\in\mc{G}_{\tau_\sigma}(\a)}\int\sum_{\g_\sigma\subset\tau_\sigma}|f_{\g_\sigma}|^2W_U\lesssim \sum_{\tau_\sigma}\int\sum_{\g_\sigma\subset\tau_\sigma}|f_{\g_\sigma}|^2\lesssim\sum_{\g}\|f_\g\|_2^2 \]
by $L^2$ orthogonality. This yields 
\begin{align}
\a^{q} |U_\a|&\le C_\e R^\e(\sigma R^{\frac{1}{2}})^{q-4}(\sigma^{-1}R^{\b_1+\b_2-\frac{1}{2}}\max(1,R^{\b_1}\sigma R^{-\frac{1}{2}},R^{\b_2}\sigma^2,\sigma^3 R^{\b_1+\b_2-\frac{1}{2}}))^{\frac{q}{2}-1}\sum_\g\|f_\g\|_2^2 \label{fromhereC}  .\end{align}

The argument divides into four ranges of $p$ and $\b_1,\b_2$, and for each of these ranges, we subdivide into four different cases depending on the maximum on the right hand side.

\noindent\fbox{Case 1: $2\le p\le 4$ and $\b_1+\b_2\ge 1$. } Begin by noting that by $L^2$ orthogonality,
\[ \a^2|U_\a|\lesssim \sum_\g\|f_\g\|_2^2. \]
From here, we would be done if $R^{(\b_1+\b_2-\frac{1}{2})}\ge\a$, so assume that $\a\ge R^{(\b_1+\b_2-\frac{1}{2})}$. Then using \eqref{fromhereC} with $q=4$, we have
\[ \a^p|U_\a|\le C_\e R^\e \frac{1}{\a^{4-p}}\sigma^{-1}R^{\b_1+\b_2-\frac{1}{2}}\max(1,R^{\b_1}\sigma R^{-\frac{1}{2}},R^{\b_2}\sigma^2,\sigma^3 R^{\b_1+\b_2-\frac{1}{2}})\sum_\g\|f_\g\|_2^2 .\]
It suffices to verify that
\[ \sigma^{-1}R^{\b_1+\b_2-\frac{1}{2}}\max(1,R^{\b_1}\sigma R^{-\frac{1}{2}},R^{\b_2}\sigma^2,\sigma^3 R^{\b_1+\b_2-\frac{1}{2}})\overset{?}{\le}R^{2(\b_1+\b_2-\frac{1}{2})}, \]
which we do presently in the following four cases. 
\begin{enumerate}[label=(\alph*)]
    \item $1=\max(1,R^{\b_1}\sigma R^{-\frac{1}{2}},R^{\b_2}\sigma^2,\sigma^3 R^{\b_1+\b_2-\frac{1}{2}})$: 
    \begin{align*}
    \sigma^{-1}R^{\b_1+\b_2-\frac{1}{2}}&\overset{?}{\le}R^{2(\b_1+\b_2-\frac{1}{2})} \qquad\impliedby \qquad R^{\frac{1}{2}}\le R^{\b_1+\b_2-\frac{1}{2}}\qquad(true). 
    \end{align*}
    \item $R^{\b_1-\frac{1}{2}}\sigma=\max(1,R^{\b_1}\sigma R^{-\frac{1}{2}},R^{\b_2}\sigma^2,\sigma^3 R^{\b_1+\b_2-\frac{1}{2}})$: clearly $R^{\b_1+\b_2-\frac{1}{2}}R^{\b_1-\frac{1}{2}}\le R^{2(\b_1+\b_2-\frac{1}{2})}$. 
    \item $R^{\b_2}\sigma^2=\max(1,R^{\b_1}\sigma R^{-\frac{1}{2}},R^{\b_2}\sigma^2,\sigma^3 R^{\b_1+\b_2-\frac{1}{2}})$:
    \begin{align*}
\sigma^{-1}R^{\b_1+\b_2-\frac{1}{2}}R^{\b_2}\sigma^2&\overset{?}{\le}R^{2(\b_1+\b_2-\frac{1}{2})} \\
\sigma &\overset{?}{\le}R^{\b_1-\frac{1}{2}}.
     \end{align*}
The final line is true since in this case, $\sigma\le R^{\frac{1}{2}-\b_1}$ and $R^{\frac{1}{2}-\b_1}$ is always bounded by $R^{\b_1-\frac{1}{2}}$.     
\item $\sigma^3R^{\b_1+\b_2-\frac{1}{2}}=\max(1,R^{\b_1}\sigma R^{-\frac{1}{2}},R^{\b_2}\sigma^2,\sigma^3 R^{\b_1+\b_2-\frac{1}{2}})$: clearly, $\sigma^{-1}R^{\b_1+\b_2-\frac{1}{2}}\sigma^3 R^{\b_1+\b_2-\frac{1}{2}}\overset{?}{\le}R^{2(\b_1+\b_2-\frac{1}{2})}$. 
\end{enumerate}

\noindent\fbox{Case 2: $2\le p\le 4$ and $\b_1+\b_2<1$. } By $L^2$ orthogonality, $\a^2|U_\a|\lesssim \sum_\g\|f_\g\|_2^2$. We would be done if $R^{(\b_1+\b_2)\frac{1}{2}}\ge\a$, so assume that $\a\ge R^{(\b_1+\b_2)\frac{1}{2}}$. Using \eqref{fromhereC} with $q=4$, it suffices to verify that
\[ \sigma^{-1}R^{\b_1+\b_2-\frac{1}{2}}\max(1,R^{\b_1}\sigma R^{-\frac{1}{2}},R^{\b_2}\sigma^2,\sigma^3 R^{\b_1+\b_2-\frac{1}{2}})\overset{?}{\le}R^{\b_1+\b_2}, \]
which we do presently in the following four cases. 
\begin{enumerate}[label=(\alph*)]
    \item $1=\max(1,R^{\b_1}\sigma R^{-\frac{1}{2}},R^{\b_2}\sigma^2,\sigma^3 R^{\b_1+\b_2-\frac{1}{2}})$: clearly $\sigma^{-1}R^{\b_1+\b_2-\frac{1}{2}}{\le}R^{\b_1+\b_2}$. 
    \item $R^{\b_1-\frac{1}{2}}\sigma=\max(1,R^{\b_1}\sigma R^{-\frac{1}{2}},R^{\b_2}\sigma^2,\sigma^3 R^{\b_1+\b_2-\frac{1}{2}})$: clearly $R^{\b_1+\b_2-\frac{1}{2}}R^{\b_1-\frac{1}{2}}\le R^{\b_1+\b_2}$. 
    \item $R^{\b_2}\sigma^2=\max(1,R^{\b_1}\sigma R^{-\frac{1}{2}},R^{\b_2}\sigma^2,\sigma^3 R^{\b_1+\b_2-\frac{1}{2}})$: 
    \[ \sigma^{-1}R^{\b_1+\b_2-\frac{1}{2}}R^{\b_2}\sigma^2{\le}R^{\b_1+\b_2}\qquad\impliedby\qquad \sigma{\le}R^{\frac{1}{2}-\b_2}  \]
    and $1\le R^{\frac{1}{2}-\b_2}$ since $\frac{1}{2}+\b_2\le \b_1+\b_2< 1$. 
    \item $\sigma^3R^{\b_1+\b_2-\frac{1}{2}}=\max(1,R^{\b_1}\sigma R^{-\frac{1}{2}},R^{\b_2}\sigma^2,\sigma^3 R^{\b_1+\b_2-\frac{1}{2}})$:  $\sigma^{-1}R^{\b_1+\b_2-\frac{1}{2}}\sigma^3 R^{\b_1+\b_2-\frac{1}{2}}\overset{?}{\le}R^{\b_1+\b_2}$ follows from noticing that $\sigma^{2}\le1\le R^{1-\b_1-\b_2}$.
\end{enumerate}

\noindent\fbox{Case 3: $4\le p\le 6$. }
Using \eqref{fromhereC} with $q=p$, we have
\begin{align*}
\a^{p} |U_\a|&(\sigma R^{\frac{1}{2}})^{p-4}(\sigma^{-1}R^{\b_1+\b_2-\frac{1}{2}}\max(1,R^{\b_1}\sigma R^{-\frac{1}{2}},R^{\b_2}\sigma^2,\sigma^3 R^{\b_1+\b_2-\frac{1}{2}}))^{\frac{p}{2}-1}\sum_\g\|f_\g\|_2^2 .\end{align*}
It suffices to verify that 
\[(\sigma R^{\frac{1}{2}})^{p-4}(\sigma^{-1}R^{\b_1+\b_2-\frac{1}{2}}\max(1,R^{\b_1}\sigma R^{-\frac{1}{2}},R^{\b_2}\sigma^2,\sigma^3 R^{\b_1+\b_2-\frac{1}{2}}))^{\frac{p}{2}-1}\le R^{(\b_1+\b_2)(\frac{p}{2}-1)}+R^{(\b_1+\b_2)(p-2)-1}. \]
\begin{enumerate}[label=(\alph*)]
    \item $1=\max(1,R^{\b_1}\sigma R^{-\frac{1}{2}},R^{\b_2}\sigma^2,\sigma^3 R^{\b_1+\b_2-\frac{1}{2}})$: 
    \begin{align*} 
    (\sigma R^{\frac{1}{2}})^{p-4}(\sigma^{-1}R^{\b_1+\b_2-\frac{1}{2}})^{\frac{p}{2}-1}&\overset{?}{\le}R^{(\b_1+\b_2)(\frac{p}{2}-1)} \\
    (\sigma^{-1})^{3-\frac{p}{2}}&{\le}(R^{\frac{1}{2}})^{3-\frac{p}{2}}\qquad(true) .
    \end{align*}  
    \item $R^{\b_1-\frac{1}{2}}\sigma=\max(1,R^{\b_1}\sigma R^{-\frac{1}{2}},R^{\b_2}\sigma^2,\sigma^3 R^{\b_1+\b_2-\frac{1}{2}})$: 
    \begin{align*} 
    (\sigma R^{\frac{1}{2}})^{p-4}(\sigma^{-1}R^{\b_1+\b_2-\frac{1}{2}}R^{\b_1}\sigma R^{-\frac{1}{2}})^{\frac{p}{2}-1}&\overset{?}{\le}R^{(\b_1+\b_2)(p-2)-1}\\
    \sigma^{p-4}R^{\b_1(\frac{p}{2}-1)}&{\le}R^{(\b_1+\b_2)(\frac{p}{2}-1)}\qquad(true).
    \end{align*}  
    \item $R^{\b_2}\sigma^2=\max(1,R^{\b_1}\sigma R^{-\frac{1}{2}},R^{\b_2}\sigma^2,\sigma^3 R^{\b_1+\b_2-\frac{1}{2}})$: 
    \begin{align*} 
    (\sigma R^{\frac{1}{2}})^{p-4}(\sigma^{-1}R^{\b_1+\b_2-\frac{1}{2}}R^{\b_2}\sigma^2)^{\frac{p}{2}-1}&\overset{?}{\le}R^{(\b_1+\b_2)(p-2)-1}\\
    (\sigma^3)^{\frac{p}{2}-1}&\overset{?}{\le}R^{(\b_1-\frac{1}{2})(\frac{p}{2}-1)}\qquad(true).
    \end{align*}  
    \item $\sigma^3R^{\b_1+\b_2-\frac{1}{2}}=\max(1,R^{\b_1}\sigma R^{-\frac{1}{2}},R^{\b_2}\sigma^2,\sigma^3 R^{\b_1+\b_2-\frac{1}{2}})$:  
    \begin{align*} 
    (\sigma R^{\frac{1}{2}})^{p-4}(\sigma^{-1}R^{\b_1+\b_2-\frac{1}{2}}\sigma^3R^{\b_1+\b_2-\frac{1}{2}})^{\frac{p}{2}-1}&\overset{?}{\le}R^{(\b_1+\b_2)(p-2)-1}\\
    \sigma^{2p-6}&\overset{?}{\le}1\qquad(true).
    \end{align*} 
\end{enumerate}

\noindent\fbox{Case 4: $6\le p$. } Again using \eqref{fromhereC} with $q=p$, it suffices to verify that
\[(\sigma R^{\frac{1}{2}})^{p-4}(\sigma^{-1}R^{\b_1+\b_2-\frac{1}{2}}\max(1,R^{\b_1}\sigma R^{-\frac{1}{2}},R^{\b_2}\sigma^2,\sigma^3 R^{\b_1+\b_2-\frac{1}{2}}))^{\frac{p}{2}-1}\le R^{(\b_1+\b_2)(p-2)-1}, \]
which we do in the following four cases. 
\begin{enumerate}[label=(\alph*)]
    \item $1=\max(1,R^{\b_1}\sigma R^{-\frac{1}{2}},R^{\b_2}\sigma^2,\sigma^3 R^{\b_1+\b_2-\frac{1}{2}})$: 
    \begin{align*} 
    (\sigma R^{\frac{1}{2}})^{p-4}(\sigma^{-1}R^{\b_1+\b_2-\frac{1}{2}})^{\frac{p}{2}-1}&\overset{?}{\le}R^{(\b_1+\b_2)(p-2)-1} \\
    (\sigma R^{\frac{1}{2}})^{\frac{p}{2}-3}&\overset{?}{\le}R^{(\b_1+\b_2)(\frac{p}{2}-1)-1} \\
    \sigma^{\frac{p}{2}-3}&{\le}1\le  R^{(\b_1+\b_2-\frac{1}{2})(\frac{p}{2}-1)} \qquad(true) .
    \end{align*}  
\end{enumerate}
Cases (b)-(d) are completely analogous to cases (b)-(d) from Case 3.

\end{proof}

\section{Appendix: General cones \label{app}}

We describe the adjustments needed to prove Theorem \ref{mainC} for general cones. To define a general cone, first let $\g(t):\R^3\to S^2$ be a nondegenerate $C^2$ curve, meaning that
\[ \det\big[\g(t)\,\dot{\g}(t)\,\ddot{\g}(t)\big]\not=0\qquad\forall t\in[-\frac{1}{2},\frac{1}{2}). \]
The corresponding general cone is
\[ \Gamma=\Gamma^\g=\{r\g(t):\,\,\frac{1}{2}\le r\le 2,\quad -\frac{1}{2}\le t<\frac{1}{2} \}. \]
For a parameter $R\in 4^\N$, define canonical caps $\theta$ of $\Gamma^\g$ to be approximately the $R^{-1}$-neighborhood of sectors 
\[ \{r\g(t):\frac{1}{2}\le r\le 2,\quad lR^{-1/2}\le t<(l+1)R^{-1/2}\}\]
where $l=-\frac{1}{2}R^{1/2}+1,\ldots,\frac{1}{2}R^{1/2}-1$. One way to specify $\theta$ is to fix a unit vector $v\in\R^3\setminus\{\g(t):0\le t\le 1\}$ and use the planes $\text{Span}\{\g(lR^{-1/2}),v\}$ and $\text{Span}\{\g((l+1)R^{-1/2}),v\}$ to define the edges of $\theta$ inside of $\mc{N}_{R^{-1}}(\Gamma)$. The collection of canonical caps $\theta$ is denoted ${\bf{S}}_{R^{-1/2}}$.

Instead of working with all of $\g$, we will cut $\g$ up into curves of length $2a$, where $a$ is chosen to be sufficiently small (depending on $\g$) below. For an interval $I_a\subset[-\frac{1}{2},\frac{1}{2})$ containing $0$, it suffices to consider $\g:I_a\to S^2$ with the extra assumptions that
\begin{align}\label{cond0} \g(0)&=\frac{1}{\sqrt{2}}(0,1,1),\quad\dot{\g}(0)=(1,0,0),\quad \frac{1}{2}\le \g_3(t)\le 1\quad\text{and}\quad |\g(t)|=|\dot{\g}(t)|=1\quad\text{for all}\quad t\in I_a  \end{align}
where $\g_3(t)$ is the third coordinate of $\g$. Define $b=[\dot{\g}(0)\times\g(0)]\cdot\ddot\g(0)$, which, without loss of generality, we assume is positive. To repeat the Kakeya step in the proof, it will be useful to view $\Gamma^\g$ as a subset of the cone 
\begin{equation}\label{def2} \{r(\frac{\g_1}{\g_3}(t),\frac{\g_2}{\g_3}(t),1):t\in I_a,\qquad r\ge 0\}. \end{equation}
Write $\tilde{\g}_1=\frac{\g_1}{\g_3}$ and $\tilde{\g}_2=\frac{\g_2}{\g_3}$. Using the inverse function theorem, we may 
view $\tilde{\g}_2$ as a $C^2$ function of $\tilde{\g}_1$ which satisfies 
\begin{equation}\label{cond4} \tilde{\g}_2(0)=1,\quad \frac{d}{d\tilde{\g}_1}\tilde{\g}_2(0)=0,\quad \text{and}\quad -4b\le \frac{d^2}{d\tilde{\g}_1^2}\tilde{\g}_2(\tilde{\g}_1)\le -\frac{b}{4}\qquad\text{for all}\qquad \tilde{\g}_1\in I \end{equation}
where $\tilde{\g}_1\in I$ is in bijection with $t\in I_a$. Here we used the fact that $\frac{d^2}{d\tilde{\g}_2^2}\tilde{\g}_1(0)=-b$. In \textsection\ref{genlorsec}, we will choose $a>0$ small enough so that \eqref{cond4} continues to hold after a general Lorentz rescaling. 

For the base case and rescaling steps, it is useful to work with a rotated coordinate system which we describe now. Define the orthonormal frame 
\[{\bf{c}}(t)=\g(t),\quad {\bf{n}}(t)=\dot{\g}(t)\times{\g}(t),\quad\text{ and }\quad{\bf{t}}(t)=\dot{\g}(t). \]
The coordinates $(\nu_1,\nu_2,\nu_3)$ are defined in terms of $\xi\in\R^3$ by $\nu_1={\bf{n}}(0)\cdot\xi$, $\nu_2={\bf{t}}(0)\cdot\xi$, and $\nu_3={\bf{c}}(0)\cdot\xi$. These are consistent with the coordinates defined in \textsection\ref{trunc}. Recall that $K$ is a constant permitted to depend on $\e$. We will assume that $a$ is small enough so that 
\begin{equation}\label{cond1} 1-\frac{1}{K}\le {\bf{c}}(0)\cdot r\g(t)\le 1\end{equation}
for all $\sqrt{1-K^{-1}}\le r\le 1$ and $t\in I_a$. This means that $\nu_3$ varies in the range $1-\frac{1}{K}\le \nu_3\le 1$. Next, consider the relation 
\[ \frac{\nu_2}{\nu_3}=\frac{{\bf{t}}(0)\cdot\g(t)}{{\bf{c}}(0)\cdot\g(t)}. \]
Calculate that $\frac{\nu_2}{\nu_3}=0$ when $t=0$ and $\left.\frac{d}{dt}\right|_{t=0}\frac{\nu_2}{\nu_3}=1$. By the inverse function theorem, for sufficiently small $a>0$, there is an interval $I_a$ of length $\sim a$ containing $0$ and an injective $C^2$ function $g:(-a,a)\to I_a$ satisfying $g(\frac{\nu_2}{\nu_3})=t$, $g(0)=0$, $g'(0)=1$, and $\frac{1}{2}\le g'(s)\le 2$. The function $g$ allows us to write 
\[ \frac{\nu_1}{\nu_3}=\frac{{\bf{n}}(0)\cdot\g(t)}{{\bf{c}}(0)\cdot\g(t)} \]
as a function of $\frac{\nu_2}{\nu_3}$. 
We work with a truncated version $\Gamma_{\frac{1}{K}}^\g$ defined by 
\begin{equation}\label{gentrunc} \Gamma_{\frac{1}{K}}^\g:=\Big\{(\nu_1,\nu_2,\nu_3): 1-\frac{1}{K}\le \nu_3\le 1, \quad\big|\frac{\nu_2}{\nu_3}\big|\le a,\quad \frac{\nu_1}{\nu_3}=\frac{{\bf{t}}(0)\cdot\g(g(\frac{\nu_2}{\nu_3}))}{{\bf{c}}(0)\cdot\g(g(\frac{\nu_2}{\nu_3}))}\Big\}. \end{equation}
Use the notation $\tilde{\nu_1}=\frac{\nu_1}{\nu_3}$ and $\tilde{\nu_2}=\frac{\nu_2}{\nu_3}$. Regarding $\tilde{\nu_1}$ as a function of $\tilde{\nu_2}$, we calculate 
\begin{equation}\label{cond2} \tilde{\nu_1}(0)=0\qquad\text{and}\qquad \frac{d}{d\tilde{\nu_2}}\tilde{\nu_1}(0)=0.  \end{equation}
Also calculate that $\frac{d^2}{d\tilde{\nu_2}^2}\tilde{\nu_1}(0)={\bf{n}}(0)\cdot\ddot{\g}(0)=b$. Take $a$  (and therefore $I_a$) small enough so that 
\begin{equation}\label{cond3} \frac{b}{2}\le \frac{d^2}{d\tilde{\nu_2}^2}\tilde{\nu_1}(\tilde{\nu_2})\le 2b \qquad\text{for all }\tilde{\nu_2}\in [-a,a]. \end{equation}

Let $\mc{F}(a,b)$ denote the set of nondegenerate $C^2$ functions $\g:I_a\to\R^3$ satisfying \eqref{cond0} and \eqref{cond1} and whose associated functions $\tilde{\nu_1}(\tilde{\nu_2})$ and $\tilde{\g}_2(\tilde{\g}_1)$ satisfy \eqref{cond2}, \eqref{cond3}, and \eqref{cond4}. 
We will define an analogue of $S_K(r,R)$ simultaneously for all general cones $\Gamma_{\frac{1}{K}}^\g$ generated by $\g\in \mc{F}(a,b)$.

For each dyadic $\sigma\in[R^{-1/2},1]$ and each canonical cap $\tau=\tau(l\sigma^{-1}R^{-1/2})$ of dimensions $1\times \sigma^{-1}R^{-1/2}\times \sigma^{-2}R^{-1}$, define the general wave envelopes $U_{\tau,R}$ to be 
\begin{align}
\label{genwe} U_{\tau,R}:= \{x\in\R^3:|x\cdot{\bf{c}}(l\sigma^{-1}R^{-1/2})|\le \sigma^{-2}\quad\text{and}&\quad|x\cdot{\bf{n}}(l\sigma^{-1}R^{-1/2})|\le R \\
&\qquad\text{and}\quad |x\cdot{\bf{t}}(l\sigma^{-1}R^{-1/2})|\le \sigma^{-1}R^{1/2}\}. \nonumber
\end{align}
As above, we write $U\|U_{\tau,R}$ to index a tiling of $\R^3$ by translates of $U_{\tau,R}$. Now we define a general $S_K^g(r,R)=S_K^g(r,R,a,b)$ exactly as in \eqref{constdef}, but allowing any Schwartz function $f:\R^3\to\C$ with Fourier support in $\mc{N}_{R^{-1}}(\Gamma_{\frac{1}{K}}^\g)$ for some $\g\in \mc{F}(a,b)$. In the following sections, we will describe the adjustments needed to prove Lemmas \ref{kaklem}, \ref{basecase}, and \ref{rescaling} for the general cone set-up. 

\subsection{Adaptation of the Kakeya step (Lemma \ref{kaklem})}

We sketch the changes needed in \textsection\ref{kak}. The parameters $\sigma_i$ continue to vary over dyadic values in the range $\sigma_1=r^{-1}\le \sigma_n\le \sigma_N=1$. For each $\sigma_n$, and each $\tau_n\in{\bf{S}}_{\sigma_n^{-1}r^{-1}}^\g$, let $U_{\tau_n,r^2}$ be a dual envelope of dimensions $ \sigma_n^{-2}\times\sigma_n^{-1}r\times r^2$ defined in \eqref{genwe}. Let $\sigma_{N_0}=(r_1/r)^{-1}$. 

In the pruning process, the definition of $\mc{G}_{\tau_n}$ becomes
\[ \mc{G}_{\tau_n}:=\{U\|U_{\tau_n,r^2}:C_b^4(\log( r^2/r_1))^{5}|U|^{-1}\int\sum_{\theta\subset\tau_n}|f_{\sigma_{n+1},\theta}|^2W_U \ge \frac{\a^2}{\#\tau_n}\}
\]
where $C_b$ depends on the parameter $b$ from the definition of $\mc{F}(a,b)$. The constant $C_b\ge1$ is chosen to be large enough in the proofs of Lemma \ref{weakhi} and the new Lemma \ref{our4.1} below, and $C_b^4$ replaces $C_0$ in the statement of Lemma \ref{weakhi}. Lemmas \ref{conepruneprops} and \ref{prune1} are unchanged. In Definition \ref{freqcutoff}, $\eta_{\Theta_{\tau_n}}=1$ on $4C_b\underline{C}(\log (r^2/r_1))\Theta_{\tau_n}$ and is supported in $8C_b\underline{C}(\log(r^2/r_1))\Theta_{\tau_n}$. For Lemmas \ref{middom} and \ref{lowdom}, we invoke Lemmas 4.1 and 4.2 of \cite{locsmooth}. Now we write versions of Lemmas 4.1 and 4.2 from \cite{locsmooth} which are adapted our general cone set-up. 

We need some more notation. For each $\theta\in{\bf{S}}_{r^{-1}}$, write $\tilde{\theta}$ for $\theta-\theta$. If $\theta=\theta(lr^{-1})$, then $\tilde{\theta}$ is contained in (and comparable to) the set 
\begin{equation*}
    \label{thetadef2} \{\w:|{\bf{c}}(lr^{-1})\cdot\w|\le1,\quad |{\bf{n}}(lr^{-1})\cdot\w|\le 2r^{-2},\quad |{\bf{t}}(lr^{-1})\cdot\w|\le 2r^{-1}\} .
\end{equation*}    
For each $\sigma_n$, ${\bf{CP}}_{\sigma_n}$ is a union of $\sigma_n^2\times 2r^{-1}\sigma_n\times 2r^{-2}$ planks $\Theta$ defined by 
\begin{equation*}\label{Thetadef2} \Theta=\Theta(\sigma_n,lr^{-1})=\{\w:|{\bf{c}}(lr^{-1})\cdot\w|\le \sigma_n^2,\quad |{\bf{n}}(lr^{-1})\cdot\w|\le 2r^{-2},\quad |{\bf{t}}(lr^{-1})\cdot\w|\le 2r^{-1}\sigma_n\} \end{equation*} 
where $l$ varies over $[-ra,ra]\cap \Z$. The $\Theta\in{\bf{CP}}_{\sigma_n}$ may be sorted into $\sigma_n r$ many essentially distinct sets. If $\Theta\in{\bf{CP}}_{\sigma_n}$ has parameters $\Theta=\Theta(\sigma_n,l\sigma_n^{-1}r^{-1})$, then $\Theta$ is naturally associated to  $\tau_n\in{\bf{S}}_{\sigma_n^{-1}r^{-1}}$ with $\tau_n=\tau_n(l\sigma_n^{-1}r^{-1})$. Also associate $\theta\in{\bf{S}}_{r^{-1}}$ with $\Theta_{\tau_n}\in{\bf{CP}}_{\sigma_n}$ where $\tau_n\supset\theta$. Let $\Omega_{\le \sigma_n}=\cup_{\Theta\in{\bf{CP}}_{\sigma_n}}\Theta$ and $\Omega_{\sigma_n}=\Omega_{\le \sigma_n}\setminus\Omega_{\le \sigma_n/2}$. Decompose $\Omega:=\cup_{\theta\in{\bf{S}}_{r^{-1}}}\tilde{\theta}$ into sets $\Omega_{\le r^{-1}}\cup\big(\cup_{r^{-1}<\sigma_n\le 1}\Omega_{\sigma_n}\big)$.

\begin{lemma}\label{our4.1} For each $\xi\in \underline{C}(\log (r^2/r_1))\Omega_{\sigma_k}$, 
\[ \sum_{\theta\in{\bf{S}}_{r^{-1}}}\widehat{|f_{\sigma_n,\theta}^{\mc{B}}|^2}(\xi)= \sum_{\tau_k\in{\bf{S}}_{\sigma_k^{-1}r^{-1}}}\sum_{\substack{\theta\in{\bf{S}}_{r^{-1}}\\  \theta\subset\tau_k}} \widehat{|f_{\sigma_n,\theta}^{\mc{B}}|^2}(\xi)\eta_{\Theta_{\tau_k}}(\xi).\]
\end{lemma}
\begin{proof} We follow the argument from Lemma 4.1 in \cite{locsmooth}. Write $\xi=\underline{C}(\log (r^2/r_1))\xi'$. Note that $\xi'\in\Omega_{\sigma_k}$,  $\xi\in\supp\widehat{|f_{\sigma_n,\theta}^{\mc{B}}|^2}$ implies that $\xi'\in\theta-\theta$, and $\xi\in\supp\eta_{\Theta_{\tau_k}}$ implies that $\xi'\in8C_b\Theta_{\tau_k}$. It suffices to show that if $\xi'\in\tilde{\theta}\cap\Omega_{\le \sigma_k}$, then $\xi'\in C_b\Theta_{\tau_k}$, where $\tau_k\supset\theta$. Begin with understanding the set $\Omega_{\le \sigma_k}$ by intersecting with the plane $\{\xi_3=h\}$. This set is empty if $h>\sigma_k^2$, so assume that $h\le \sigma_k^2$. The set of $\Theta\in{\bf{CP}}_{\sigma_k}$ are planks that are tangent to the cone
\[\{r(\tilde{\g}_1,\tilde{\g}_2(\tilde{\g}_1),1):\tilde{\g}_1\in I,\qquad r\ge 0\}  \]
where we use the description \eqref{def2}. The intersection of that cone with the plane $\{\xi_3=h\}$ is the curve 
\begin{equation}\label{curve} \{h(\tilde{\g}_1,\tilde{\g}_2(\tilde{\g}_1),1):\tilde{\g}_1\in I\}. \end{equation} 
If $\Theta\in{\bf{CP}}_{\sigma_k}$ has associated parameters $\Theta=\Theta(\sigma_k,l\sigma_k^{-1}r^{-1})$, then $\Theta\cap\{\xi_3=h\}$ is a rectangle of dimensions $\sim \sigma_k r^{-1}\times r^{-2}$ that is tangent to the curve \eqref{curve} at the point $h(\tilde{\g}_1(l\sigma_k^{-1}r^{-1}),\tilde{\g}_2(\tilde{\g}_1(l\sigma_k^{-1}r^{-1})),1)$ (recalling that we may view $\tilde{\g}_1$ as a function of $t$). If we allow $l$ to vary over $\sigma_k rI_a\cap\Z$ and recall that $\tilde{\g}_2''\sim -b$, then we obtain the following neighborhood of \eqref{curve}:
\begin{equation}\label{nbhd}
\{h(\tilde{\g}_1,\tilde{\g}_2(\tilde{\g}_1),1)+\lambda(-\tilde{\g}_2'(\tilde{\g}_1),1,0):\tilde{\g}_1\in I,\qquad0\le \lambda\lesssim \min(h^{-1}\sigma_k^2r^{-2},\sigma_k r^{-1})\}. 
\end{equation}

With this description of $\Omega_{\le\sigma_k}$ in hand, consider the intersection $\tilde{\theta}\cap\Omega_{\le\sigma_k}\cap\{\xi_3=h\}$. First note that ${\theta}\cap\{\xi_3=h\}$ is a rectangle with dimensions $\sim r^{-1}\times r^{-2}$ and tangent to the curve \eqref{curve} at the point $h(\tilde{\g}_1(lr^{-1}),\tilde{\g}_2(\tilde{\g}_1(lr^{-1})),1)$. The intersection of $\tilde{\theta}$ with the set \eqref{nbhd} is a shorter rectangle with dimensions $\sim \tilde{C}_b \sigma_k r^{-1}\times r^{-2}$, recalling that $|\tilde{\g}_2''|\sim b$. If we take $l$ such that $\theta=\theta(lr^{-1})$ and $l'$ so that $|lr^{-1}-l'\sigma_k^{-1}r^{-1}|<\sigma_k^{-1}r^{-1}$, then we may choose $C_b$ large enough so that the rectangle $\tilde{\theta}\cap\Omega_{\le\sigma_k}\cap\{\xi_3-h\}$ is contained in $\Theta=\Theta(\sigma_k,l')=\Theta_{\tau_k}$ dilated by a factor of $C_b$, which finishes the proof. 

\end{proof}

\begin{lemma}\label{our4.2} For each $1\le n\le k\le N_0$, $\xi\in \underline{C}(\log (r^2/r_1))\Omega_{\sigma_k}$, there are $\lesssim_b 1$ many $\tau_k\in{\bf{S}}_{\sigma_k^{-1}r^{-1}}$ so that 
\[\sum_{\substack{\theta\in{\bf{S}}_{r^{-1}}\\  \theta\subset\tau_k}} \widehat{|f_{\sigma_n,\theta}^{\mc{B}}|^2}(\xi)\eta_{\Theta_{\tau_k}}(\xi)\]
is nonzero. 
\end{lemma}
\begin{proof} We follow the argument from Lemma 4.2 in \cite{locsmooth}. Also build on the description of $\Omega_{\sigma_k}$ from the proof of Lemma \ref{our4.1}. It suffices to show that for each $\xi'\in\Omega_{\sigma_k}\cap\{\xi_3=h\}$ with $h\le \sigma_k^2$, there are $\lesssim_b 1$ many $\tau_k\in{\bf{S}}_{\sigma_k^{-1}r^{-1}}$ so that $\xi'\in \Theta_{\tau_k}^b\cap\{\xi_3=h\}$ where $\Theta_{\tau_k}^b$ is anisotropically dilated so that it is a $\sim \sigma_k^2\times C_b\sigma_kr^{-1}\times r^{-2}$ plank (which contains $[\underline{C}(\log (r^2/r_1))]^{-1}\supp\widehat{|f_{\sigma_n,\theta}^{\mc{B}}|^2}\eta_{\Theta_{\tau_k}}$ for each $\theta\subset\tau_k$). First consider the case that $\sigma_k^2/4\le h\le \sigma_k^2$. Then $\Theta_{\tau_k}^b\cap\{\xi_3=h\}$ is a $\sim C_b\sigma_k r^{-1}\times r^{-2}$ rectangle which is tangent to \eqref{curve}. If  $\Theta_{\tau_k}^b\cap\Theta_{\tau_k'}^b\cap\{\xi_3=h\}$ is nonempty, then for some $|\lambda|,|\lambda'|\lesssim C_bh^{-1}\sigma_kr^{-1}$ and for $\tilde{\g}_1(l)=\tilde{\g}(t=l\sigma_k^{-1}r^{-1})$, $\tilde{\g}_1(l')=\tilde{\g}_1(t=l'\sigma_k^{-1}r^{-1})$, 
\begin{align}\label{rel}
    |(\tilde{\g}_1(l),\tilde{\g}_2(\tilde{\g}_1(l)))+\lambda(1,\tilde{\g}_2'(\tilde{\g}_1(l)))-(\tilde{\g}_1(l'),\tilde{\g}_2(\tilde{\g}_1(l')))-\lambda'(1,\tilde{\g}_2'(\tilde{\g}_1(l')))|\lesssim h^{-1}r^{-2}.
\end{align}
Using the sizes of $\lambda$ and $\lambda'$, and the assumption that $h^{-1}\sim\sigma_k^{-2}$, by considering the first component above, we have
\[ |\tilde{\g}_1(l)-\tilde{\g}_1(l')|\lesssim C_b\sigma_k^{-1}r^{-1}. \]
By the mean value theorem, there exists $x\in(l'\sigma_k^{-1}r^{-1},l\sigma_k^{-1}r^{-1})$ satisfying $\tilde{\g}_1(l)-\tilde{\g}_1(l')=\tilde{\g}_1'(x)(l-l')\sigma_k^{-1}r^{-1}$. Since $|\tilde{\g}_1'|\sim 1$, we conclude that $|l-l'|\lesssim_b 1$, as desired. 

The remaining case is that $h\le \sigma_k^2/4$. Then the the intersection of $\Theta_{\tau_k}^b\cap\{\xi_3=h\}\cap\Omega_{\sigma_k}$ is the union of two $\sim C_b\sigma_k r^{-1}\times r^{-2}$-rectangles which are $\frac{1}{2}\sigma_k r^{-1}$-separated. If $\Theta_{\tau_k}^b\cap\Theta_{\tau_k'}^b\cap\{\xi_3=h\}\cap\Omega_{\sigma_k}$ is nonempty, then \eqref{rel} holds with $\frac{1}{2}h^{-1}\sigma_k r^{-1}\le|\lambda|,|\lambda'|\le C_b h^{-1}\sigma_k r^{-1}$. Considering the two components in \eqref{rel} separately, we have
\begin{align}
    |\tilde{\g}_1(l)+\lambda-\tilde{\g}_1(l')-\lambda'|&\lesssim h^{-1} r^{-2}\label{A0}
    \\
    |\tilde{\g}_2(\tilde{\g}_1(l))-\tilde{\g}_2(\tilde{\g}_1(l'))+\lambda[\tilde{\g}_2'(\tilde{\g}_1(l))-\tilde{\g}_2'(\tilde{\g}_1(l'))]-(\tilde{\g}_1(l)-\tilde{\g}_1(l'))\tilde{\g}_2'(\tilde{\g}_1(l'))|&\lesssim h^{-1} r^{-2} \label{A}. 
\end{align}
Recall that $\g_1(I_a)=I$, which we now label $I=[-a_1,a_2]$. Define $H:[-a_1,a_2]\to\R$ by 
\[ H(s)= \tilde{\g}_2(\tilde{\g}_1(l))-\tilde{\g}_2(s)+\lambda[\tilde{\g}_2'(\tilde{\g}_1(l))-\tilde{\g}_2'(s)]-(\tilde{\g}_1(l)-s)\tilde{\g}_2'(s) . \] 
Note that $H(\tilde{\g}_1(l))=0$. The first derivative is
\[ H'(s)=
-\tilde{\g}_2''(s)(\lambda+\tilde{\g}_1(l)-s). \]
Recall that $\frac{b}{2}\le -\tilde{\g}_2''(s)\le 2b$ and $b>0$. For $s\in[-a_1,a_2]$, assume that 
\begin{equation}\label{asmpt}-Ch^{-1}r^{-2}\le H(s)\le Ch^{-1}r^{-2}.  \end{equation}
In the following three cases, we will show that \eqref{asmpt} implies that $s$ is in either one or two intervals with centers determined by $l$ and $\lambda$ and length $\sim b^{-1}h^{-1}r^{-2}/\lambda$. Then since $\lambda\gtrsim h^{-1}\sigma_k r^{-1}$, the length of these intervals is $\lesssim b^{-1}\sigma_k^{-1}r^{-1}$. The points $\tilde{\g}_1(l')$, varying over $l'\in\sigma_k rI_a\cap\Z$, form a $\sim \sigma_k^{-1}r^{-1}$-separated set, so we conclude that there are $\lesssim b^{-1}$ many $l'$ satisfying \eqref{A}. It remains to analyze \eqref{asmpt}. 

\noindent \underline{Case 1}: $\lambda+\tilde{\g}_1(l)<-a_1$. In this case, $\lambda<0$ and $H'(s)$ is always negative. The inequality \eqref{asmpt} implies that $s$ is in a neighborhood of $\tilde{\g}_1(l)$. In particular, for $\d>0$ and $\tilde{\g}_1(l)+\d\le a_2$, we calculate that
\[ H(\tilde{\g}_1(l)+\d)=\int_{\tilde{\g}_1(l)}^{\tilde{\g}_1(l)+\d}-\tilde{\g}_2''(t)(\lambda+\tilde{\g}_1(l)-t)dt \le -\frac{b}{2}\int_{\tilde{\g}_1(l)}^{\tilde{\g}_1(l)+\d}(t-\lambda-\tilde{\g}_1(l))dt \le \frac{b}{4}\d(2\lambda-\d)  .\]
Combining the above line with $\d>0$ and $-Ch^{-1}r^{-2}\le H(\tilde{\g}_1(l)+\d)$, conclude that $0\le \d \lesssim \frac{b^{-1}h^{-1}r^{-2}}{\lambda}$. A similar calculation show that if $H(\tilde{\g}_1(l)-\d)\le Ch^{-1}r^{-1}$, then $-\frac{b^{-1}h^{-1}r^{-2}}{\lambda}\lesssim -\d\le 0$. Conclude in this case that $|s-\tilde{\g}_1(l)|\lesssim \frac{b^{-1}h^{-1}r^{-2}}{\lambda}$.

\noindent \underline{Case 2}: $-a_1\le \lambda+\tilde{\g}_1(l)\le a_2$. Then $H'(s)\ge 0$ if $s\le \lambda+\tilde{\g}_1(l)$ and $H'(s)\le 0$ if $s\ge \lambda+\tilde{\g}_1(l)$. The maximum is $H(\lambda+\tilde{\g}_1(l))=\tilde{\g}_2(\tilde{\g}_1(l))-\tilde{\g}_2(\lambda+\tilde{\g}_1(l))+\lambda\tilde{\g}_2'(\tilde{\g}_1(l))$. By Taylor's theorem, there is some $|z_0|\le |\lambda|$ so that $H(\lambda+\tilde{\g}_1(l))=-\frac{1}{2}\tilde{\g}_2''(z_0+\tilde{\g}_1(l))\lambda^2\ge \frac{b}{4}\lambda^2$. We may assume that $\frac{b}{4}\lambda^2\ge Ch^{-1}r^{-2}$ since if not, then we would have $h\sim_b\sigma_k^2$. Repeating the argument from the case $h\ge \frac{\sigma_k^2}{4}$ finishes that case. Assuming now that $h(\lambda+\tilde{\g}_1(l))>Ch^{-1}r^{-2}$, the solutions to \eqref{asmpt} occur in at most two intervals. If $h'(s)> 0$, then by an argument similar to Case 1, $|s-\tilde{\g}_1(l)|\lesssim \frac{b^{-1}h^{-1}r^{-2}}{\lambda}$. If $h'(s)<0$, then $s$ is contained in a $\lesssim \frac{b^{-1}h^{-1}r^{-2}}{\lambda}$-neighborhood of either $a_2$ or $s_0$ satisfying $h(s_0)=0$ and $\lambda+\tilde{\g}_1(l)<s_0<a_2$. 

\noindent \underline{Case 3}: $a_2<\lambda+\tilde{\g}_1(l)$. Then $h'(s)$ is always positive. This case is analogous the Case 1 and we conclude that $|s-\tilde{\g}_1(l)|\lesssim \frac{b^{-1}h^{-1}r^{-2}}{\lambda}$.

\end{proof}

\subsection{General Lorentz rescaling \label{genlorsec}}

In the following two subsections (\textsection\ref{basecaseappsec} and \ref{rescalingappsec}), we will use the description \eqref{gentrunc} of $\Gamma^\g_{\frac{1}{K}}$. Describing the tools in our argument in the $(\nu_1,\nu_2,\nu_3)$-coordinates is analogous to \textsection{5} of \cite{locsmooth}. For each dyadic $s\in(0,a]$, define $\tau\in{\bf{S}}_s$ in the $(\nu_1,\nu_2,\nu_3)$-coordinates to be the $s^2$-neighborhood of 
\begin{equation}\label{nutau} \{(\nu_1,\nu_2,\nu_3):1-\frac{1}{K}\le \nu_3\le 1,\quad
\big|\frac{\nu_2}{\nu_3}-ls\big|\le\frac{s}{2},\quad \frac{\nu_1}{\nu_3}=\frac{{\bf{t}}(0)\cdot\g(g(\frac{\nu_2}{\nu_3}))}{{\bf{c}}(0)\cdot\g(g(\frac{\nu_2}{\nu_3}))}\} \end{equation}
where $l\in [-s^{-1}a,s^{-1}a]\cap\Z$. These $\tau$ are comparable to the canonical caps defined at the beginning of \eqref{app}. Fix a $\tau\in{\bf{S}}_s$ with $\tau=\tau(l)$. We define the Lorentz rescaling $\mc{L}:\R^3\to\R^3$ to be the linear map
\begin{equation}\label{genLdef} \begin{cases}
x_1\mapsto  \frac{4a^2}{s^2}(x_1-\tilde{\nu_1}'(ls)x_2-\tilde{\nu_1}(ls)x_3+\tilde{\nu_1}'(ls)lsx_3)\\ 
x_2\mapsto \frac{2a}{s}(x_2-lsx_3)\\
x_3\mapsto x_3\end{cases} \end{equation}
where the notation $\tilde{\nu}_1=\frac{\nu_1}{\nu_3}$, $\tilde{\nu}_2=\frac{\nu_2}{\nu_3}$, and $\tilde{\nu}_1'$ means $\frac{d}{d\tilde{\nu}_2}\tilde{\nu}_1$, as in \eqref{cond2}. Applying $\mc{L}$ to \eqref{nutau} yields 
\begin{align}\label{Ltau} \big\{(\w_1,\w_2,\w_3):&1-\frac{1}{K}\le \w_3\le 1,\quad\big|\frac{\w_2}{\w_3}\big|\le a,\\
&\qquad\quad \frac{\w_1}{\w_3}=\frac{4a^2}{s^2}(\tilde{\nu}_1(\frac{s}{2a}\frac{\w_2}{\w_3}+ls)-\tilde{\nu}_1'(ls)[\frac{s}{2a}\frac{\w_2}{\w_3}+ls]-\tilde{\nu}_1(ls)+\tilde{\nu}_1'(ls)ls)\big\} .\nonumber \end{align}
Our goal is to regard this set as $\Gamma_{\frac{1}{K}}^{\underline{\g}}$ for some $\underline{\g}\in{\mc{F}}(a,b)$. Write $\tilde{\w}_2=\frac{\w_2}{\w_3}$ and $\tilde{\w}_1=\frac{\w_1}{\w_3}$. Viewing $\tilde{\w_1}$ as a function of $\tilde{\w}_2$, we have 
\[ \tilde{\w}_1'(\tilde{\w}_2)=\frac{2a}{s}(\tilde{\nu}_1'(\frac{s}{2a}\tilde{\w}_2+ls)-\tilde{\nu}_1'(ls))\quad\text{and}\quad \tilde{\w}_1''(\tilde{\w}_2)=\tilde{\nu}_2''(\frac{s}{2a}\tilde{\w}_2+ls). \]
We readily see that $\tilde{\w}_1(0)=0$, $\tilde{\w}_1'(0)=0$, and $\frac{\b}{2}\le \tilde{\w}_1''(\tilde{\w}_2)\le 2b$ for all $|\tilde{\w}_2|\le a$, which verifies the conditions \eqref{cond2} and \eqref{cond3}. To figure out which $\tilde{\g}$ the set \eqref{Ltau} corresponds to, we return to the $(\xi_1,\xi_2,\xi_3)$-coordinates (recalling the relationship between $\xi$ and $\nu$ from the beginning of \textsection\ref{app}), obtaining
\begin{align*}
\big\{(\w_2,\frac{\w_3-\w_1}{\sqrt{2}},\frac{\w_3+\w_1}{\sqrt{2}}):&1-\frac{1}{K}\le \w_3,\quad\big|\frac{\w_2}{\w_3}\big|\le a,\\
&\quad \tilde{\w}_1=\frac{4a^2}{s^2}(\tilde{\nu}_1(\frac{s}{2a}\tilde{\w}_2+ls)-\tilde{\nu}_1'(ls)[\frac{s}{2a}\tilde{\w}_2+ls]-\tilde{\nu}_1(ls)+\tilde{\nu}_1'(ls)ls)\big\} .    
\end{align*}
The curve $\underline{\g}$ is has components $\underline{\g}_i$ which satisfy $\frac{\underline{\g}_1}{\underline{\g}_3}=\underline{\tilde{\g}}_1=\frac{\sqrt{2}\tilde{\w}_2}{1+\tilde{\w}_1}$, and $\frac{\underline{\g}_2}{\underline{\g}_3}=\tilde{\underline{\g}}_2=\frac{1-\tilde{\w}_1}{1+\tilde{\w}_1}$. In order to verify that $\tilde{\g}\in\mc{F}(a,b)$, we need to check the properties \eqref{cond4} with $\underline{\tilde{\g}}_i$ in place of $\tilde{\g}$. Taking $\underline{\tilde{\g}}_1=0$ means that $\tilde{\w}_2=0$, so regarding $\underline{\tilde{\g}}_2$ as a function of $\underline{\tilde{\g}}_1$, $\underline{\tilde{\g}}_2(0)=1$. Next, calculate 
\[ \frac{d\underline{\tilde{\g}}_2}{d\underline{\tilde{\g}}_1}=-\frac{(2+\tilde{\w}_1)\tilde{\w}_1'}{\sqrt{2}(1+\tilde{\w}_1-\tilde{\w}_2\tilde{\w}_1')},\]
where $\tilde{\w}_1=\tilde{\w}_1(\tilde{\w}_2)$ and $\tilde{\w}_1'=\tilde{\w}_1'(\tilde{\w}_2)$. Note that $\underline{\tilde{\g}}_2'(0)=0$. Finally, we have
\[ \frac{d^2\underline{\tilde{\g}}_2}{d\underline{\tilde{\g}}_1^2}=\frac{-(1+\tilde{\w}_1)^2(\tilde{\w}_1')^2}{2(1+\tilde{\w}_1-\tilde{\w}_2\tilde{\w}_1')^2} -\frac{(1+\tilde{\w}_1)^3(2+\tilde{\w}_1)}{2(1+\tilde{\w}_1-\tilde{\w}_2\tilde{\w}_1')^3}\tilde{\w}_1'', \]
so $\underline{\tilde{\g}}_2''(0)=-\tilde{\nu}_2''(ls)$. 
As $\tilde{\w}_2$ varies over the interval $|\tilde{\w}_2|\le a$, using the mean value theorem and the fact that $b\le 1$, we have $|\tilde{\w}_1'(\tilde{\w}_2)|\le 2a$ and $|\tilde{\w}_1(\tilde{\w}_2)|\le 2a$. Since $\frac{b}{2}\le \tilde{\w}_1''\le 2b$, we may assume that $a>0$ is small enough so that the right hand side displayed above is in the range $[-4b,-b/4]$. This finishes the verification of \eqref{cond4}. 

To specify the curve $\underline{\g}$, write 
\[ F(\tilde{\w}_2)=\frac{(\underline{\tilde{\g}}_1(\tilde{\w}_2),\underline{\tilde{\g}}_2(\tilde{\w}_2),1)}{[(\underline{\tilde{\g}}_1)^2+(\underline{\tilde{\g}}_2)^2+1]^{1/2}}. \]
It is easy to verify that for $a>0$ sufficiently small, the third component satisfies $\frac{1}{2}\le F_3\le 2$. Also, $F(0)=\frac{1}{\sqrt{2}}(0,1,1)$ and $\frac{dF}{d\tilde{\w}_2}(0)=(1,0,0)$. Letting $t=\int_0^{\tilde{\w}_2}|\frac{dF}{d\tilde{\w}_2}|$, we define $\underline{\g}$ to be the unit-speed parametrization $F(\tilde{\w}_2(t))$, which then satisfies \eqref{cond0}. This finishes showing that the image of \eqref{Ltau} under $\mc{L}$ is $\Gamma_{\frac{1}{K}}^{\underline{\g}}$ for $\underline{\g}\in{\mc{F}}(a,b)$. 

Finally, we identify $\mc{L}(\tau')$, where $\tau'\in{\bf{S}}_{s'}$ and $\tau'\subset\tau$, with a plank in ${\bf{S}}_{2as's^{-1}}$ which is part of a partition of the $(2as's^{-1})^2$-neighborhood of $\Gamma_{\frac{1}{K}}^{\underline{\g}}$. The plank $\tau'$ is comparable to the convex hull of 
\[ \{(\nu_1,\nu_2,\nu_3):1-\frac{1}{K}\le \nu_3\le 1,\quad
\big|\frac{\nu_2}{\nu_3}-l's'\big|\le\frac{s'}{2},\quad \frac{\nu_1}{\nu_3}=\frac{{\bf{t}}(0)\cdot\g(g(\frac{\nu_2}{\nu_3}))}{{\bf{c}}(0)\cdot\g(g(\frac{\nu_2}{\nu_3}))}\}  \]
for some $l'\in[(-s')^{-1}a,(s')^{-1}a]\cap\Z$ satisfying $|l's'-ls|\le \frac{s}{2}$. Applying $\mc{L}$ to the above set yields 
\[ \{ (\w_3\tilde{\w}_1(\tilde{\w}_2),\w_2,\w_3):1-\frac{1}{K}\le \w_3,\quad\big|\frac{\w_2}{\w_3}-(ls(s')^{-1}-l')2as's^{-1}\big|\le as's^{-1}\}, \]
which has convex hull comparable to a plank in ${\bf{S}}_{2as's^{-1}}$. Since taking convex hulls commutes with linear transformations, we are done. 

For $r^{-1}<s$, it remains to note that $(\mc{L}^{-1})^*(U_{\tau,r^2})$ is comparable to $U_{\mc{L}(\tau),(2a)^{-2}s^2r^2}$. Since we are using the $(\nu_1,\nu_2,\nu_3)$-coordinate from now on, by $U_{\tau,r^2}$, we mean a rotated version of the set defined in \eqref{genwe}. As is done in \cite{locsmooth}, we use the characterization of $U_{\tau,r^2}$ as the convex hull of $\cup_{\theta\subset\tau}\tilde{\theta}^*$, where the union is taken over $\theta\in{\bf{S}}_{r^{-1}}$ contained in $\tau$. If for a subset $A\subset\R^3$, we define $A^*=\{x\in\R^3:|x\cdot y|\le 1\qquad\forall y\in A\}$, then it is easy to verify that $(\mc{L}^{-1})^*(\tilde{\theta}^*)=\mc{L}(\tilde{\theta})^*$. It follows that $(\mc{L}^{-1})^*(U_{\tau,r^2})\approx U_{\mc{L}(\tau),(2a)^{-2}s^2r^2}$.

With this background for the general Lorentz transformations established, the remainder of the argument is straightforward.

\subsubsection{Adaptation of the base-case step (Lemma \ref{basecase}) \label{basecaseappsec}}

Begin by noting that $\mc{N}_{\frac{1}{K}}(\Gamma_{\frac{1}{K}}^\g)$ is approximately the $\frac{1}{K}$-neighborhood of 
\[ \Big\{(\nu_1,\nu_2,\nu_3): 1-\frac{1}{K}\le \nu_3\le 1,\quad \Big|\frac{\nu_2}{\nu_3}\Big|\le a,\quad \nu_1=\frac{{\bf{t}}(0)\cdot\g(g(\nu_2))}{{\bf{c}}(0)\cdot\g(g(\nu_2))}  \Big\}  .\]
This is because in $\Gamma_{\frac{1}{K}}^\g$, $\nu_1=\nu_3\frac{{\bf{t}}(0)\cdot\g(g(\frac{\nu_2}{\nu_3}))}{{\bf{c}}(0)\cdot\g(g(\frac{\nu_2}{\nu_3}))}$ and by the mean value theorem, $|\frac{{\bf{t}}(0)\cdot\g(g(\frac{\nu_2}{\nu_3}))}{{\bf{c}}(0)\cdot\g(g(\frac{\nu_2}{\nu_3}))} -\frac{{\bf{t}}(0)\cdot\g(g(\nu_2))}{{\bf{c}}(0)\cdot\g(g(\nu_2))} |\le \frac{1}{K}$. 
The set displayed above is a cylinder over the curve $(\nu_1(\nu_2),\nu_2,0)$. The proof of Proposition \ref{basecaseprop} in \textsection\ref{trunc} goes through, except we must invoke the general version of Theorem \ref{mainP} proved in \textsection\ref{mainPgensec}.

\subsubsection{Adaptation of the rescaling step (Lemma \ref{rescaling})\label{rescalingappsec}}

The pigeonholing step Lemma \ref{pig} which regularizes the support of $\widehat{f}$ is unchanged. In the proof of Lemma \ref{rescaling}, we invoke the general Lorentz rescaling defined in \eqref{genLdef} rather than the map from \cite{locsmooth}.

\bibliographystyle{alpha}
\bibliography{AnalysisBibliography}

\end{document}